\RequirePackage[reqno]{amsmath}
\documentclass[a4paper]{amsart}
\usepackage[english]{babel}
\usepackage{amsfonts}
\usepackage{amssymb}
\usepackage{amsthm}
\usepackage{amsmath}
\usepackage{upgreek}
\usepackage{indentfirst}
\usepackage{empheq}
\usepackage{todo}
\usepackage[colorlinks=false,backref=true,pagebackref=false,pdftex,pdfauthor={Denis, Filippo, Ederson},pdftitle={Periodic solutions and torsional instability in a nonlinear nonlocal plate equation}]{hyperref}

\usepackage[utf8]{inputenc}
\usepackage{setspace}
\DeclareMathOperator{\cn}{cn}
\DeclareMathOperator{\sn}{sn}
\DeclareMathOperator{\dn}{dn}
\usepackage{xkvltxp}
\usepackage[margin=false,inline=true,index=true]{fixme}

\usepackage{graphicx}
\usepackage{color}
\usepackage{enumerate}

\newtheorem{theorem}{Theorem}
\newtheorem{lemma}[theorem]{Lemma}
\newtheorem{corollary}[theorem]{Corollary}
\newtheorem{proposition}[theorem]{Proposition}
\newtheorem{definition}[theorem]{Definition}
\newtheorem{remark}[theorem]{\bf Remark}

\usepackage{mathtools}
\mathtoolsset{showonlyrefs}
\usepackage{xcolor}
\def\neweq#1{\begin{equation}\label{#1}}
\def\endeq{\end{equation}}
\def\eq#1{(\ref{#1})}

\newcommand{\hs}{H^2_*(\Omega)}
\newcommand{\ho}{H^2_\mathcal{O}(\Omega)}
\newcommand{\he}{H^2_\mathcal{E}(\Omega)}

\newcommand\xit{(\xi,t)}

\newcommand\into{\int_\Omega}
\newcommand\R{\mathbb R}

\newcommand\dxit{(\xi,t)\, d\xi}

\begin{document}

\title[Nonlinear nonlocal plate equation]{Periodic solutions and torsional instability\\
in a nonlinear nonlocal plate equation}

\author[Denis BONHEURE -- Filippo GAZZOLA -- Ederson MOREIRA dos SANTOS]{Denis BONHEURE$^\flat$ -- Filippo GAZZOLA$^\sharp$ -- Ederson MOREIRA dos SANTOS$^\dag$\\
{\tiny $\flat$ D\'epartement de Math\'ematique - Universit\'e Libre de Bruxelles, Belgium}\\
{\tiny $\sharp$ Dipartimento di Matematica - Politecnico di Milano, Italy}\\
{\tiny $\dag$ Instituto de Ci\^encias Matem\'aticas e de Computa\c c\~ao - Universidade de S\~ao Paulo, Brazil}}

\date{\today}


\begin{abstract}
A thin and narrow rectangular plate having the two short edges hinged and the two long edges free is considered. A nonlinear nonlocal evolution
equation describing the deformation of the plate is introduced: well-posedness and existence of periodic solutions are proved. The natural phase space is a particular second order Sobolev space that can be orthogonally split into two subspaces
containing, respectively, the {\em longitudinal} and the {\em torsional} movements of the plate. Sufficient conditions for the stability
of periodic solutions and of solutions having only a longitudinal component are given. A stability analysis of the so-called {\em prevailing mode} is
also performed. Some numerical experiments show that instabilities may occur.
This plate can be seen as a simplified and qualitative model for the
deck of a suspension bridge, which does not take into account the complex interactions between all the components of a real bridge.
\vspace{8pt}
\begin{center}
To Cec\' \i lia
\end{center}
\end{abstract}
\maketitle

\tableofcontents

\section{Introduction}

\listoffixmes

We consider a thin and narrow rectangular plate with the two short edges hinged while the two long edges are free.
In absence of forces, the plate lies horizontally flat and is represented by the planar domain $\Omega=(0,\pi)\times(-\ell,\ell)$ with $0<\ell\ll\pi$.
The plate is subject both to dead and live loads acting orthogonally on $\Omega$ and to compressive forces along the edges, the so-called
buckling loads. We follow the plate model suggested by Berger \cite{berger}; see also the former beam model by Woinowsky-Krieger \cite{woinowsky}
and, independently, by Burgreen \cite{burg}. Then the nonlocal evolution equation modeling the deformation of the plate reads
\begin{empheq}{align}\label{enlm}
\left\{
\begin{array}{rl}
u_{tt}+\delta u_t + \Delta^2 u + \left[P - S \int_\Omega u_x^2\right] u_{xx}=g  &\textrm{in }\Omega\times(0,T)\\
u = u_{xx}= 0 &\textrm{on }\{0,\pi\}\times[-\ell,\ell]\\
u_{yy}+\sigma u_{xx} = u_{yyy}+(2-\sigma)u_{xxy}= 0 &\textrm{on }[0,\pi]\times\{-\ell,\ell\}\\
u(x,y, 0) = u_0(x,y), \quad \quad u_t(x,y, 0) = v_0(x,y) &\textrm{in }\Omega\, .
\end{array}
\right.
\end{empheq}

All the parameters in \eqref{enlm} and their physical meaning will be discussed in detail in Section \ref{quant}.
The plate $\Omega$ can be seen as a simplified model for the deck of a suspension bridge. Even if the model does not take into account the complex
interactions between all the components of a real bridge, we expect to observe the phenomena seen on built bridges. Therefore we will often refer to
the scenario described in the engineering literature and tackle the stability issue only qualitatively.\par
A crucial role in the collapse of several bridges is played by the mode of oscillation. In particular, as shown in the video \cite{tacoma},
the ``two waves'' were torsional oscillations and were considered the main cause of the Tacoma Narrows Bridge (TNB) collapse \cite{ammann,scott}.
The very same oscillations also caused several other bridges collapses: among others, we mention the
Brighton Chain Pier in 1836, the Menai Straits Bridge in 1839, the Wheeling Suspension Bridge in 1854, the Matukituki Suspension Footbridge in
1977; see \cite[Chapter 1]{2015gazzola} for a detailed description of these collapses. The distinguished civil and aeronautical engineer Robert
Scanlan \cite[p.209]{scanlan} attributed the appearance of torsional oscillations at the TNB to {\em some fortuitous condition}: the word ``fortuitous'' denotes a lack of rigorous explanations and, according to \cite{scott}, no fully satisfactory explanation has been reached in subsequent years.
In fact, no purely aerodynamic explanation was able to justify the origin of the torsional oscillation, which is the main culprit for the collapse of the TNB.
More recently \cite{arioligazzola,bfg1}, for slightly different bridge models, the attention was put on the nonlinear structural behavior of suspension bridges:
the bridge was considered isolated from aerodynamic effects and dissipation. The results therein show that the origin of torsional instability is structural
and explain why the TNB withstood larger longitudinal oscillations on low modes, but failed for smaller longitudinal oscillations on higher modes,
see \cite[pp.28-31]{ammann}. It is shown in \cite{arioligazzola,bfg1} that if the longitudinal oscillation is sufficiently large, then a structural
instability appears and this is the onset of torsional oscillations. In these papers, the main focus was on the structural behavior and the
action of the wind was missing.\par
From the Report \cite[pp.118-120]{ammann} we learn that for the recorded oscillations at the TNB {\em one definite mode of oscillation
prevailed over a certain interval of time. However, the modes frequently changed}. The suggested target was to find a possible {\em correlation
between the wind velocity and the prevailing mode} and the conclusion was that {\em a definite correlation exists between frequencies and wind
velocities: higher velocities favor modes with higher frequencies}. On the other hand, just a few days prior to the TNB collapse, the project engineer
L.R.\ Durkee wrote a letter (see \cite[p.28]{ammann}) describing the oscillations which were so far observed at the TNB. He wrote: {\em Altogether,
seven different motions have been definitely identified on the main span of the bridge, and likewise duplicated on the model. These different wave
actions consist of motions from the simplest, that of no nodes, to the most complex, that of seven nodes}. According to Eldridge \cite[V-3]{ammann},
a witness on the day of the TNB collapse, {\em the bridge appeared to be behaving in the customary manner} and the motions {\em were considerably less
than had occurred many times before}. Moreover, also Farquharson \cite[V-10]{ammann} witnessed the collapse and wrote that {\em the motions, which a moment
before had involved a number of waves (nine or ten) had shifted almost instantly to two}.\par
Aiming to explain and possibly reproduce these phenomena, we proceed as follows.
Firstly, we introduce in \eqref{enlm} the aerodynamic and dissipative effects, thereby completing the isolated models in \cite{arioligazzola,bfg1}. Then
we try to give answers to the questions left open by the above discussion:
\begin{enumerate}[(a)]
\item What is the correlation between the wind velocity and the prevailing  mode of oscillation?
\item How stable is the prevailing longitudinal mode with respect to torsional perturbations?
\end{enumerate}

To this end, the first step is to go through fine properties of the vibrating modes, both by estimating their frequencies (eigenvalues of a
suitable problem) and by classifying them into {\em longitudinal} and {\em torsional}. This is done in Section \ref{section:eigenvalue} where
we also decompose the phase space of \eqref{enlm} as direct sum of the orthogonal subspaces of {\em longitudinal functions} and of
{\em torsional functions}. Theorems \ref{exuniq} and \ref{th:peridoc} show that \eqref{enlm} is well-posed and that the
equation admits periodic solutions whenever the source $g$ is itself periodic. These solutions play an important role in our stability analysis
which is characterized in Definition \ref{defstab}: roughly speaking, we say that \eqref{enlm} is {\em torsionally stable} if the torsional part
of any solution tends to vanish at infinity and {\em torsionally unstable} otherwise, see also Proposition \ref{NSC}. In Theorem \ref{SCstability}
we establish that if the forcing term $g$ is sufficiently small, then \eqref{enlm} has a ``squeezing property'' as in
\cite[(2.7)]{ghisigobbinoharaux}, namely all its solutions have the
same behavior as $t\to\infty$. This enables us to prove both the uniqueness of a periodic solution (if $g$ is periodic) and to obtain a sufficient
condition for the torsional stability (if $g$ is even with respect to $y$). By exploiting an argument by Souplet \cite{souplet}, in Theorem
\ref{th:multiple-periodic-solutions} we show that this smallness condition is ``almost necessary'' since multiple periodic solutions may exist
in general. Theorem \ref{SCstability2} states a similar property, but more related to applications: we obtain torsional stability for any given
force $g$, provided that the damping coefficient $\delta$ is sufficiently large. Finally, in Theorem \ref{smalltorsion} we show that the responsible
for possible instabilities is the nonlinear nonlocal term $\|u_x\|_{L^2}^2u_{xx}$ which acts a coupling term and allows transfer of energy between
longitudinal and torsional oscillations.\par
Our results are complemented with some numerics aiming to describe the behavior of the solutions of \eqref{enlm} and to discuss the just
mentioned sufficient conditions for the torsional stability. Overall, the numerical results, combined with our theorems, allow to answer to
question (b): the stability of the prevailing longitudinal mode depends on its amplitude of oscillation, on its frequency of oscillation,
and on the torsional mode that perturbs the motion. In order to answer to question (a), in Section \ref{linearanal} we perform a
linear analysis. Our conclusion is that the prevailing mode is determined by the frequency of the forcing term $g$: there exist ranges of
frequencies for $g$, each one of them exciting a particular longitudinal mode which then plays the role of the prevailing mode.\par
This paper is organized as follows. In order to have physically meaningful results, in Section \ref{quant} we describe in detail the model and the physical
meaning of all the parameters in \eqref{enlm}. In Section \ref{section:eigenvalue}
we recall and improve some results about the spectrum of the linear elliptic operator in \eqref{enlm}. In Section \ref{main} we state our main
results. These results are complemented with the linear analysis of Section \ref{linearanal}, that enables us to answer to question (a),
and with the numerical experiments reported in Section \ref{numres}, that enable us to answer to question (b). Section \ref{enest} contains
some energy bounds, useful for the proofs of our results that are contained in the remaining sections, from \ref{proof1} to \ref{proof3}.

\section{The physical model}\label{quant}

In this section we perform space and time scalings that will reduce the dimensional equation
\begin{equation}\label{original}
M\, u_{tt}(\xi,t)+\varepsilon u_t(\xi,t)+D\, \Delta^2 u(\xi,t)+\left[P-\frac{AE}{2L}\int_\Omega u_x^2(z,t)dz\right]u_{xx}(\xi,t)=g(\xi,t)
\quad\mbox{in }\Omega\times(0,T)
\end{equation}
to the slightly simpler form \eqref{enlm}. Here and in the sequel, for simplicity we put
$$\xi:=(x,y)\in\Omega\, .$$

Let us explain the meaning of the structural constants appearing in \eqref{original}:\par\smallskip
$\Omega=(0,L)\times(-\ell,\ell)$ = the horizontal face of the rectangular plate\par
$L$ = length of the plate\par
$2\ell$ = width of the plate\par
$d$ = thickness of the plate\par
$H$ = frontal dimension (the height of the windward face of the plate)\par
$A=2\ell d$ = cross-sectional area of the plate\par
$\sigma$ = Poisson ratio of the material composing the plate\par
$D$ = flexural rigidity of the plate (the force couple required to bend it in one unit of curvature)\par
$M$ = surface density of mass of the plate\par
$P$ = prestressing constant (see \cite{burg})\par
$\varepsilon$ = damping coefficient\par\smallskip
If the plate is a perfect rectangular parallelepiped, that is, $(0,L)\times(-\ell,\ell)\times(0,d)$ with constant height $d$, then $H=d$. But in some cases,
such as for the collapsed TNB, the cross section of the plate is H-shaped: in these cases one has $H>d$. In many instances of
fluid-structure interaction the deck of a bridge is modeled as a Kirchhoff-Love plate and a 3D object is reduced to a 2D plate. Indeed, since the
thickness $d$ is constant, it may be considered as a rigidity parameter and one can focus the attention on the middle horizontal cross section $\Omega$
(the intersection of the parallelepiped with the plane $z=d/2$):
$$
\Omega=(0,L)\times(-\ell,\ell)\subset\mathbb{R}^2\, .
$$
This is physically justifiable as long as the vertical displacements remain in a certain range that usually covers the displacements of the deck.
The deflections of this plate are described by the function $u=u(x,y,t)$ with $(x,y)\in\Omega$.
The parameter $P$ is the buckling constant: one has $P>0$ if the plate is compressed and $P<0$ if the plate is stretched in the $x$-direction.
Indeed, for a partially hinged plate such as $\Omega$, the buckling load only acts in the $x$-direction and therefore one obtains the term $\int_\Omega u_x^2$
as for a one-dimensional beam; see \cite{knightly}. The Poisson ratio of metals lies around $0.3$ while for concrete it is between $0.1$ and $0.2$;
since the deck of a bridge is a mixture of metal and concrete we take
\begin{equation}\label{conditionsigma}
\sigma=0.2\, .
\end{equation}
The flexural rigidity $D$ is the resistance offered by the structure while bending, see e.g.\ \cite[Section 2.3]{ventsel}.
A reasonable value for the damping coefficient $\varepsilon>0$ has to be fixed. It is clear that large $\varepsilon$ make the solution of an equation converge
more quickly to its limit behavior and that smaller $\varepsilon$ may lead to solutions which have many oscillations around their limit behavior before
stabilizing close to it. Our choice of $\varepsilon$ is motivated by the following argument. Imagine that we focus on a time instant, that we shift to $t=0$,
where a certain mode is excited with a given amplitude and that, in this precise instant, the wind ceases to blow. The mode will tend asymptotically
(as $t\to\infty$) to
rest; although it will never reach the rest position, we aim at quantifying how much time is needed to reach an ``approximated rest position''. This means that
we estimate the time needed for the oscillations to become considerably smaller than the initial ones. A reasonable measure seems to be 100 times less,
that is, 1cm if the bridge was initially oscillating with an amplitude of 1m. De Miranda \cite{mdm} told us that a heavy oscillating structure like a bridge
is able to reduce the oscillation to 1\% of its initial amplitude in about 40 seconds. Since the oscillations tend to become small, we can linearize
and reduce to the prototype equation
$$M\ddot z(t)+\varepsilon\dot z(t)+\alpha z(t)=0\qquad(\mbox{with $\varepsilon<2\sqrt{\alpha M}$})$$
whose solutions are linear combinations of $z_1(t)=e^{-\varepsilon t/2M}\cos(\chi t)$ and $z_2(t)=e^{-\varepsilon t/2M}\sin(\chi t)$, where
$\chi=\sqrt{4\alpha M-\varepsilon^2}/2M$. The upper bound for $\varepsilon$ is justified by the fact that a bridge reaches its equilibrium with oscillations and not
monotonically as would occur if $\varepsilon$ overcomes the bound. The question now reduces to: which $\varepsilon>0$ yields solutions of this problem having amplitudes
of oscillations equal to $1/100$ of the initial amplitude after a time $t=40\, s$? Therefore, we need to solve the equation $e^{-20\varepsilon/M}=1/100$ which gives
\begin{empheq}{align}\label{delta}
\varepsilon=\frac{M\log 100}{20}=0.23 M.
\end{empheq}
We emphasize that this value is {\em independent of} $\alpha>0$, that only plays a role in the upper bound for $\varepsilon$.\par
Next, we turn our attention to the aerodynamic parameters:\par\smallskip
$\rho$ = air density\par
$W$ = scalar velocity of the wind blowing on the plate\par
$C_L$ = aerodynamic coefficient of lift\par
St = Strouhal number\par
$AE/2L$ = coefficient of nonlinear stretching (see formula (1) in \cite{burg}).\par\smallskip
The function $g:\Omega \times [0,T]\rightarrow\mathbb{R}$ represents the vertical load over the plate and may depend on space and time.
In bridges, the vertical loads can be either pedestrians, vehicles, or the vortex shedding due to the wind: we focus our attention on the latter.
In absence of wind and external loads, the deck of a bridge remains still; when the wind hits the deck (a bluff body) the flow is modified and goes around
the deck, this creates alternating low-pressure vortices on the downstream side of the deck which then tends to move towards the low-pressure zone.
Therefore, behind the deck, the flow creates vortices which are, in general, asymmetric. This asymmetry generates a forcing lift which starts the
vertical oscillations of the deck. This forcing lift is described by $g$.
We asked to Mario De Miranda, a worldwide renowned civil engineer from the Consulting Engineering Firm {\em De Miranda Associati} \cite{dma}, to
describe the force due to the vortex shedding. He told us that they are usually modeled following the {\em European Eurocode 1} \cite{eurocode} and he
suggested \cite{mdm} that a simplified but quite accurate forcing term due to the vortex shedding may be found as follows. One can assume that it does not
depend on the position $\xi$ nor on the motion of the structure and that it acts only on the vertical component of the motion. Moreover, it varies periodically
with the same law governing the vortex shedding, that is,
\begin{equation}\label{deMiranda}
g(t)=\frac\rho2  W^2 \frac{H}{2\ell} C_L\sin(\omega t)
\end{equation}
with $\omega={\rm St}W/H$,  see respectively \cite[(8.2)]{eurocode}, \cite{Giosan} and \cite[(2)]{billah}. The Strouhal number St is a dimensionless
number describing oscillating flow mechanisms, see for instance \cite[p.120]{billah} and \cite[Figure E.1]{eurocode}:
it depends on the shape and measures of the cross-section of the deck.\par
By a convenient change of scales, \eqref{original} reduces to \eqref{enlm} with
\begin{equation}\label{coefficients}
\delta=\frac{L^2}{\pi^2}\frac{\varepsilon}{\sqrt{D\cdot M}}\  \text{ and }\ S=\frac{A\, E\, L}{2D\pi^2}.
\end{equation}
Therefore, $S>0$ depends on the elasticity of the material composing the plate and $S\int_\Omega u_x^2$ measures the
geometric nonlinearity of the plate due to its stretching.\par
The scaled edges of the plate now measure
\begin{equation}\label{new-dim}
L'=\pi,\quad H' = \frac\pi L H,\quad \ell' = \frac\pi L \ell,
\end{equation}
whereas, from \eqref{deMiranda} and the new time and space scales, the forcing term (still denoted by $g$ for simplicity) can be taken as
\begin{equation}\label{newf}
g(t) = W^2 \sin(\omega' t),
\end{equation}
where $\omega' = \sqrt{\frac{M}{D}}\frac{L^2}{\pi^2}\omega$.\par
The parameter $P>0$ has not been modified while going from \eqref{original} to \eqref{enlm} because it
represents prestressing and the exact value is not really important. One just needs to know that
it usually belongs to the interval $[0,\lambda_2)$ (we will in fact always assume that $0\le P<\lambda_1$), where $\lambda_{1}$ and $\lambda_{2}$ are the first and the second eigenvalues of the linear stationary operator, see \eqref{eq:eigenvalueH2L2} below.\par
The functions $u_0$ and $v_0$ are, respectively, the initial position and velocity of the plate.
The boundary conditions on the short edges are named after Navier \cite{navier} and model the fact that the plate is hinged;
note that $u_{xx}=\Delta u$ on $\{0,\pi\}\times(-\ell,\ell)$. The boundary conditions on the long edges model the fact that the plate is free;
they may be derived with an integration by parts as in \cite{mansfield,ventsel}. We refer to \cite{2014al-gwaizNATMA,2015ferreroDCDSA} for the derivation
of \eqref{enlm}, to the recent monograph \cite{2015gazzola} for the complete updated story, and to \cite{villaggio} for a classical reference on models
from elasticity. The behavior of rectangular plates subject to a variety of boundary conditions is studied in \cite{braess,grunau,gruswe,gruswe2,sweers}. Finally, we mention that equations of the kind of \eqref{enlm} (but with with a slightly different structure) have been considered in
\cite{ghisigobbinoharaux}, with the purpose of analyzing the stability of stationary solutions.

\section{Longitudinal and torsional eigenfunctions} \label{section:eigenvalue}

Throughout this paper we deal with the functional space
\begin{empheq}{align}
H^{2}_*(\Omega)=\{U\in H^2(\Omega); \ U=0 \ \textrm{on }\{0,\pi\}\times[-\ell,\ell]\}\,,
\end{empheq}
and with its dual space $(H^{2}_*(\Omega))'$.
We use the angle brackets $\langle{\cdot, \cdot}\rangle$ to denote the duality of $(H^{2}_*(\Omega))'\times H^{2}_*(\Omega)$, $({\cdot, \cdot})_{L^2}$
for the inner product in $L^2(\Omega)$ with the corresponding norm $\|{ \cdot }\|_{L^2}$, $({\cdot \, , \cdot})_{H^2_*}$ for the inner
product in $H^{2}_*(\Omega)$ defined by
\begin{empheq}{align}(U,V)_{H^2_*}\!=\!\int_{\Omega}\left( \Delta U \Delta V\!-\!(1\!-\!\sigma)\big(U_{xx}V_{yy}\!+\!U_{yy}V_{xx}\!-\!2U_{xy}V_{xy}
\big) \right), \quad U,V\in H^{2}_*(\Omega)\,.
\end{empheq}
Since $\sigma\in(0,1)$, see \eqref{conditionsigma}, this inner product defines a norm which makes $H^{2}_*(\Omega)$ a Hilbert space; see
\cite[Lemma 4.1]{2015ferreroDCDSA}.\par
Our first purpose is to introduce a suitable basis of $H^{2}_*(\Omega)$ and to define what we mean by vibrating modes of \eqref{enlm}. To this end, we consider
the eigenvalue problem
\begin{equation}\label{eq:eigenvalueH2L2}
\left\{
\begin{array}{rl}
\Delta^2 w = \lambda w& \textrm{in }\Omega \\
w = w_{xx} = 0& \textrm{on }\{0,\pi\}\times[-\ell,\ell]\\
w_{yy}+\sigma w_{xx} = 0& \textrm{on }[0,\pi]\times\{-\ell,\ell\}\\
w_{yyy}+(2-\sigma)w_{xxy} = 0& \textrm{on }[0,\pi]\times\{-\ell,\ell\}
\end{array}
\right.
\end{equation}
which can be rewritten as $(w,z)_{H^2_*}=\lambda (w,z)_{L^2}$ for all $z\in H^{2}_*(\Omega)$. By combining results in
\cite{bebuga,bfg1,2015ferreroDCDSA}, we obtain the following statement.

\begin{proposition}\label{spectrum}
The set of eigenvalues of \eqref{eq:eigenvalueH2L2} may be ordered in an increasing sequence of strictly positive numbers diverging to $+\infty$ and
any eigenfunction belongs to $C^\infty(\overline\Omega)$. The set of eigenfunctions of \eqref{eq:eigenvalueH2L2} is a complete system in $H^2_*(\Omega)$.
Moreover:\par\noindent
$(i)$ for any $m\ge1$, there exists a unique eigenvalue $\lambda=\mu_{m,1}\in((1-\sigma^2)m^4,m^4)$ with corresponding eigenfunction
$$\left[\big[\mu_{m,1}^{1/2}-(1-\sigma)m^2\big]\, \tfrac{\cosh\Big(y\sqrt{m^2+\mu_{m,1}^{1/2}}\Big)}{\cosh\Big(\ell\sqrt{m^2+\mu_{m,1}^{1/2}}\Big)}+
\big[\mu_{m,1}^{1/2}+(1-\sigma)m^2\big]\, \tfrac{\cosh\Big(y\sqrt{m^2-\mu_{m,1}^{1/2}}\Big)}{\cosh\Big(\ell\sqrt{m^2-\mu_{m,1}^{1/2}}\Big)}\right]\sin(mx)\, ;$$
$(ii)$ for any $m\ge1$ and any $k\ge2$ there exists a unique eigenvalue $\lambda=\mu_{m,k}>m^4$ satisfying\par
$\left(m^2+\frac{\pi^2}{\ell^2}\left(k-\frac{3}{2}\right)^2\right)^2<\mu_{m,k}<\left(m^2+\frac{\pi^2}{\ell^2}\left(k-1\right)^2\right)^2$
and with corresponding eigenfunction
$$
\left[\big[\mu_{m,k}^{1/2}-(1-\sigma)m^2\big]\, \tfrac{\cosh\Big(y\sqrt{\mu_{m,k}^{1/2}+m^2}\Big)}{\cosh\Big(\ell\sqrt{\mu_{m,k}^{1/2}+m^2}\Big)}
+\big[\mu_{m,k}^{1/2}+(1-\sigma)m^2\big]\, \tfrac{\cos\Big(y\sqrt{\mu_{m,k}^{1/2}-m^2}\Big)}{\cos\Big(\ell\sqrt{\mu_{m,k}^{1/2}-m^2}\Big)}\right]\sin(mx)\, ;
$$
$(iii)$ for any $m\ge1$ and any $k\ge2$ there exists a unique eigenvalue $\lambda=\nu_{m,k}>m^4$ with corresponding eigenfunctions
$$
\left[\big[\nu_{m,k}^{1/2}-(1-\sigma)m^2\big]\, \tfrac{\sinh\Big(y\sqrt{\nu_{m,k}^{1/2}+m^2}\Big)}{\sinh\Big(\ell\sqrt{\nu_{m,k}^{1/2}+m^2}\Big)}
+\big[\nu_{m,k}^{1/2}+(1-\sigma)m^2\big]\, \tfrac{\sin\Big(y\sqrt{\nu_{m,k}^{1/2}-m^2}\Big)}{\sin\Big(\ell\sqrt{\nu_{m,k}^{1/2}-m^2}\Big)}\right]\sin(mx)\, ;
$$
$(iv)$ for any $m\ge1$ satisfying $\tanh(\sqrt{2}m\ell)<\left(\tfrac{\sigma}{2-\sigma}\right)^2\sqrt{2}m\ell$ there exists a unique
eigenvalue $\lambda=\nu_{m,1}\in(\mu_{m,1},m^4)$ with corresponding eigenfunction
$$\left[\big[\nu_{m,1}^{1/2}-(1-\sigma)m^2\big]\, \tfrac{\sinh\Big(y\sqrt{m^2+\nu_{m,1}^{1/2}}\Big)}{\sinh\Big(\ell\sqrt{m^2+\nu_{m,1}^{1/2}}\Big)}
+\big[\nu_{m,1}^{1/2}+(1-\sigma)m^2\big]\, \tfrac{\sinh\Big(y\sqrt{m^2-\nu_{m,1}^{1/2}}\Big)}{\sinh\Big(\ell\sqrt{m^2-\nu_{m,1}^{1/2}}\Big)}\right]
\sin(mx)\, .$$
\end{proposition}

In fact, if the unique positive solution $s>0$ of the equation
\begin{empheq}{align}\label{never}
\tanh(\sqrt{2}s\ell)=\left(\frac{\sigma}{2-\sigma}\right)^2\, \sqrt{2}s\ell
\end{empheq}
is not an integer, then the only eigenvalues and eigenfunctions are the ones given in Proposition \ref{spectrum}.
Condition \eqref{never} has probability 0 to occur in general plates; if it occurs, there is an additional eigenvalue and eigenfunction,
see \cite{2015ferreroDCDSA}. In particular, if we assume \eqref{lsigma}, then \eqref{never} is not satisfied and no eigenvalues of \eqref{eq:eigenvalueH2L2}
other than $(i)-(ii)-(iii)-(iv)$ exist.

\begin{remark}
The nodal regions of the eigenfunctions found in Proposition \ref{spectrum} all have a rectangular shape.
The indexes $m$ and $k$ quantify the number of nodal regions of the eigenfunction in the $x$ and $y$ directions. More precisely,
$\mu_{m,k}$ is associated to a longitudinal eigenfunction having $m$ nodal regions in the $x$-direction and $2k-1$ nodal regions in the $y$-direction
whereas $\nu_{m,k}$ is associated to a torsional eigenfunction having $m$ nodal regions in the $x$-direction and $2k$ nodal regions in the $y$-direction.
Hence, the only positive eigenfunction (having only one nodal region in both directions) is associated to $\mu_{1,1}$.
\end{remark}

From \cite{FerGazMor} we know that the least eigenvalue $\lambda_1$ of \eqref{eq:eigenvalueH2L2} satisfies
$$
\lambda_1:=\mu_{1,1}\, =\, \min_{v\in H^2_*}\ \frac{\|v\|_{H^2_*}^2}{\|v_x\|_{L^2}^2}\, =\, \min_{v\in H^2_*}\ \frac{\|v\|_{H^2_*}^2}{\|v\|_{L^2}^2}\qquad
\mbox{and}\qquad\min_{v\in H^2_*}\ \frac{\|v_x\|_{L^2}^2}{\|v\|_{L^2}^2}=1\, .
$$
These three identities yield the following embedding inequalities:
\begin{equation}\label{embedding}
\|v\|_{L^2}^2\le\|v_x\|_{L^2}^2\, ,\quad\lambda_1\|v\|_{L^2}^2\le\|v\|_{H^2_*}^2\, ,\quad\lambda_1\|v_x\|_{L^2}^2\le\|v\|_{H^2_*}^2\qquad
\forall v\in H^2_*(\Omega)\, .
\end{equation}

Proposition \ref{spectrum} states that for any $m\ge1$ there exists a divergent sequence of eigenvalues (as $i\to\infty$, including both $\mu_{m,i}$
and $\nu_{m,i}$) with corresponding eigenfunctions
\begin{equation}\label{eigenfunction}
w_{m,i} (x,y)= \varphi_{m,i}(y)\sin(mx)\, , \ \ m,i\in{\mathbb{N}}\, .
\end{equation}
The functions $\varphi_{m,i}$ are linear combinations of hyperbolic and trigonometric sines and cosines, being either even or odd with respect
to $y$.

Observe that if $w_{1,1}$ is $L^2$-normalized, then we also have $\|(w_{1,1})_{xx}\|_{L^2}^2=1$ and  $\|w_{1,1}\|_{H^2_*}^2=\lambda_1$, whereas for every
$v\in H^2_*(\Omega)$, we have
\begin{empheq}{align}
\|v\|_{H^2_*}^2=\int_{\Omega}\left( v_{xx}^2+v_{yy}^2+2(1-\sigma)v_{xy}^2  +2\sigma v_{xx}v_{yy}\right)\ge (1-\sigma^2) \int_{\Omega} v_{xx}^2.
\end{empheq}
This shows that the inequality
\begin{equation}\label{embedding-H^2_*-u_xx}
\gamma \|v_{xx}\|_{L^2}^2\le\|v\|_{H^2_*}^2\qquad\forall v\in H^2_*(\Omega)
\end{equation}
holds for some optimal constant $\gamma\in [1-\sigma^2,\lambda_1]$.

\begin{definition}[Longitudinal/torsional eigenfunctions]\label{df:longtors}
If $\varphi_{m,i}$ is even we say that the eigenfunction \eqref{eigenfunction} is longitudinal while if $\varphi_{m,i}$ is odd we say that the
eigenfunction \eqref{eigenfunction} is torsional.
\end{definition}

Let us now explain how the eigenfunctions of \eqref{eq:eigenvalueH2L2} enter in the stability analysis of \eqref{enlm}.
We approximate the solution of \eqref{enlm} through its decomposition in Fourier components. The numerical results obtained
in \cite{FerGazMor} suggest to restrict the attention to the lower eigenvalues. In order to select a reasonable number of
low eigenvalues, let us exploit what was seen at the TNB.
The already mentioned description by Farquharson \cite[V-10]{ammann} ({\em the motions, which a moment before had involved a number
of waves (nine or ten) had shifted almost instantly to two}) shows that an instability occurred and changed the motion of the deck from the ninth or
tenth longitudinal mode to the second torsional mode. In fact, Smith-Vincent \cite[p.21]{tac2} state that this shape of torsional oscillations is the only
possible one, see also \cite[Section 1.6]{2015gazzola} for further evidence and more historical facts. Therefore, the relevant eigenvalues corresponding
to oscillations visible in actual bridges should include (at least!) ten longitudinal modes and two torsional modes.\par
Following Section \ref{quant} and the measures of the TNB, we take $\Omega = (0, \pi) \times (-\ell, \ell)$ with
\begin{empheq}{align}\label{lsigma}
\ell=\frac{\pi}{150}\ ,\quad\sigma=0.2\, .
\end{empheq}
Let us determine the least 20 eigenvalues of \eqref{eq:eigenvalueH2L2} when \eqref{lsigma} holds.
In this case, from $(ii)$ and $(i)$ we learn that
\begin{empheq}{align}\label{learn1}
\mu_{m,k}\ge\mu_{1,2}>\left(1+\frac{\pi^2}{4\ell^2}\right)^2=(1+75^2)^2>75^4>\mu_{75,1}\qquad\forall m\ge1\, ,\ k\ge2\, ,
\end{empheq}
so that the least $75$ longitudinal eigenvalues are all of the kind $(i)$. Therefore,
\begin{center}
{\em no eigenvalues of the kind $\mu_{m,2}$ in $(ii)$ are among the least 20.}
\end{center}
Concerning the torsional eigenvalues, we need to distinguish two cases. If
$$
\tanh(\sqrt{2}m\ell)<\left(\tfrac{\sigma}{2-\sigma}\right)^2\, \sqrt{2}m\ell\,,
$$
then from \eqref{lsigma} we infer that $m\ge2\,734$ so that, by $(iv)$,
\begin{empheq}{align}\label{learn2}
\nu_{m,1}\ge\nu_{2734,1}>\mu_{2734,1}\qquad\mbox{for all $m\ge2\,734$}\, .
\end{empheq}
Therefore,
\begin{center}
{\em no eigenvalues of the kind $\nu_{m,1}$ in $(iv)$ are among the least 20.}
\end{center}
If
\begin{empheq}{align}\label{iff1}
\tanh(\sqrt{2}m\ell)>\left(\tfrac{\sigma}{2-\sigma}\right)^2\, \sqrt{2}m\ell\, ,
\end{empheq}
then the torsional eigenfunction $w_{m,k}$ with $k\ge2$ is given by $(iii)$ and from \cite{2015ferreroDCDSA} we know that the associated eigenvalue
$\nu_{m,k}$ are the solutions $\lambda>m^4$ of the equation
$$
\sqrt{\lambda^{1/2}\!-\!m^2}\big(\lambda^{1/2}\!+\!(1\!-\!\sigma)m^2\big)^2\tanh(\ell\sqrt{\lambda^{1/2}\!+\!m^2})\!=
\!\sqrt{\lambda^{1/2}\!+\!m^2}\big(\lambda^{1/2}\!-\!(1\!-\!\sigma)m^2\big)^2\tan(\ell\sqrt{\lambda^{1/2}\!-\!m^2})\, .
$$
Put $s=\lambda^{1/2}$ and, related to this equation, consider the function
\begin{eqnarray*}
Z(s) &:=& \sqrt{s^2\!-\!m^4}\big[s\!-\!(1\!-\!\sigma)m^2\big]^2
\left\{\left(\frac{s\!+\!(1\!-\!\sigma)m^2}{s\!-\!(1\!-\!\sigma)m^2}\right)^2\, \frac{\tanh(\ell\sqrt{s\!+\!m^2})}{\sqrt{s\!+\!m^2}}\!-
\frac{\tan(\ell\sqrt{s\!-\!m^2})}{\sqrt{s\!-\!m^2}}\right\}\\
\ &=:& \sqrt{s^2\!-\!m^4}\big[s\!-\!(1\!-\!\sigma)m^2\big]^2\, Z(s)\, .
\end{eqnarray*}
In each of the subintervals of definition for $Z$ (and $s>m^2$), the maps $s\mapsto\frac{s\!+\!(1\!-\!\sigma)m^2}{s\!-\!(1\!-\!\sigma)m^2}$, $s\mapsto\frac{\tanh(\ell\sqrt{s\!+\!m^2})}{\sqrt{s\!+\!m^2}}$, and $s\mapsto-\frac{\tan(\ell\sqrt{s\!-\!m^2})}{\sqrt{s\!-\!m^2}}$
are strictly decreasing, the first two being also positive. Since, by \eqref{iff1},
$\lim_{s\to m^2}Z(s)=\left(\frac{2-\sigma}\sigma \right)^2\, \frac{\tanh(\sqrt2 \, \ell m)}{\sqrt2 \, m}-\ell>0$, the function $Z$ starts positive,
ends up negative and it is strictly decreasing in any subinterval, it admits exactly one zero there, when $\tan(\ell\sqrt{s\!-\!m^2})$ is positive.
Hence, $Z$ has exactly one zero on any interval and we have proved that
\begin{equation}\label{doublebound}
\big(m^2+\tfrac{\pi^2}{\ell^2}(k-2)^2\big)^2<\nu_{m,k}<\big(m^2+\tfrac{\pi^2}{\ell^2}(k-\tfrac32 )^2\big)^2\qquad\forall k\ge2\, ;
\end{equation}
in particular,
\begin{empheq}{align}\label{learn3}
\nu_{m,k}\ge\nu_{1,3}>(1+150^2)^2>150^4>\mu_{150,1}\qquad\forall m\ge1\, ,\ k\ge3\, ,
\end{empheq}
so that
\begin{center}
{\em no eigenvalues of the kind $\nu_{m,k}$ in $(iii)$ with $k\ge3$ are among the least 20.}
\end{center}

Summarizing, from \eqref{learn1}-\eqref{learn2}-\eqref{learn3}, we infer that the candidates to be among the least 20 eigenvalues are the
$\mu_{m,1}$ in $(i)$ and the $\nu_{m,2}$ in $(iii)$. By taking \eqref{lsigma}, we numerically find the least 20 eigenvalues of \eqref{eq:eigenvalueH2L2}
as reported in Table \ref{tableigen}.

\begin{table}[ht!]
\begin{center}
\begin{tabular}{|c|c|c|c|c|c|c|c|c|c|c|}
\hline
eigenvalue & $\lambda_1$ & $\lambda_2$ & $\lambda_3$ & $\lambda_4$ & $\lambda_5$ & $\lambda_6$ & $\lambda_7$ & $\lambda_8$ & $\lambda_9$ & $\lambda_{10}$\\
\hline
kind & $\mu_{1,1}$ & $\mu_{2,1}$ & $\mu_{3,1}$ & $\mu_{4,1}$ & $\mu_{5,1}$ & $\mu_{6,1}$  & $\mu_{7,1}$ & $\mu_{8,1}$ & $\mu_{9,1}$  & $\mu_{10,1}$\\
\hline
$\sqrt{\lambda_j}\approx$ & $0.98$ & $3.92$ & $8.82$ & $15.68$ & $24.5$ & $35.28$ & $48.02$ & $62.73$ & $79.39$ & $98.03$\\
\hline
\end{tabular}
\par\smallskip
\begin{tabular}{|c|c|c|c|c|c|c|c|c|c|c|}
\hline
eigenvalue & $\lambda_{11}$ & $\lambda_{12}$ & $\lambda_{13}$ & $\lambda_{14}$ & $\lambda_{15}$ & $\lambda_{16}$ & $\lambda_{17}$ & $\lambda_{18}$ &
$\lambda_{19}$ & $\lambda_{20}$\\
\hline
kind & $\nu_{1,2}$ & $\mu_{11,1}$ & $\mu_{12,1}$ & $\mu_{13,1}$ & $\mu_{14,1}$ & $\nu_{2,2}$ & $\mu_{15,1}$ & $\mu_{16,1}$ & $\mu_{17,1}$ & $\nu_{3,2}$\\
\hline
$\sqrt{\lambda_j}\approx$ & $104.61$ & $118.62$ & $141.19$ & $165.72$ & $192.21$ & $209.25$ & $220.68$ & $251.12$ & $283.53$ & $313.94$\\
\hline
\end{tabular}\vskip1mm
\caption{Approximate value of the least 20 eigenvalues of (\ref{eq:eigenvalueH2L2}), assuming (\ref{lsigma}).}\label{tableigen}
\end{center}
\end{table}
We reported the squared roots since these are the values to be used while explicitly writing the eigenfunctions, see Proposition \ref{spectrum}.
We also refer to \cite{bfg1} for numerical values of the eigenvalues for other choices of $\sigma$ and $\ell$: although their values are
slightly different, the eigenvalues maintain the same order.\par
By combining Proposition \ref{spectrum} with Table \ref{tableigen} we find that the eigenfunctions corresponding to the least 20 eigenvalues
of \eqref{eq:eigenvalueH2L2}, labeled by a unique index $k$, are given by:\par
$\bullet$ 17 longitudinal eigenfunctions: for $k\in\{1,2,3,4,5,6,7,8,9,10,12,13,14,15,17,18,19\}$ with corresponding
$m_k\in\{1,2,3,4,5,6,7,8,9,10,11,12,13,14,15,16,17\}$
$$w_k(x,y)=\left[\big[\lambda_k^{1/2}-\tfrac45 m_k^2\big]\, \tfrac{\cosh\Big(y\sqrt{m_k^2+\lambda_k^{1/2}}\Big)}{\cosh\Big(\tfrac{\pi}{150}\sqrt{m_k^2+\lambda_k^{1/2}}\Big)}+
\big[\lambda_k^{1/2}+\tfrac45 m_k^2\big]\, \tfrac{\cosh\Big(y\sqrt{m_k^2-\lambda_k^{1/2}}\Big)}{\cosh\Big(\tfrac{\pi}{150}\sqrt{m_k^2-\lambda_k^{1/2}}\Big)}\right]\sin(m_kx)\, ;$$

$\bullet$ 3 torsional eigenfunctions: for $k\in\{11,16,20\}$ with corresponding $m_k\in\{1,2,3\}$
$$
w_k(x,y)=\left[\big[\lambda_k^{1/2}-\tfrac45 m_k^2\big]\, \tfrac{\sinh\Big(y\sqrt{\lambda_k^{1/2}+m_k^2}\Big)}{\sinh\Big(\tfrac{\pi}{150}\sqrt{\lambda_k^{1/2}+m_k^2}\Big)}
+\big[\lambda_k^{1/2}+\tfrac45 m_k^2\big]\, \tfrac{\sin\Big(y\sqrt{\lambda_k^{1/2}-m_k^2}\Big)}{\sin\Big(\tfrac{\pi}{150}\sqrt{\lambda_k^{1/2}-m_k^2}\Big)}\right]\sin(m_kx)\, .
$$

We consider these $20$ eigenfunctions of \eqref{eq:eigenvalueH2L2}, we label them with a unique index $k$, and we shorten their explicit form with
\begin{equation}\label{twenty}
w_k(x,y)=\varphi_k(y)\sin(m_kx)\qquad(k=1,...,20)\, :
\end{equation}
we denote by $\lambda_k$ the corresponding eigenvalue. We assume that the $w_k$ are normalized in $L^2(\Omega)$:
\begin{empheq}{align}\label{normalized}
1=\int_\Omega w_k^2=\int_{|y|<\ell}\varphi_k(y)^2\cdot\int_0^\pi\sin^2(m_kx)\ \Longrightarrow\ \int_{|y|<\ell}\varphi_k(y)^2=\frac{2}{\pi}\, .
\end{empheq}
Then we define the numbers
\begin{equation}\label{gammak}
\gamma_k=\int_\Omega w_k
\end{equation}
and we remark that
\begin{equation}\label{gamma0}
\gamma_k=0\mbox{ if $w_k$ is a torsional eigenfunction or if $m_k$ is even}
\end{equation}
since for odd $\varphi_k$ one has $\int_{|y|<\ell}\varphi_k(y)=0$, whereas for even $m_k$ one has $\int_0^\pi\sin(m_kx)=0$.
For the remaining $\gamma_k$ (corresponding to longitudinal eigenfunctions with odd $m_k$), from \eqref{normalized} and the H\"older inequality we deduce
\neweq{Holder}
\gamma_k=\int_{|y|<\ell}\varphi_k(y)\cdot\int_0^\pi\sin(m_kx)\le\sqrt{2\ell}\left(\int_{|y|<\ell}\varphi_k(y)^2\right)^{1/2}\cdot\frac{2}{m_k}=
\frac{4}{m_k}\sqrt{\frac{\ell}{\pi}}\, .
\endeq
By assuming \eqref{lsigma}, this estimate becomes
$$
\gamma_k\le \frac{4}{m_k\sqrt{150}}=:\overline{\gamma}_k\, .
$$
In Table \ref{tablegamma} we quote the values of $\gamma_k$ for the symmetric (with respect to $x=\frac\pi2$) longitudinal eigenfunctions within
the above family and of their bound $\overline{\gamma}_k$. It turns out that $\gamma_k\approx \overline{\gamma}_k$ for all $k$.

\begin{table}[ht]
\begin{center}
\begin{tabular}{|c|c|c|c|c|c|c|c|c|c|}
\hline
\!eigenvalue\! & $\lambda_1$ & $\lambda_3$ & $\lambda_5$ & $\lambda_7$ & $\lambda_9$ & $\lambda_{12}$ & $\lambda_{14}$ & $\lambda_{17}$ & $\lambda_{19}$\\
\hline
$10\gamma_k\approx$\!&\!$3.26599$\!&\!$1.08866$\!&\!$0.653197$\!&\!$0.466569$\!&\!$0.362887$\!&\!$0.296908$\!&\!$0.251229$\!&\!$0.217732$\!&\!$0.192116$\!\\
\hline
$10\overline{\gamma}_k\approx$\!&\!$3.26599$\!&\!$1.08866$\!&\!$0.653197$\!&\!$0.466569$\!&\!$0.362887$\!&\!$0.296908$\!&\!$0.25123$\!&\!$0.217732$\!&\!$0.192117$\!\\
\hline
\end{tabular}\vskip1mm
\caption{Approximate value of $\gamma_k$ and $\overline{\gamma}_k$, as defined in Proposition \ref{spectrum} and (\ref{Holder}), assuming
(\ref{lsigma}).}\label{tablegamma}
\end{center}
\end{table}

We conclude this section by introducing the subspaces of even and odd functions with respect to $y$:
$$\begin{array}{cc}
\he:=\{u\in\hs:\, u(x,-y)=u(x,y)\ \forall(x,y)\in\Omega\}\, ,\\
\ho:=\{u\in\hs:\, u(x,-y)=-u(x,y)\ \forall(x,y)\in\Omega\}\, .
\end{array}$$
Then we have
\begin{equation}\label{decomposition}
\he\perp\ho\, ,\qquad \hs=\he\oplus\ho
\end{equation}
and, for all $u\in\hs$, we denote by $u^L\in\he$ and $u^T\in\ho$ its components according to this decomposition:
\begin{equation}\label{eq:DEO}
u^L(x,y)=\frac{u(x,y)+u(x,-y)}{2}\, ,\qquad u^T(x,y)=\frac{u(x,y)-u(x,-y)}{2}\, .
\end{equation}
The space $\he$ is spanned by the longitudinal eigenfunctions (classes $(i)$ and $(ii)$ in Proposition \ref{spectrum}) whereas the space
$\ho$ is spanned by the torsional eigenfunctions (classes $(iii)$ and $(iv)$). We will use these spaces to decompose the solutions of \eqref{enlm}
in their longitudinal and torsional components.\par
For all $\alpha>0$, it will be useful to study the time evolution of the ``energies'' defined by
\begin{equation}\label{energyalpha}
E_\alpha(t):=\frac12 \|u_t(t)\|_{L^2}^2 +\frac12 \|u(t)\|_{H^2_*}^2 -\frac{P}{2}\|u_x(t)\|_{L^2}^2+\frac{S}{4}\|u_x(t)\|_{L^2}^4+\alpha
\into u\xit u_t\xit \, d\xi\, .
\end{equation}
This energy can be decomposed according to \eqref{decomposition} as
\begin{equation}\label{energydecomposition}
E_\alpha(t)=E_\alpha^L(t)+E_\alpha^T(t)+E_\alpha^C(t)
\end{equation}
$$=\frac12 \|u_t^L(t)\|_{L^2}^2 +\frac12 \|u^L(t)\|_{H^2_*}^2 -\frac{P}{2}\|u_x^L(t)\|_{L^2}^2+\frac{S}{4}\|u_x^L(t)\|_{L^2}^4+\alpha
\into u^L\xit u^L_t\xit \, d\xi$$
$$+\frac12 \|u_t^T(t)\|_{L^2}^2 +\frac12 \|u^T(t)\|_{H^2_*}^2 -\frac{P}{2}\|u_x^T(t)\|_{L^2}^2+\frac{S}{4}\|u_x^T(t)\|_{L^2}^4+\alpha
\into u^T\xit u^T_t\xit \, d\xi$$
$$+\frac{S}{2}\|u_x^L(t)\|_{L^2}^2\|u_x^T(t)\|_{L^2}^2\, ,$$
where $E_\alpha^L$ represents the longitudinal energy, $E_\alpha^T$ the torsional energy, $E_\alpha^C$ the coupling energy.

\section{Main results}\label{main}

In this section we present our results concerning the problem

\begin{empheq}{align}\label{enlmg}
\left\{
\begin{array}{rl}
u_{tt}+\delta u_t + \Delta^2 u +\left[P-S\int_\Omega u_x^2\right]u_{xx}= g(\xi,t)  &\textrm{in }\Omega\times(0,T)\\
u = u_{xx}= 0 &\textrm{on }\{0,\pi\}\times[-\ell,\ell]\\
u_{yy}+\sigma u_{xx} = u_{yyy}+(2-\sigma)u_{xxy}= 0 &\textrm{on }[0,\pi]\times\{-\ell,\ell\}
\end{array}
\right.
\end{empheq}
complemented with some initial conditions
\begin{equation}\label{initialc}
u(\xi,0) = u_0(\xi), \quad \quad u_t(\xi,0) = v_0(\xi)\qquad\textrm{in }\Omega\, .
\end{equation}

Let us first make clear what we mean by solution of \eqref{enlmg}.

\begin{definition}[Weak solution]\label{df:weaksolution}
Let $g\in C^0([0,T],L^2(\Omega))$ for some $T>0$. A weak solution of \eqref{enlmg} is a function
\begin{empheq}{align}
u\in C^ 0([0,T],H^{2}_*(\Omega))\cap C^1([0,T],L^2(\Omega))\cap C^2([0,T],(H^{2}_*(\Omega))')
\end{empheq}
such that
\begin{empheq}{align}\label{weakform}
\langle u_{tt},v \rangle + \delta (u_t,v)_{L^2} + (u,v)_{H^2_*} +\big[S\|u_x\|_{L^2}^2-P\big](u_x,v_x)_{L^2}= (g,v)_{L^2}\,,
\end{empheq}
for all $t\in[0,T]$ and all $v\in H^{2}_*(\Omega)$.
\end{definition}

The following result holds.

\begin{theorem}\label{exuniq}
Given $\delta>0$, $S>0$, $P\in[0,\lambda_1)$, $T>0$, $g\in C^0([0,T],L^2(\Omega))$, $u_0\in H^{2}_*(\Omega)$ and $ v_0\in L^2(\Omega)$,
there exists a unique weak solution $u$ of \eqref{enlmg}-\eqref{initialc}.  Moreover, if $g\in C^1([0,T],L^2(\Omega))$, $u_0\in H^4\cap H^{2}_*(\Omega)$
and $v_0\in H^2_*(\Omega)$, then
$$u\in C^ 0([0,T],H^4\cap H^{2}_*(\Omega))\cap C^1([0,T],H^2_*(\Omega))\cap C^2([0,T],L^2(\Omega))$$
and $u$ is a strong solution of  \eqref{enlmg}-\eqref{initialc}.
\end{theorem}

We point out that alternative regularity results may be obtained under different assumptions on the source, see e.g.\ \cite[Theorem 2.2.1]{harauxbook}.

From now on we are interested in global in time solutions and their torsional stability (in a suitable sense): for this analysis, a crucial role is played by periodic solutions.

\begin{theorem}\label{th:peridoc}
Let $\delta>0$, $S>0$, $P\in[0,\lambda_1)$, $g\in C^0(\mathbb{R},L^2(\Omega))$. If $g$ is $\tau$-periodic in time for some $\tau>0$ (that is, $g(\xi,t+\tau)=g(\xi,t)$ for all $\xi$ and $t$), then there exists a $\tau$-periodic solution of \eqref{enlmg}.
\end{theorem}

Let $\{w_k\}$ denote the sequence of all the eigenfunctions of \eqref{eq:eigenvalueH2L2} labeled with a unique index $k$ and, for a given
$g\in C^0(\mathbb{R}_+,L^2(\Omega))$, let
\begin{equation}\label{coef-g}
g_k(t)=\int_\Omega g(\xi,t)w_k(\xi)\, d\xi\, .
\end{equation}
Also write a solution $u$ of \eqref{enlmg} in the form
\begin{equation}\label{Fourier}
u(\xi,t)=\sum_{k=1}^\infty h_k(t)w_k(\xi)\, ,
\end{equation}
so that $u$ is identified by its Fourier coefficients which satisfy the infinite dimensional system
\begin{equation}\label{infsystem}
\ddot{h}_k(t)+\delta\dot{h}_k(t)+\lambda_kh_k(t)+m_k^2\left[-P+S\sum_{j=1}^\infty m_j^2h_j(t)^2\right]h_k(t)=g_k(t)
\end{equation}
for all integer $k$, where $m_k$ is the frequency in the $x$-direction, see \eqref{twenty}. In fact, more can be said. According to Proposition
\ref{spectrum}, the eigenfunctions $w_k$ of \eqref{eq:eigenvalueH2L2} belong to two categories: longitudinal and torsional. Then we use the decomposition
\eqref{decomposition} in order to write \eqref{Fourier} in the alternative form
\begin{equation}\label{LT}
u(\xi,t)=u^L(\xi,t)+u^T(\xi,t)\, ,
\end{equation}
that is, by emphasizing its longitudinal and torsional parts.\par

\begin{definition}[Torsional stability/instability]\label{defstab}
We say that $g=g(\xi,t)$ makes the system \eqref{enlmg} torsionally stable if every solution of \eqref{enlmg}, written in the form \eqref{LT}, satisfies
$\|u_t^T(t)\|_{L^2}+\|u^T(t)\|_{H^2_*}\to 0$ as $t\to\infty$. We say that $g=g(\xi,t)$ makes the system \eqref{enlmg} torsionally unstable if there exists
a solution of \eqref{enlmg} such that $\displaystyle\limsup_{t\to\infty}(\|u_t^T(t)\|_{L^2}+\|u^T(t)\|_{H^2_*})>0$.
\end{definition}

In particular, the embedding $H^2_*(\Omega)\hookrightarrow L^\infty(\Omega)$ enables us to infer that, if $g$ makes
\eqref{enlmg} torsionally stable, then the torsional component of any solution $u$ of \eq{enlmg} tends uniformly to zero, namely,
\[
\lim_{t\to\infty}\|u^T(t)\|_{L^\infty}=0\, .
\]

As we shall see, the torsional stability strongly depends on the amplitude of the force $g$ at infinity. More precisely, it is necessary to assume that
\begin{equation}\label{smallginfty}
g\in C^0(\mathbb{R}_+,L^2(\Omega))\ ,\qquad g_\infty:=\limsup_{t\to\infty}\|g(t)\|_{L^2}<+\infty.
\end{equation}

The torsional stability, as characterized by Definition \ref{defstab}, has a physical interpretation in terms of energy.

\begin{proposition}\label{NSC}
Let $\nu_{1,2}$ be the least torsional eigenvalue, see Proposition \ref{spectrum}.
Let $E_\alpha$ be the energy defined in \eqref{energyalpha} and assume \eqref{smallginfty} and  that $0<\alpha^2<\nu_{1,2}-P$.
Then $g=g(\xi,t)$ makes the system \eqref{enlmg} torsionally stable if and only if every solution of \eqref{enlmg}, written in the form
\eqref{LT} has vanishing torsional energy (see \eqref{energydecomposition}) at infinity:
$$\lim_{t\to\infty}E_\alpha^T(t)=0\, .$$
\end{proposition}

Note that the upper bound $\nu_{1,2}-P$ is very large, see Table \ref{tableigen}, much larger than the values of $\alpha^2$ used to obtain the
energy bounds in Section \ref{enest}. To prove this statement, we observe that \eqref{embedding} may be improved for torsional functions:
\begin{equation}\label{improvedtorsional}
\nu_{1,2}\|v\|_{L^2}^2\le\|v\|_{H^2_*}^2\, ,\quad\nu_{1,2}\|v_x\|_{L^2}^2\le\|v\|_{H^2_*}^2\qquad\forall v\in\ho\, .
\end{equation}
This shows that
$$E_\alpha^T(t)\ge \frac12 \|u_t^T(t)\|_{L^2}^2 +\frac{\nu_{1,2}-P}2 \|u^T(t)\|_{L^2}^2
+\alpha\into u^T\xit u^T_t\xit \, d\xi$$
and, with the assumption on $\alpha$, the right hand side of this inequality is a positive definite quadratic form with respect to
$\|u_t^T(t)\|_{L^2}$ and $\|u^T(t)\|_{L^2}$.\par
We now give a sufficient condition for the torsional stability.

\begin{theorem}\label{SCstability}
Assume that $\delta>0$, $S>0$, $0\le P<\lambda_1$, and \eqref{smallginfty}. There exists $g_0=g_0(\delta,S,P,\lambda_1)>0$ such that if
$g_\infty<g_0$, then:\par
$\bullet$ there exists $\eta>0$ such that, for any couple $(u,v)$ of solutions of \eqref{enlmg}, one has
\begin{equation}\label{squeezing}
\lim_{t\to\infty}{\rm e}^{\eta t}\Big(\|u_t(t)-v_t(t)\|_{L^2}^2+\|u(t)-v(t)\|_{H^2_*}^2\Big)=0\, ;
\end{equation}

$\bullet$ if $g$ is $\tau$-periodic for some $\tau>0$, then \eqref{enlmg} admits a unique periodic solution $U^p$ and
$$\lim_{t\to\infty}{\rm e}^{\eta t}\Big(\|u_t(t)-U^p_t(t)\|_{L^2}^2+\|u(t)-U^p(t)\|_{H^2_*}^2\Big)=0$$

for any other solution $u$ of \eqref{enlmg}; \par
$\bullet$ if $g$ is even with respect to $y$, then $g$ makes the system \eqref{enlmg} torsionally stable.
\end{theorem}

Several comments are in order. Theorem \ref{SCstability} {\em is not} a perturbation statement, the constant $g_0$ can be explicitly computed,
see Lemma \ref{lemme-naze} below. In particular, it can be seen that $g_0\asymp1/\sqrt{S}$ as $S\to\infty$. This shows that the nonlinearity plays
against uniqueness and stability results: for large $S$ only very small forces $g$ ensure the validity of Theorem \ref{SCstability}.

In Section \ref{numres} we will give numerical evidence that Theorem \ref{SCstability} is somehow sharp, for large $g$ the stability statement
seems to be false. Here we show that if $g$ is large then multiple periodic solutions may exist. We recall that a function $w$ is called
$\tau-$antiperiodic if $w(t + \tau) = - w(t)$ for all $t$. In particular, a $\tau-$antiperiodic function is also $2\tau$-periodic.

\begin{theorem}\label{th:multiple-periodic-solutions}
There exist $\tau>0$ and a $\tau$-antiperiodic function $g=g(\xi,t)$, such that the equation \eqref{enlmg}
admits at least two distinct $\tau$-antiperiodic solutions for a suitable choice of the parameters $\delta$, $P$ and $S$.
\end{theorem}

To prove Theorem \ref{th:multiple-periodic-solutions} we follow very closely the arguments in \cite{souplet}. For alternative statements and proofs we refer to \cite{2012GasmiHaraux, souplet2}.

In real life, it is more interesting to consider the converse problem: given a maximal intensity of the wind in the region where the bridge will
be built, can one design a structure that remains torsionally stable under that wind? The next statement shows that it is enough to have
a sufficiently large damping.

\begin{theorem}\label{SCstability2}
Assume that $S>0$, $0\le P<\lambda_1$. Assume \eqref{smallginfty} and that $g$ is even with respect to $y$.
There exists $\delta_0=\delta_0(g_\infty,S,P,\lambda_1)>0$ such that if $\delta>\delta_0$, then there exists $\eta>0$ such that
$$\lim_{t\to\infty}{\rm e}^{\eta t}\Big(\|u_t^T(t)\|_{L^2}^2+\|u^T(t)\|_{H^2_*}^2\Big)=0$$
for any solution $u$ of \eqref{enlmg}.
\end{theorem}

Theorem \ref{SCstability2} implies that $g$ makes the system \eqref{enlmg} torsionally stable if the damping is large enough. A natural question is then
to find out whether a full counterpart of Theorem \ref{SCstability} holds. More precisely, is it true that
under the assumptions of Theorem \ref{SCstability2} we have the ``squeezing property'' \eqref{squeezing}, provided that $\delta$ is sufficiently
large? In particular, if $g$ is periodic, is it true that there exists a unique periodic solution of \eqref{enlmg} whenever $\delta$
is large enough? We conjecture both these questions to have a positive answer but we leave them as open problems.\par
Finally, we show that the nonlinear term is responsible for the possible torsional instability.

\begin{theorem}\label{smalltorsion}
Assume that $\delta>0$, $S>0$, $0\le P<\lambda_1$, $g\in C^0(\mathbb{R}_+,L^2(\Omega))$ even with respect to $y$.
There exists $\chi=\chi(\delta,S,P,\lambda_1)>0$ such that if a solution $u$ of \eqref{enlmg} (written in the form \eqref{LT}) satisfies
\begin{equation}\label{chidelta}
\limsup_{t\to\infty}\|u_x(t)\|_{L^2}^2<\chi\, ,
\end{equation}
then its torsional component vanishes exponentially as $t\to\infty$; more precisely, there exists $\eta>0$ such that
$$\lim_{t\to\infty} {\rm e}^{\eta t}\left(\|u^T(t)\|_{H^2_*}^2+\|u^T_t(t)\|_{L^2}^2\right)=0.$$
\end{theorem}

Theorem \ref{smalltorsion} shows that, with no smallness nor
periodicity assumptions on $g$ and no request of large $\delta$, the possible culprit for the torsional instability of a given solution is a large
nonlinear term: from a physical point of view, this means that if the {\em stretching energy} of the solution is eventually
small, then the torsional component of the solution vanishes exponentially fast as $t\to\infty$. This result does not come unexpected: the nonlinearity
of the system is concentrated in the stretching term which means that ``small stretching implies small nonlinearity'' which, in turn, implies
``little instability''. The nonlinear term, which is a coupling term between the longitudinal and torsional movements, see \eqref{energydecomposition},
acts as a force able to transfer energy from one component to the other. Even if $g$ has no torsional part
(when $g$ is even with respect to $y$), it may happen that the solution $u$ displays a nonvanishing torsional part $u^T$ as $t\to\infty$.\par
Since it only refers to some particular solution $u$ of \eqref{enlmg}, Theorem \ref{smalltorsion} {\em does not} give a sufficient condition for $g$
to make \eqref{enlmg} torsionally stable according to Definition \ref{defstab}. Nevertheless, from Theorem \ref{smalltorsion} we deduce

\begin{corollary}
Assume that $\delta>0$, $S>0$, $0\le P<\lambda_1$, $g\in C^0(\mathbb{R}_+,L^2(\Omega))$ even with respect to $y$.
Let $\chi=\chi(\delta,S,P,\lambda_1)>0$ be as in Theorem \ref{smalltorsion}. If every solution $u$ of \eqref{enlmg} (written in the form \eqref{LT})
satisfies \eqref{chidelta}, then $g$ makes \eqref{enlmg} torsionally stable.
\end{corollary}

Clearly, a sufficient condition for \eqref{chidelta} to hold for any solution, is that $g$ is small; in this case we are back to Theorem \ref{SCstability}.

\begin{remark}\label{onlyt}
If the force $g$ does not depend on the space variable $\xi$, that is $g=g(t)$ as in most cases of a wind acting on the deck of a bridge, the same proofs
of Theorems \ref{SCstability2} and \ref{smalltorsion} show that the skew-symmetric (with $m$ even) longitudinal components also decay exponentially to zero.
\end{remark}

Overall, the results stated in the present section give some answers to question (b). We have seen that the stability of a longitudinal prevailing mode
is ensured provided that $g$ is sufficiently small and/or $\delta$ is sufficiently large. Moreover, the responsibility for the torsional instability is only
the stretching energy and not the bending energy.

\section{How to determine the prevailing mode: linear analysis}\label{linearanal}

In order to give an answer to question (a) we seek a criterion to predict which will be the prevailing mode of oscillation.
As already mentioned, the wind flow generates vortices that appear periodic in time and have the shape of \eqref{deMiranda}, where the frequency
and amplitude depend increasingly on the scalar velocity $W>0$.
Initially, the deck is still and the wind starts its transversal action on the deck. For some time, the oscillation of the deck will be small.
This suggests to neglect the nonlinear term in \eqref{enlmg} and to consider the linear problem
\begin{equation}\label{linearpb}
\left\{
\begin{array}{rl}
u_{tt}+\delta u_t + \Delta^2 u +Pu_{xx}= W^2\sin(\omega t)  &\textrm{in }\Omega\times(0,T)\\
u = u_{xx}= 0 &\textrm{on }\{0,\pi\}\times[-\ell,\ell]\\
u_{yy}+\sigma u_{xx} = u_{yyy}+(2-\sigma)u_{xxy}= 0 &\textrm{on }[0,\pi]\times\{-\ell,\ell\}\\
u(\xi,0) = u_t(\xi,0) = 0 &\textrm{in }\Omega\, .
\end{array}
\right.
\end{equation}

Arguing as in the proof of \cite[Theorem 7]{FerGazMor}, we deduce that both the torsional and the longitudinal
skew-symmetric components of the solution are zero. Therefore, we may write the solution of \eqref{linearpb} as
\[
u(\xi,t)=\sum_{k=1}^{\infty}S_k(t)w_k(\xi),
\]
where $w_k$ are the symmetric longitudinal eigenfunctions, see cases ($i$) and ($ii$) in Proposition \ref{spectrum} with $m$ odd.
Denote by $\lambda_k$ the eigenvalue of \eqref{eq:eigenvalueH2L2} associated to $w_k$. Let $\gamma_k$ be as in \eqref{gammak},
then the coefficients $S_k(t)$ satisfy the ODE
\begin{equation}\label{ODE-linear}
\left\{
\begin{array}{rcl}
\ddot{S}_k + \delta \dot{S}_k + (\lambda_k - P m^2) S_k &\!\!\!\! = \!\!\!\!& \gamma_k W^2\sin(\omega t) \quad\textrm{in } (0,\infty)\\
S_k(0) = \dot{S}_k(0) &\!\!\!\! = \!\!\!\!& 0
\end{array}\right.
\end{equation}

A standard computation shows that the explicit solutions of \eqref{ODE-linear} are given by
\[
\begin{array}{l}
S_k(t)=W^2\, \frac{\gamma_k}{(\lambda_k-Pm^2-\omega^2)^2+\delta^2\omega^2} \left\{  \omega e^{\frac{-\delta t}{2}}\left[ \delta \cos\left( \frac{\sqrt{4 (\lambda_k - P m^2) - \delta^2}}{2}  \, t\right) + \right. \right. \vspace{5pt}\\
\left.\left. \frac{\delta^2 - 2 (\lambda_k - P m^2 - \omega^2)}{ \sqrt{4 (\lambda_k - P m^2) - \delta^2} } \sin\left( \frac{\sqrt{4 (\lambda_k - P m^2) - \delta^2}}{2}  \, t\right)\right]  + (\lambda_k -  P m^2- \omega^2) \sin(\omega t) - \delta \omega \cos(\omega t)\right\}
\end{array}
\]

These functions $S_k$ are composed by a damped part (multiplying the negative exponential) and a linear combination of trigonometric functions.
In fact,
$$
\max_t\big|(\lambda_k -  P m^2- \omega^2) \sin(\omega t) - \delta \omega \cos(\omega t)\big|=\sqrt{(\lambda_k-P m^2-\omega^2)^2+\delta^2\omega^2}\, .
$$
Hence, the parameter measuring the amplitude of each of the $S_k$'s is
$$
\frac{\gamma_k}{\sqrt{(\lambda_k-Pm^2-\omega^2)^2+\delta^2\omega^2}}\, .
$$
But we also need to take into account the size of the $w_k$'s (recall that they are normalized in $L^2$, see \eqref{normalized}):
therefore, the amplitude of oscillation of each mode is
\begin{equation}\label{relevant}
A_k(\omega):=\frac{\gamma_k\, \|w_k\|_{L^\infty}}{\sqrt{(\lambda_k-Pm^2-\omega^2)^2+\delta^2\omega^2}}\, .
\end{equation}
It is readily seen that
$$\omega\mapsto A_k(\omega)\quad\mbox{attains its maximum at}\qquad\left\{\begin{array}{ll}
\omega=0 & \mbox{if }\delta^2\ge2(\lambda_k-Pm^2)\\
\omega^2=\lambda_k-Pm^2-\delta^2/2 & \mbox{if }\delta^2<2(\lambda_k-Pm^2).
\end{array}\right.$$

We notice that the eigenfunctions in the family $(i)$ of Proposition \ref{spectrum} with $m$ odd attain their maximum at $(\pi/2,\ell)$.
Then we numerically obtain the results of Table \ref{Linftynorm}.

\begin{table}[ht]
\begin{center}
\begin{tabular}{|c|c|c|c|c|c|c|c|c|c|}
\hline
eigenvalue & $\mu_{1,1}$ & $\mu_{3,1}$ & $\mu_{5,1}$ & $\mu_{7,1}$ & $\mu_{9,1}$ & $\mu_{11,1}$  & $\mu_{13,1}$ & $\mu_{15,1}$ & $\mu_{17,1}$\\
\hline
$\|w_k\|_{L^\infty}\approx$ & 2.764 & 14.37 & 30.92 & 51.23 & 74.73 & 101 & 129.9 & 161.1 & 194.6\\
\hline
\end{tabular}
\caption{Approximate value of the $L^\infty$-norm of some $L^2$-normalized eigenfunctions of \eqref{eq:eigenvalueH2L2}.}\label{Linftynorm}
\end{center}
\end{table}

From now on, we take the values of $\lambda_k$ from Table \ref{tableigen}, the values of $\gamma_k$ from Table \ref{tablegamma}, the values
of $\|w_k\|_{L^\infty}$ from Table \ref{Linftynorm}. Moreover, we fix $\delta=0.58$.\par
In Figure \ref{plotsAk} we represent the functions $A_1$, $A_3$, $A_5$, $A_7$, as defined in \eqref{relevant}, for $\omega\in(0,60)$ and for $P=0$.
\begin{figure}[!h]
\begin{center}
\includegraphics[height=26mm,width=39mm]{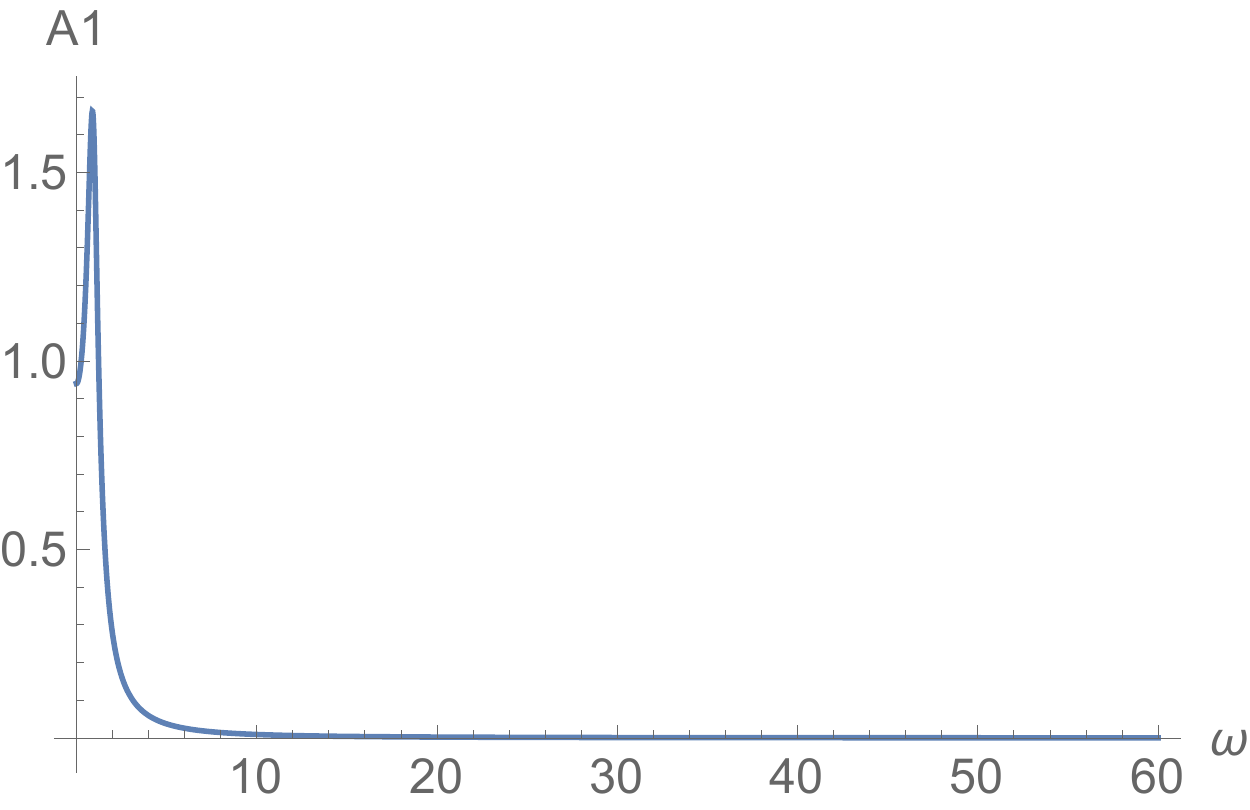}\ \includegraphics[height=26mm,width=39mm]{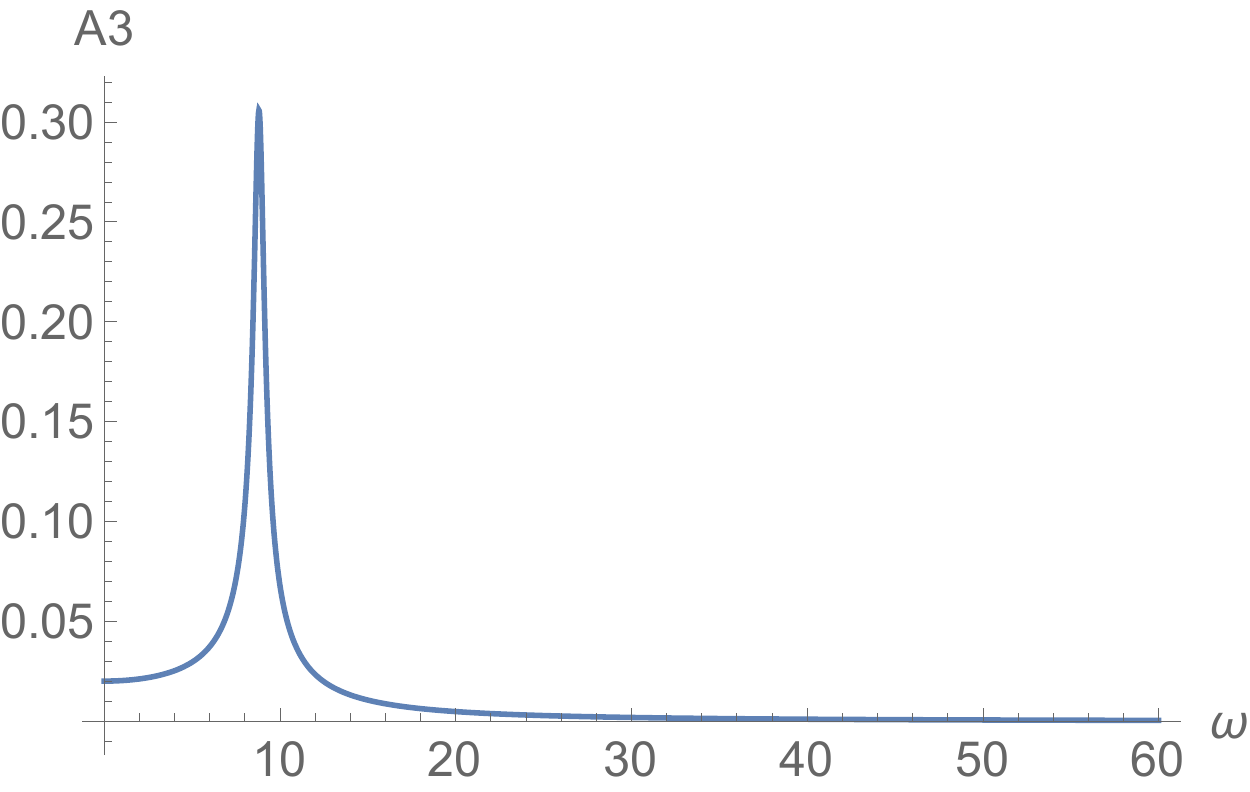}\
\includegraphics[height=26mm,width=39mm]{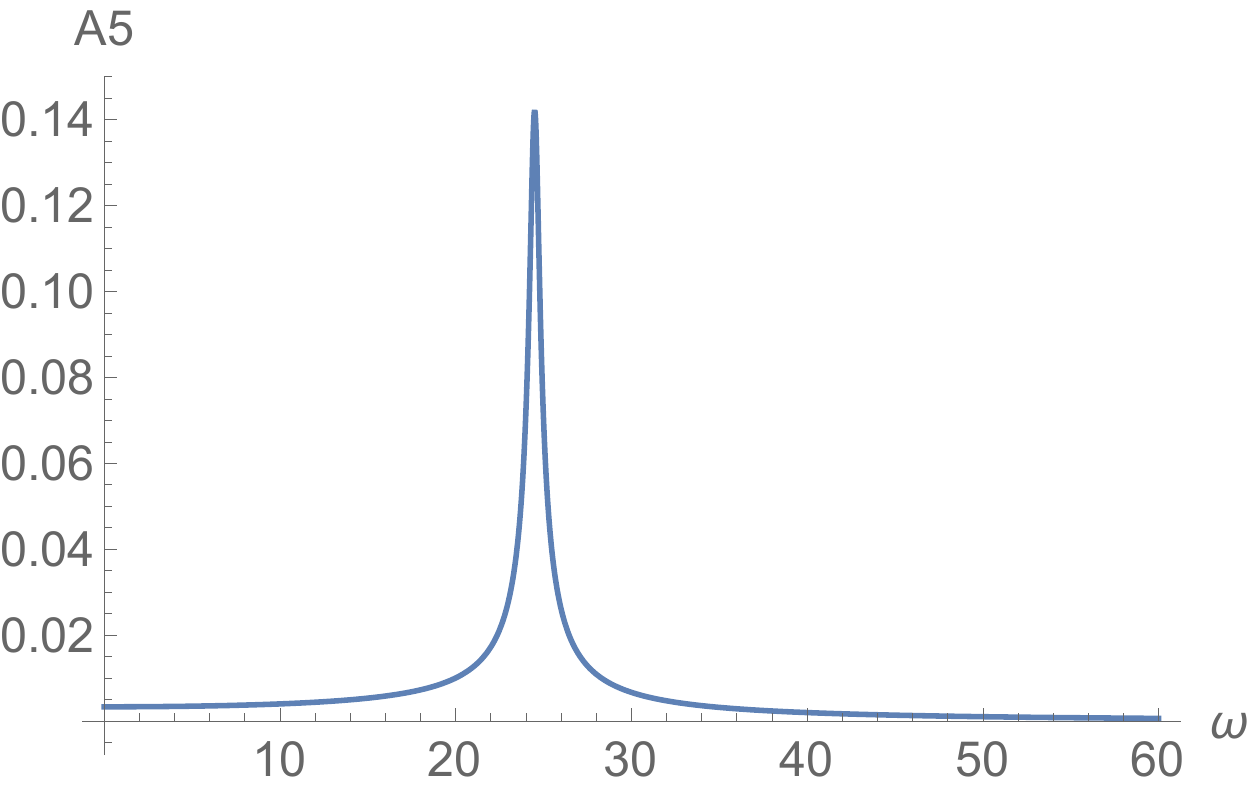}\ \includegraphics[height=26mm,width=39mm]{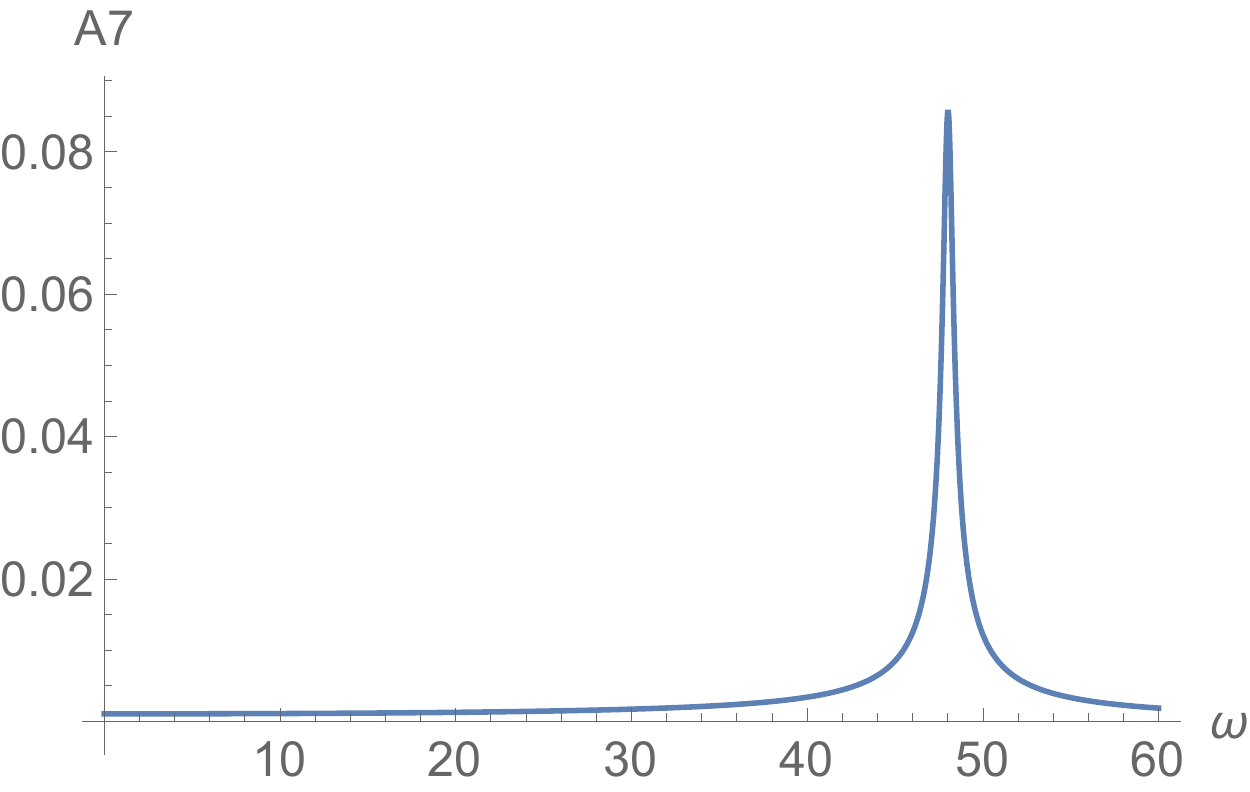}
\caption{Plots of the functions $\omega\mapsto A_k(\omega)$ in \eqref{relevant} for $k=1,3,5,7$.}\label{plotsAk}
\end{center}
\end{figure}
It turns out that these functions all have a steep spike close to their maximum but elsewhere they are fairly small, of several orders of magnitude
less. The height of the spikes is decreasing with respect to $k$ and the maximum is moving to the right (larger $\omega$).\par
We are now in position to give an answer to question (a). For a given $\omega>0$, the prevailing mode $w_k$ is the one maximizing $A_k(\omega)$.
For each $\omega>0$ we numerically determine which $k$ maximizes $A_k(\omega)$. We consider the values $P=0$ and $P=1/2$ and we obtain the results
summarized in Tables \ref{P0} and \ref{P12}, where $k_p$ is the prevailing mode so that
$$k_p=k_p(\omega)\mbox{ is such that }A_{k_p}(\omega)=\max_kA_k(\omega)\, .$$

\begin{table}[ht]
\begin{center}
{\small
\begin{tabular}{|c|c|c|c|c|c|c|c|c|}
\hline
$\omega\!\in$\!\!&\!\!$(0,5.39)$\!\!&\!$(5.39,17.48)$\!\!&\!\!$(17.48,37.17)$\!\!&\!\!$(37.17,64.64)$\!\!&\!\!$(64.64,100)$\!\!&\!\!$(100,143.1)$\!\!&\!\!$(143.1,194.2)$\!\! &\!\!$(194.2,253.1)$\!\!\\
\hline
$k_p$ & 1 & 3 & 5 & 7 & 9 & 11 & 13 & 15\\
\hline
\end{tabular}
}
\caption{Prevailing mode $k_p$ in terms of the frequency $\omega$ (for $P=0$).}\label{P0}
\end{center}
\end{table}

\begin{table}[ht]
\begin{center}
\begin{tabular}{|c|c|c|c|c|c|c|c|c|}
\hline
$\omega\!\in$\!\!&\!\!$(0,5.2)$\!\!&\!$(5.2,17.26)$\!\!&\!\!$(17.26,36.92)$\!\!&\!\!$(36.92,64.4)$\!\!&\!\!$(64.4,99.7)$\!\!&\!\!$(99.7,142.9)$\!\!&\!\!$(142.9,193.9)$\!\! &\!\!$(193.9,252.8)$\!\!\\
\hline
$k_p$ & 1 & 3 & 5 & 7 & 9 & 11 & 13 & 15\\
\hline
\end{tabular}
\caption{Prevailing mode $k_p$ in terms of the frequency $\omega$ (for $P=1/2$).}\label{P12}
\end{center}
\end{table}

It appears evident that the prestressing constant $P$ does not influence too much the prevailing mode, only a slight shift of the intervals.
In order to find the related wind velocity $W$, from Section \ref{quant} we recall that $\omega$ is proportional to $W$:
$$\omega=\frac{{\rm St}}{H}\, W\, .$$

\section{Numerical results}\label{numres}

For our numerical experiments, we first consider external forces $g=g(\xi,t)$ able to identify the
``prevailing mode'' (which should be longitudinal), then we investigate whether this longitudinal mode is stable with respect to the torsional modes.
According to Definition \ref{defstab}, in order to emphasize torsional instability we need to find a particular solution
of \eqref{enlmg} having a torsional component which does not vanish at infinity. So, we select the longitudinal mode candidate to become the prevailing mode, that is, one of
the ($L^2$-normalized) eigenfunctions in Proposition \ref{spectrum} $(i)$. Indeed, according to Table \ref{tableigen}, the least 17 longitudinal eigenvalues are all of the kind $\mu_{m,1}$ ($m=1,...,17$) that are associated to
this kind of eigenfunction. Let us denote by $L_m$ the associated ($L^2$-normalized) longitudinal eigenfunction:
$$L_m(x,y)=C_m\, \left[\big[\mu_{m,1}^{1/2}-(1-\sigma)m^2\big]\, \tfrac{\cosh\Big(y\sqrt{m^2+\mu_{m,1}^{1/2}}\Big)}{\cosh\Big(\ell\sqrt{m^2+\mu_{m,1}^{1/2}}\Big)}+
\big[\mu_{m,1}^{1/2}+(1-\sigma)m^2\big]\, \tfrac{\cosh\Big(y\sqrt{m^2-\mu_{m,1}^{1/2}}\Big)}{\cosh\Big(\ell\sqrt{m^2-\mu_{m,1}^{1/2}}\Big)}\right]\sin(mx)\, ,$$
where $C_m$ is a normalization constant, see \eqref{normalized}. Then we consider the external force in the particular form
\begin{equation}\label{fpart}
g_m(\xi,t)=A\, b\, L_m(\xi)\sn(bt,k)\dn(bt,k)\, ,\quad
b=\sqrt{\mu_{m,1}+\frac{Sm^4A^2}{\delta^2}}\, ,\quad k=\sqrt{\frac{Sm^4A^2}{2(\mu_{m,1}\delta^2+Sm^4A^2)}}\, ,
\end{equation}
where $A>0$ has to be fixed while $\sn$ and $\dn$ are the Jacobi elliptic functions: the function $\sn(bt,k)\dn(bt,k)$ is a modification of the trigonometric
sine which becomes particularly useful when dealing with Duffing equations, see \cite{abramst}. This choice of $g_m$ only slightly modifies the form given in \eqref{newf}.
Then we prove

\begin{proposition}\label{explicit}
Assume that $P=0$ and that $g(\xi,t)=g_m(\xi,t)$ for some integer $m$, as defined in \eqref{fpart}. Then the function
$$U^p(\xi,t)=-\frac{A}{\delta}\, \cn(bt,k)\, L_m(\xi)$$
is a periodic solution of \eqref{enlmg}. Moreover, if $A>0$ is sufficiently small, then $U^p$ is the unique periodic solution
of \eqref{enlmg}: in such case, the prevailing mode $U^p$ is torsionally stable.
\end{proposition}
\begin{proof} Take $b$ and $k$ as in \eqref{fpart} and let $a=-A/\delta$. From \cite{burg} we know that the function $z(t)=a\, \cn(bt,k)$ solves the problem
$$
\ddot{z}(t)+\mu_{m,1}\, z(t)+Sm^4\, z(t)^3=0\, ,\qquad z(0)=a\, ,\qquad \dot{z}(0)=0\, .
$$
Since $\frac{d}{dt}\cn(bt,k)=-b\sn(bt,k)\dn(bt,k)$, the function $z$ also solves
\begin{equation}\label{duff}
\ddot{z}(t)+\delta\, \dot{z}(t)+\mu_{m,1}\, z(t)+Sm^4\, z(t)^3=\delta\, \dot{z}(t)=A\, b\, \sn(bt,k)\dn(bt,k)\, ,\qquad z(0)=a\, ,\qquad \dot{z}(0)=0\, .
\end{equation}
Therefore, $z$ is a periodic solution of \eqref{duff} and, in turn, the function $U^p(\xi,t)=z(t)\, L_m(\xi)$ is a periodic solution
of \eqref{enlmg}.\par
Since $k^2<1/2$ in view of \eqref{fpart} and since $\dn(bt,k)^2+k^2\, \sn(bt,k)^2\equiv1$, by the properties of the Jacobi functions (see \cite{abramst})
we know that
$$
\max_{t>0}\big|\sn(bt,k)\dn(bt,k)\big|=\sqrt{1- k^2}\, .
$$
Hence, recalling that $L_m$ is $L^2$-normalized, we have
$$\sup_{t>0}\int_\Omega g_m(\xi,t)^2\, d\xi=A^2\, b^2\, \max_{t>0}\, \big|\sn(bt,k)\dn(bt,k)\big|^2=\frac{2\mu_{m,1}\delta^2+Sm^4A^2}{2\delta^2}\,
A^2 \, .
$$
Therefore, if $A$ is sufficiently small, then the assumptions of Theorem \ref{exuniq} are
fulfilled and the periodic solution of \eqref{enlmg} is unique and torsionally stable.\end{proof}

From \cite{burg} we also know that the period $\tau$ of the forcing term (and of the solution) is given by the elliptic integral
\begin{equation}\label{periodtau}
\tau=\frac{4}{b}\int_0^{\pi/2}\frac{d\varphi}{\sqrt{1-k^2\sin^2\varphi}}\, .
\end{equation}
Note that for $0<k^2<1/2$ we have $\frac\pi2\approx1.57<\int_0^{\pi/2}\frac{d\varphi}{\sqrt{1-k^2\sin^2\varphi}}<1.86$ so that $\tau$ has small variations.\par
For all our numerical experiments we take $P=0$ and we wish two emphasize two kinds of behaviors of the solution: existence of multiple periodic solutions and
torsional instability.\par\smallskip
{\bf Existence of multiple periodic solutions.} We select one longitudinal eigenfunction $L_m$ associated to
some eigenvalue $\mu_{m,1}$ (for $m=1,...,17$), see Table \ref{tableigen}, and one torsional eigenfunction $T_n$ associated to some eigenvalue
$\nu_{n,2}$ (for $n=1,2,3$), see again Table \ref{tableigen}. We take the external force $g=g_m$ to be as in \eqref{fpart} and initial conditions
\eq{initialc} such as
\begin{equation}\label{icnum}
u(\xi,0) =\alpha L_m(\xi)+\beta T_n(\xi), \qquad u_t(\xi,0) =0\qquad\textrm{in }\Omega
\end{equation}
for some $\alpha,\beta\in\mathbb{R}$. The uniqueness statement of Theorem \ref{exuniq} then shows that the solution $u=u(\xi,t)$ of
\eqref{enlmg} satisfying the initial conditions \eq{icnum} necessarily has the form
$$u(\xi,t)=\phi(t)L_m(\xi)+\psi(t)T_n(\xi)$$
for some $C^2$-functions $\phi$ and $\psi$ satisfying the following nonlinear system of ODE's:
\begin{equation}\label{system}
\left\{\begin{array}{l}
\ddot{\phi}(t)+\delta\dot{\phi}(t)+\mu_{m,1}\phi(t)+Sm^2[m^2\phi(t)^2+n^2\psi(t)^2]\phi(t)=A\, b\, \sn(bt,k)\dn(bt,k)\\
\ddot{\psi}(t)+\delta\dot{\psi}(t)+\nu_{n,2}\psi(t)+Sn^2[m^2\phi(t)^2+n^2\psi(t)^2]\psi(t)=0
\end{array}\right.
\end{equation}
while the initial conditions \eqref{icnum} become
\begin{equation}\label{icsyst}
\phi(0)=\alpha\, ,\quad\psi(0)=\beta\, ,\quad\dot{\phi}(0)=\dot{\psi}(0)=0\, .
\end{equation}
We notice that if $\beta=0$ then the solution of \eqref{system}-\eqref{icsyst} satisfies $\psi\equiv0$, which
means that there is no torsional component at all. When $A$ is small, Proposition \ref{explicit} states that $U^p$ is the unique periodic
solution of \eqref{enlmg} and that $g_m$ in \eqref{fpart} makes \eqref{enlmg} torsionally stable. Our strategy then consists in taking
$\beta>0$ and studying the behavior of the solution of \eqref{system}-\eqref{icsyst} when $A$ becomes large, aiming to emphasize
multiplicity of periodic solutions.\par
If we take $\alpha=-A/\delta$ and $\beta=0$, then the solution of \eqref{system}-\eqref{icsyst} is given by
$$\phi(t)=-\frac{A}{\delta}\, \cn(bt,k)\, ,\qquad\psi(t)\equiv0$$
while
\begin{equation}\label{UP}
U^p(\xi,t)=-\frac{A}{\delta}\, \cn(bt,k)\, L_m(\xi)
\end{equation}
is a periodic solution of \eqref{enlmg}, see Proposition \ref{explicit}. If it was the only periodic solution, then Theorem \ref{SCstability}
would ensure that $(\phi-U^p,\psi)(t)\to0$ uniformly as $t\to\infty$ for any solution of \eqref{system}. Hence, in order to display multiple periodic solutions
of \eqref{enlmg} it suffices to exhibit a solution of \eqref{system} that does not satisfy this condition.\par
In Figure \ref{threeplots} we display the graphs of $U^p$ and of the solution $(\phi,\psi)$ of \eqref{system}-\eqref{icsyst} with
\begin{equation}\label{param1}
m=2\, ,\ n=1\, ,\ \delta=0.58\, ,\ S=279\, ,\ A=0.2645\, ,\ b\mbox{ as in \eqref{fpart}}\, ,\ \alpha=0\, ,\ \beta=0.01\, .
\end{equation}
Note that the ``frequency'' of $\psi$ is considerably larger than the frequency of $\phi$. This is due to the fact that $\nu_{1,2}\gg\mu_{2,1}$,
see Table \ref{tableigen}.
\begin{figure}[!h]
\begin{center}
\includegraphics[height=34mm,width=78mm]{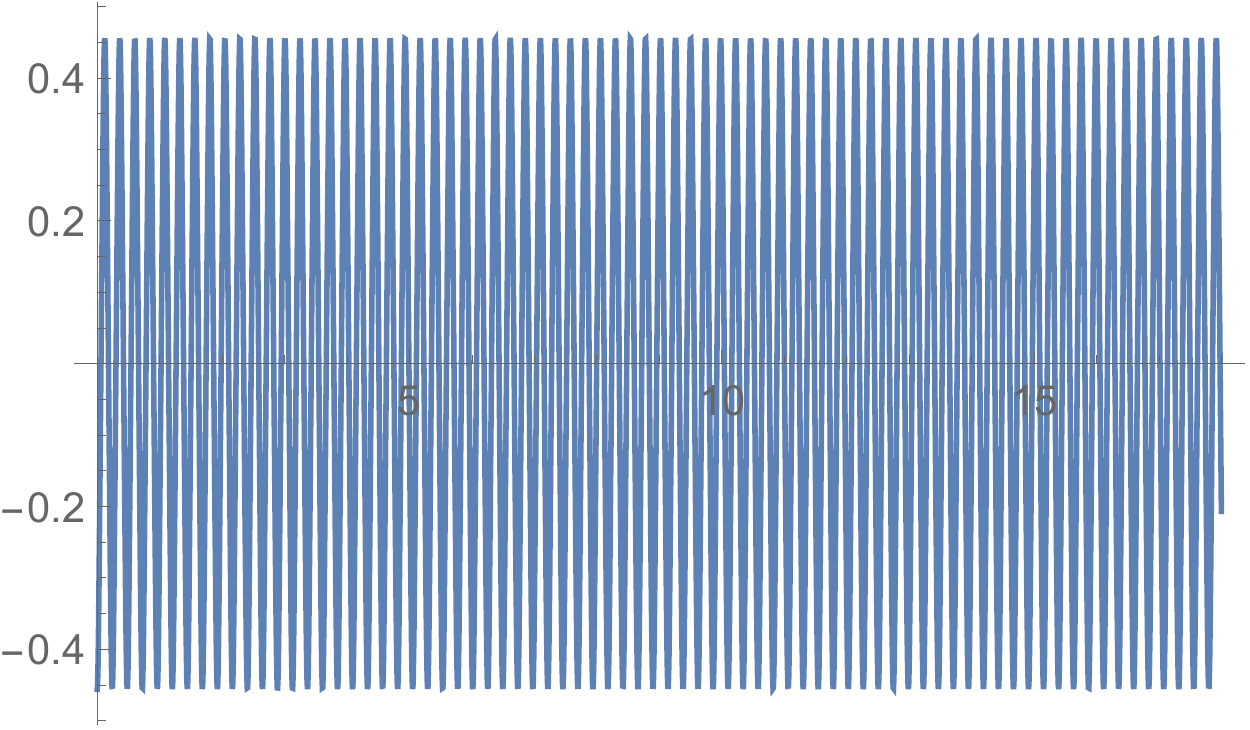}\ \includegraphics[height=34mm,width=78mm]{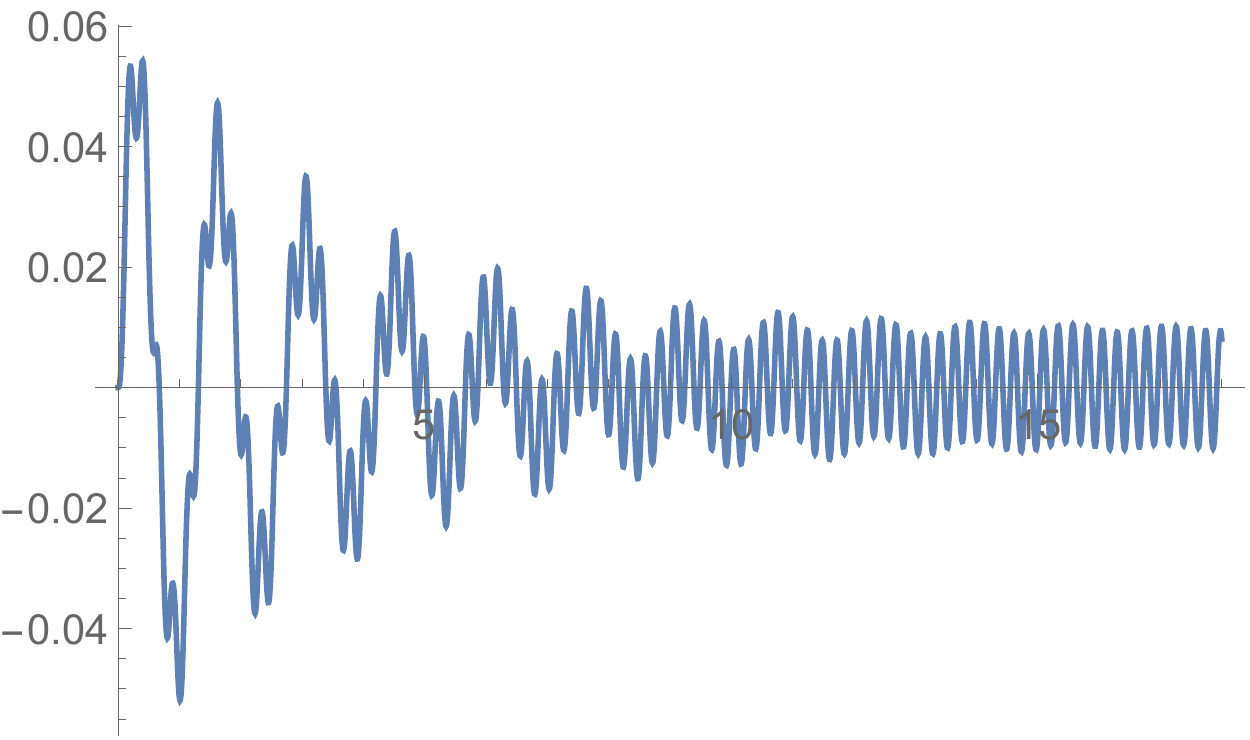}\\
\includegraphics[height=34mm,width=78mm]{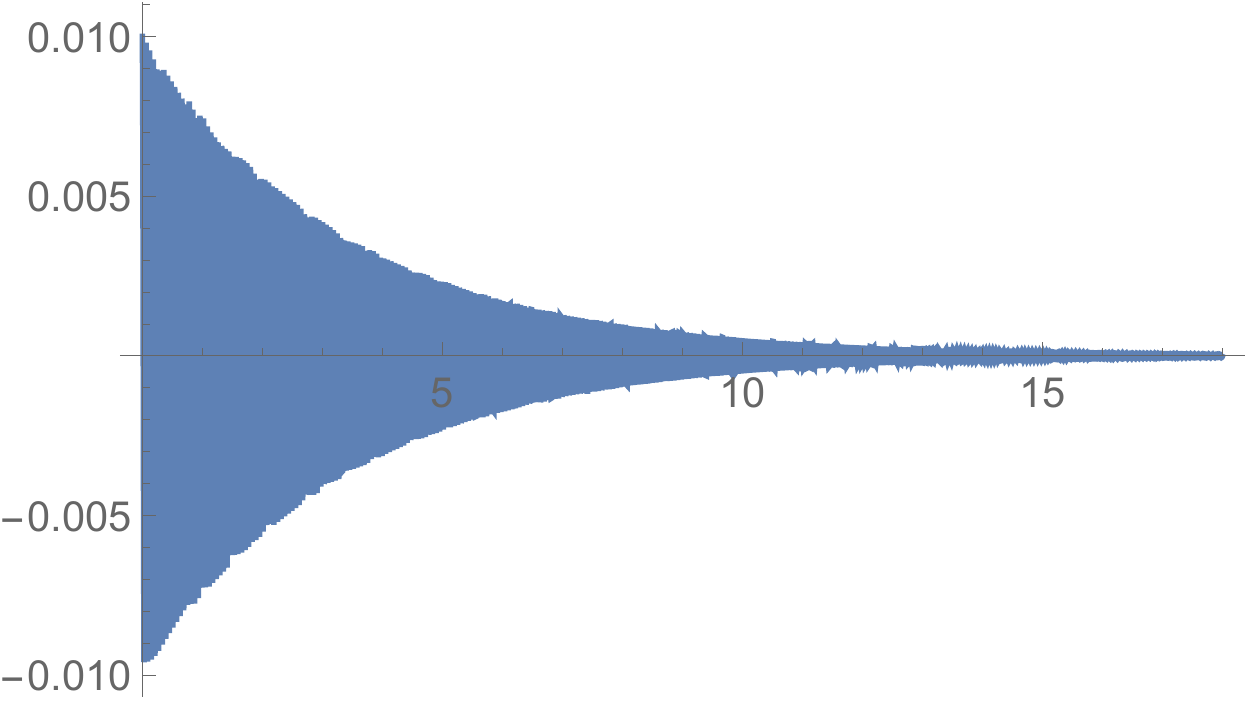}\ \includegraphics[height=34mm,width=78mm]{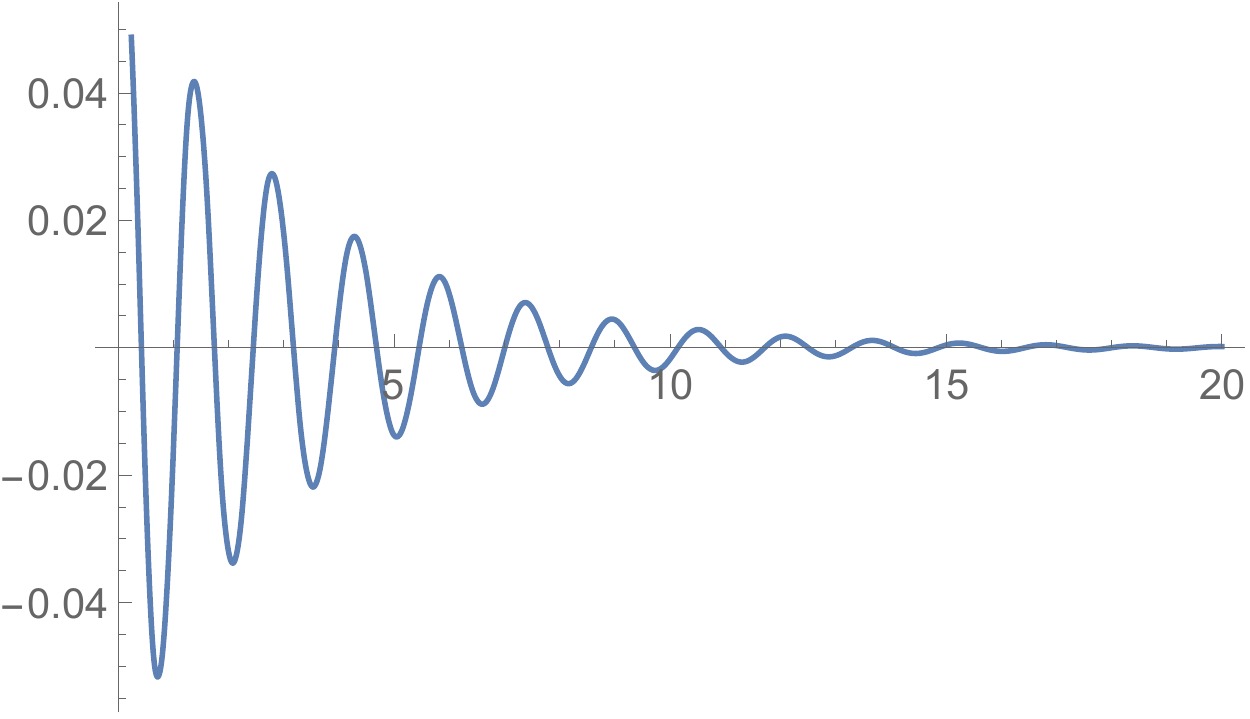}
\caption{Plot of the functions $U^p$ in \eqref{UP} (top left), $\phi$ (top right) and $\psi$ (bottom left) solving \eqref{system}-\eqref{icsyst},
and $t\mapsto\phi(t)-\phi(t-\tau)$ (bottom right), when \eqref{param1} holds.}\label{threeplots}
\end{center}
\end{figure}

It appears also quite visible that $\phi-U^p$ {\em does not} converge uniformly to $0$, see in particular the amplitudes on the vertical axis.
In the bottom right picture of Figure \ref{threeplots} we plot the graph of $\phi(t)-\phi(t-\tau)$, where $\tau$ is as in \eqref{periodtau}, that is, the period
of the force $g_m$ in \eqref{fpart} and in \eqref{system}. This plot seems to say that $\phi$ is converging to a periodic regime and this would
prove that the attractor of \eqref{enlmg} consists of at least two periodic solutions. Moreover, since $\psi\to0$ uniformly (bottom left plot in Figure
\ref{threeplots}), this would prove that \emph{``multiplicity of periodic solutions and torsional instability are not equivalent facts''}.

Finally, it is worth emphasizing that, by perturbing slightly the initial data $\alpha=-A/\delta$ and $\beta=0$ in \eqref{icsyst}, it was quite
evident that the periodic solution $U^p$ was unstable, the $\phi$-component always behaved as in the top right picture of Figure \ref{threeplots}
for large $t$.

\begin{figure}[!h]
\begin{center}
\includegraphics[height=34mm,width=78mm]{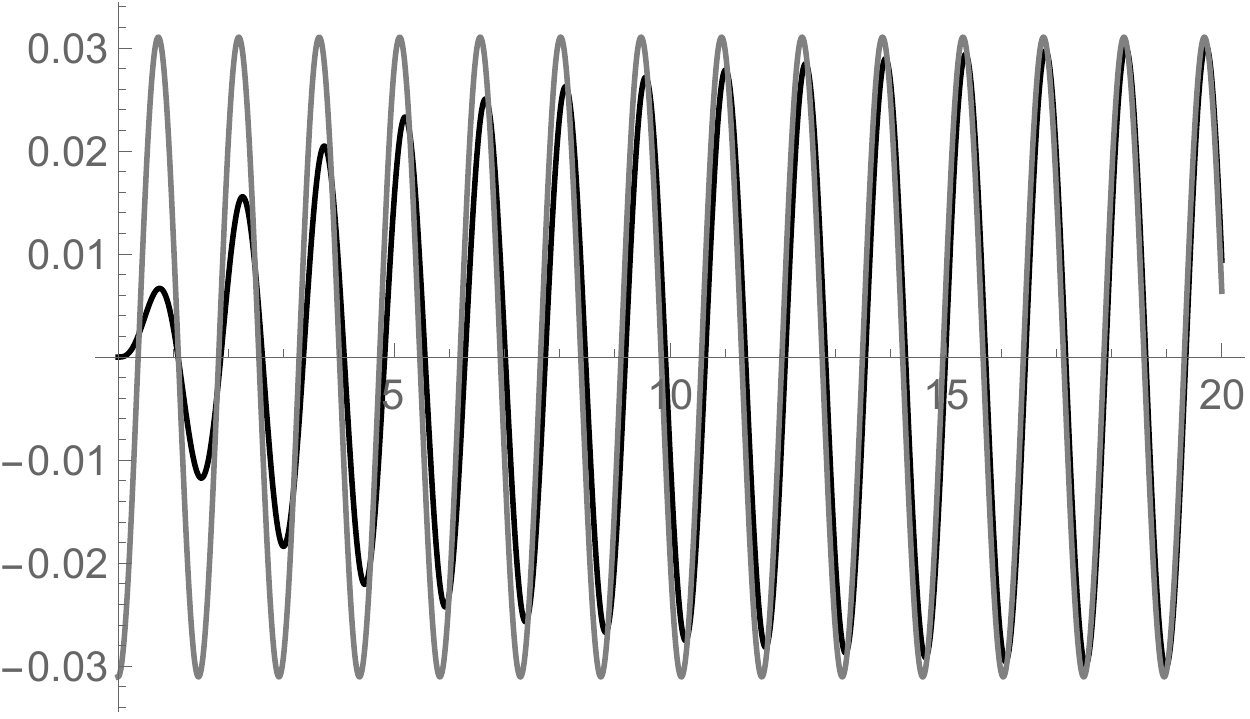}
\caption{The functions $U^p$ in \eqref{UP} (gray) and $\phi$ (black) solving \eqref{system}-\eqref{icsyst} when \eqref{param2} holds.}\label{doublegraph}
\end{center}
\end{figure}

We then diminished the amplitude $A$, modified $b$ according to \eqref{fpart}, and maintained all the other parameters as in \eqref{param1}; we took
\begin{equation}\label{param2}
m=2\, ,\ n=1\, ,\ \delta=0.58\, ,\ S=279\, ,\ A=0.018\, ,\ b\mbox{ as in \eqref{fpart}}\, ,\ \alpha=0\, ,\ \beta=0.01\, .
\end{equation}
In Figure \ref{doublegraph} we plot both the graphs of $U^p$ in \eqref{UP} (gray) and $\phi$ solving \eqref{system}-\eqref{icsyst}
(black). It appears that now $\phi$ approaches very quickly $U^p$. This probably means that the amplitude $A$ is sufficiently small so that
Theorem \ref{SCstability} applies and $U^p$ is the unique (and hence, stable) periodic solution. Note that the frequency is considerably
smaller than in the plots of Figure \ref{threeplots}.\par
We performed several other experiments for different couples of integers $(m,n)$, thereby changing the modes involved in the stability analysis, and we
always found qualitatively similar results: two periodic solutions for large $A$ and a (probably) unique periodic solution for small $A$.\par\smallskip
{\bf Torsional instability.} We were not able to detect any torsional instability for \eqref{system}, even by taking very large $A$. The reason
seems to be that taking $b$ as in \eqref{fpart} leaves too little freedom: the frequency is directly related to the amplitude. And large
frequencies are difficult to handle numerically due to large values of the derivatives of the solutions. Therefore, we considered the more
standard problem
\begin{equation}\label{system2}
\left\{\begin{array}{l}
\ddot{\phi}(t)+\delta\dot{\phi}(t)+\mu_{m,1}\phi(t)+Sm^2[m^2\phi(t)^2+n^2\psi(t)^2]\phi(t)=A\, \sin(\omega t)\\
\ddot{\psi}(t)+\delta\dot{\psi}(t)+\nu_{n,2}\psi(t)+Sn^2[m^2\phi(t)^2+n^2\psi(t)^2]\psi(t)=0
\end{array}\right.
\end{equation}
for some mutually independent values of the amplitude $A$ and of the frequency $\omega$. In the left plot of Figure \ref{instab} we depict the graph of
the $\psi$-component of the solution $(\phi,\psi)$ of \eqref{system2} and \eqref{icsyst} with the following choice of the parameters
\begin{equation}\label{param3}
m=2\, ,\ n=2\, ,\ \delta=0.4\, ,\ S=250\, ,\ A=62500\, ,\ \omega=275\, ,\ \alpha=0\, ,\ \beta=0.01\, .
\end{equation}
One clearly sees that $\psi(t)\not\to0$ as $t\to\infty$, which means torsionally instability. We performed several other experiments by considering
different couples $(m,n)$ and obtained qualitatively the same graph with $\psi$ growing up in some disordered way. We plotted
the graphs of $s\mapsto\psi(t)-\psi(t-2k\pi/\omega)$ (for some $k$) that also displayed a fairly disordered behavior, showing that no periodicity seems
to appear. If confirmed, this would show that the $\omega$-limit set of \eqref{enlmg} does not only contain periodic solutions for large $g$.

\begin{figure}[!h]
\begin{center}
\includegraphics[height=34mm,width=78mm]{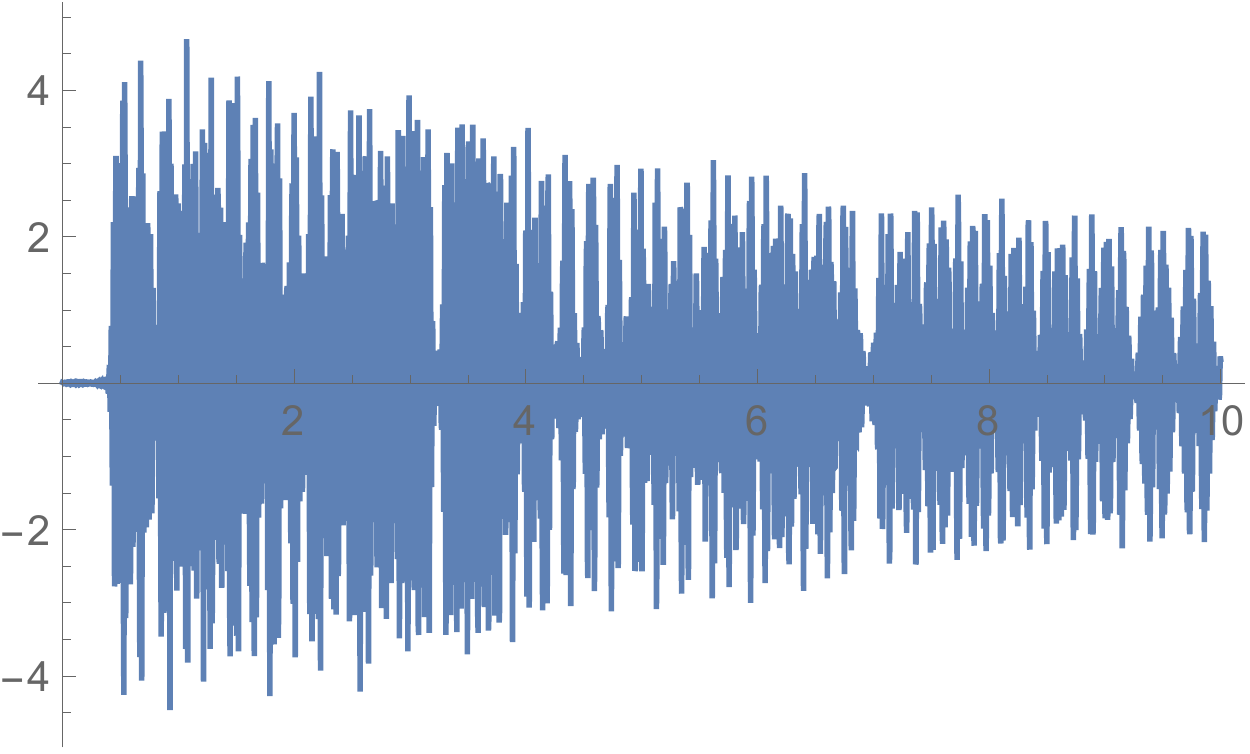}\ \includegraphics[height=34mm,width=78mm]{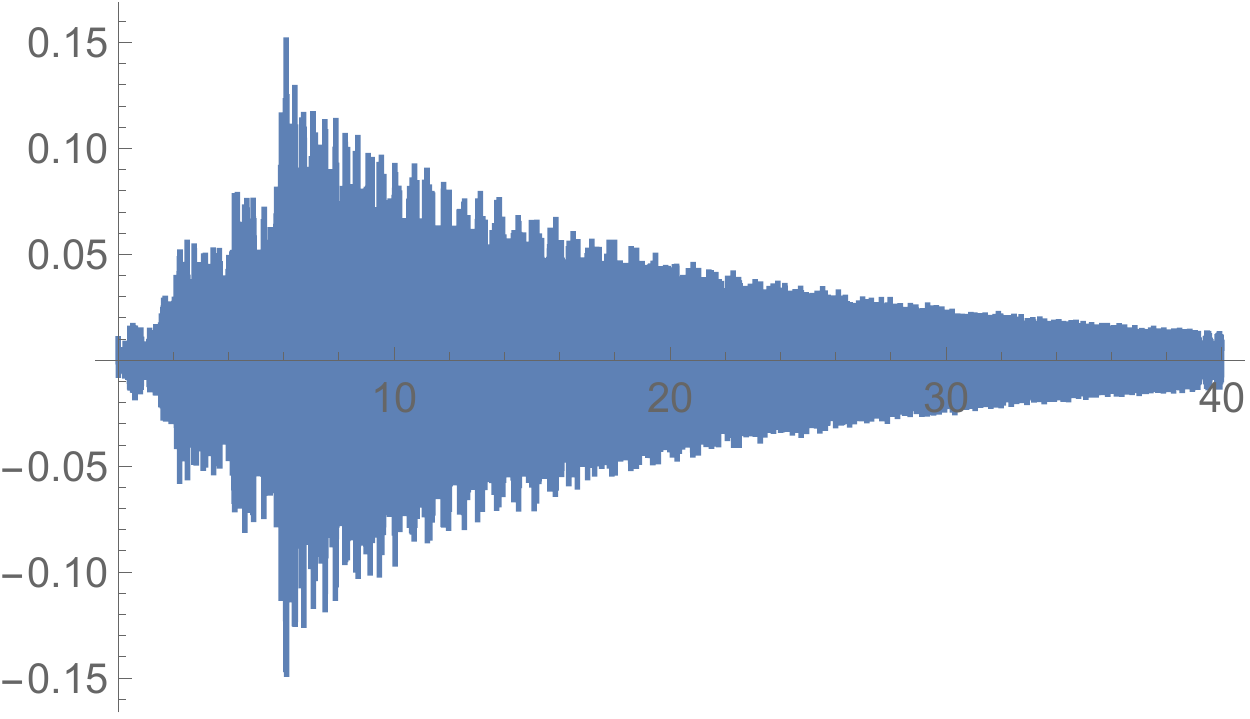}
\caption{Left: plot of the $\psi$-component of the solution of \eqref{system}-\eqref{icsyst} when \eqref{param3} holds.
Right: local torsional instability and eventual stability obtained when \eqref{param4} holds.}\label{instab}
\end{center}
\end{figure}

{\bf Local torsional instability and eventual stability.} The classification of Definition \ref{defstab} is mathematically exhaustive since a force $g$
makes the system either torsionally stable or torsionally unstable. However, some forces making the system stable may still be dangerous from a physical
(engineering) point of view and we need to focus our attention to a particular class of solutions.

\begin{definition}[Local torsional instability and eventual stability]\label{defstab2}
We say a solution $u$ of \eqref{enlmg} is locally torsionally unstable and eventually stable if
$$\lim_{t\to\infty}(\|u_t^T(t)\|_{L^2}+\|u^T(t)\|_{H^2_*})=0$$
and if one of its torsional Fourier coefficients, say $h_n=h_n(t)$, satisfies
$$
\dot{h}_n(0)=0\, ,\quad|h_n(t)|>10\, |h_n(0)|\ge0.1\mbox{ for some }t>0\, .$$
\end{definition}

This means that, although all the torsional components of the solution tend asymptotically to $0$, the amplitude of one torsional component has grown of
at least one order of magnitude at some time $t$ compared to the initial datum which was bounded away from zero: moreover, this is not due to the initial
kinetic datum $\dot{h}_n(0)$ since it is set to zero.\par
In the right picture of Figure \ref{instab} one sees an example of this situation: we depict there the graph of the $\psi$-component of the solution of
\eqref{system2} and \eqref{icsyst} with the following values of the parameters
\begin{equation}\label{param4}
m=4\, ,\ n=2\, ,\ \delta=0.12\, ,\ S=258\, ,\ A=6400\, ,\ \omega=160.8\, ,\ \alpha=0\, ,\ \beta=0.01\, .
\end{equation}
It appears that $\psi(t)$ grows up until about $0.15$, that is, 15 times as much as its initial value. Then it tends to vanish as $t\to\infty$.

\section{Energy estimates}\label{enest}

In this section we use the family of energies
$$E_\alpha(t):=\frac12 \|u_t(t)\|_{L^2}^2 +\frac12 \|u(t)\|_{H^2_*}^2 -\frac{P}{2}\|u_x(t)\|_{L^2}^2+\frac{S}{4}\|u_x(t)\|_{L^2}^4+\alpha
\into u\xit u_t\xit \, d\xi,$$
where $\alpha>0$, and we derive bounds for $E_\alpha$. The aim is to obtain bounds for the solutions of \eqref{enlmg} from the energy bounds.
Before starting, let us rigorously justify once forever the computations that follow. The regularity of weak solutions {\em does not allow}
to take $v=u_t$ in \eqref{weakform}. Therefore, we need to justify the differentiation of the energies $E_\alpha$, a computation that we use throughout the paper.
In this respect, let us recall a general result, see \cite[Lemma 4.1]{temam}.

\begin{lemma}\label{justification}
Let $(V,H,V')$ be a Hilbert triple. Let $a$ be a coercive bilinear continuous form on $V$, associated with the continuous isomorphism $A$ from $V$ to $V'$ such that
$a(u,v)=\langle Au,v\rangle$ for all $u,v\in V$. If $w$ is such that
$$w\in L^2(0,T;V)\, ,\quad w_t\in L^2(0,T;H)\, ,\quad w_{tt}+Aw\in L^2(0,T;H)\, ,$$
then, after modification on a set of measure zero, $w\in C^0([0,T],V)$, $w_t\in C^0([0,T],H)$ and, in the sense of distributions on $(0,T)$,
$$\langle w_{tt}+Aw,w_t\rangle=\frac12 \frac{d}{dt}\big(\|w_t\|_{L^2}^2+a(w,w)\big)\, .$$
\end{lemma}

We may now derive some energy bounds in terms of $g_\infty$, see \eqref{smallginfty}.

\begin{lemma}\label{energybound-PDE}
Assuming that $0\le P<\lambda_1$ and that $u$ is a solution of \eqref{enlmg}, we have
\begin{enumerate}[(a)]
\item for $\delta^2\le 4(\lambda_1-P)$,
\begin{equation}
\label{limtimeestimateforEdelta}
E_{\delta/2}(\infty):=\limsup_{t\to\infty}E_{\delta/2}(t)\le \frac{2\, g_{\infty}^2}{\delta^2}\, ,\end{equation}
\item for $\delta^2\ge4(\lambda_1-P)$,
$$
E_\mu(\infty):=\limsup_{t\to\infty}E_\mu(t)\le\frac{g_{\infty}^2}{2(\lambda_1-P)}\, ,
$$
where $\mu:=\frac\delta2-\frac12\sqrt{\delta^2-4(\lambda_1-P)}$.
\end{enumerate}
\end{lemma}
\begin{proof} Take any $\alpha\in(0,\frac23 \delta)$. From the definition of $E_\alpha$ and, by using Lemma \ref{justification} and \eqref{enlmg}, we infer that
$$\begin{array}{rl}
\dot E_\alpha(t)+ \alpha  E_\alpha(t) = &\!\!\!\!\! \displaystyle \left(\frac{3\alpha}2 -\delta\right)\|u_t(t)\|_{L^2}^2 -\frac\alpha{2}\|u(t)\|_{H^2_*}^2
+\frac{\alpha P}{2}\|u_x(t)\|_{L^2}^2-\frac{3S\alpha}{4}\|u_x(t)\|_{L^2}^4 \\
\ &\!\!\!\!\! \displaystyle+ \alpha(\alpha-\delta)\into u\xit u_t\xit \, d\xi + \into g\xit \big(u_t\xit +\alpha u\xit \big)\, d\xi\, .
\end{array}$$

Hence, by using \eqref{embedding} and the Young inequality, we obtain

\begin{empheq}{align}\label{dif-ineq-energy}
\begin{array}{rl}
\dot E_\alpha(t)+ \alpha  E_\alpha(t) \le & \displaystyle\left(\frac{3\alpha}2-\delta+\gamma\right)\|u_t(t)\|_{L^2}^2 -\frac\alpha{2\lambda_1}(\lambda_1-P-2\alpha\gamma)\|u(t)\|_{H^2_*}^2 \\
\ & \displaystyle + \alpha(\alpha-\delta+2\gamma)\into u\xit u_t\xit \, d\xi\, + \frac{1}{4\gamma}\|g(t)\|_{L^2}^2,
\end{array}
\end{empheq}
for every $\gamma>0$. To get a global estimate, we seek $\gamma>0$ such that
\begin{center}
(i) $\frac32\alpha-\delta+\gamma\le 0$,\qquad(ii) $\lambda_1-P\ge 2\alpha\gamma$,\qquad(iii) $\alpha-\delta+2\gamma = 0$.
\end{center}
These three conditions are satisfied if we choose
$$
\begin{array}{lll}
\displaystyle\alpha=\frac\delta2\, , \quad & \displaystyle \gamma=\frac\delta4\, , \qquad\qquad & \mbox{if }\delta^2\le4(\lambda_1-P)\\
\displaystyle\alpha=\mu\, , \quad & \displaystyle\gamma=\frac{\delta+\sqrt{\delta^2-4(\lambda_1-P)}}4\, ,\qquad\qquad & \mbox{if }\delta^2\ge4(\lambda_1-P)\, .
\end{array}
$$

Then, by using (i)-(ii)-(iii) we see that \eqref{dif-ineq-energy} entails
$$
\dot E_\alpha(t)+ \alpha  E_\alpha(t)\le \frac{1}{4\gamma}\|g(t)\|_{L^2}^2,
$$
and this implies, for all $t_0>0$, that
\begin{equation*}
E_\alpha(t)\le\,  {\rm e}^{-\alpha(t-t_0)} E_\alpha(t_0) +  \frac{\left(1-{\rm e}^{-\alpha(t-t_0)}\right)}{4\alpha\gamma} \sup_{t\ge t_0}\|g(t)\|_{L^2}^2.
\end{equation*}
By letting $t\to\infty$, we deduce that
\begin{equation}\label{timeestimateforEalpha-limsup}
E_\alpha(\infty):=\limsup_{t\to\infty}E_\alpha(t)\le\frac{g_{\infty}^2}{4\alpha\gamma}\, .
\end{equation}
The conclusions follow from \eqref{timeestimateforEalpha-limsup} and the respective choices of $\alpha$ and $\gamma$ according to the size of $\delta$. \end{proof}

Next we show that a bound on $E_\alpha(t)$ gives asymptotic bounds on all the norms of the solution. We start with $L^2$-bounds on $u$ and $u_x$ which are uniform in time. As it will become clear from the proofs, we can assume that $E_\alpha(\infty)\ge 0$.

\begin{lemma}[$L^2$-bound on $u$]\label{lem:L2bound}
Assume that $0\le P<\lambda_1$, that $\limsup_{t\to\infty}\|g(t)\|_{L^2}<\infty$ and that $u$ is a solution of \eqref{enlmg}.
Let $\alpha$ and $E_\alpha(\infty)$ be as in Lemma $\ref{energybound-PDE}$, then
\begin{equation}\label{eq:L2bound}
\limsup_{t\to \infty}\|u(t)\|^2_{L^2}\le\frac{4E_\alpha(\infty)}{\sqrt{(\lambda_1-P)^2+4SE_\alpha(\infty)}+(\lambda_1-P)}=:\Psi\, .
\end{equation}
\end{lemma}
\begin{proof} Let $\alpha$ be as in Lemma $\ref{energybound-PDE}$ and observe that
$$E_\alpha(t)=\frac\alpha2 \frac d{dt}\|u(t)\|^2_{L^2}+ \frac12 \|u_t(t)\|^2_{L^2}+\frac12 \|u(t)\|_{H^2_*}^2-\frac{P}{2}\|u_x(t)\|_{L^2}^2
+\frac{S}{4}\|u_x(t)\|_{L^2}^4\, .$$
From Lemma \ref{energybound-PDE} we know that there exist $C,t_0>0$ such that $E_\alpha(t)\le C$ for all $t\ge t_0$. Then, setting $\Upsilon(t):=\frac12\|u(t)\|^2_{L^2}$, the previous inequality and \eqref{embedding} imply that
\begin{equation}\label{riccati}
\alpha\dot\Upsilon (t)+ (\lambda_1-P) \Upsilon(t) + S \Upsilon(t)^2\le C\qquad\forall t\ge t_0\, .
\end{equation}

Two cases may occur. If there exists $\bar t\ge t_0$ such that
\begin{equation}\label{bart}
\Upsilon (\bar t) \le \frac{\sqrt{(\lambda_1-P)^2+4SC}-(\lambda_1-P)}{2S}=:\overline\Upsilon\, ,
\end{equation}
then from \eqref{riccati} we see that, necessarily, $\Upsilon (t)\le\overline\Upsilon$ for all $t\ge \bar t$ since $\dot\Upsilon (t)<0$ whenever
$\Upsilon(t)>\overline\Upsilon$. If there exist no $\bar t\ge t_0$ such that \eqref{bart} holds, then $\dot\Upsilon(t)<0$ for all $t\ge t_0$ and
$\Upsilon (t)$ has a limit at infinity, necessarily $\overline\Upsilon$. Therefore, in any case we have that
$$\limsup_{t\to\infty}\Upsilon(t)\le\overline\Upsilon\, .$$

By applying this argument for all $t_0$ (so that the bound $C$ approaches $E_\alpha(\infty)$ when $t_0\to\infty$) and by recalling the definition
of $\Upsilon(t)$, we obtain \eqref{eq:L2bound}.\end{proof}

\begin{lemma}[$L^2$-bound on $u_x$]\label{lem:L2bound-u_x}
Assume that $0\le P<\lambda_1$, that $\limsup_{t\to\infty}\|g(t)\|_{L^2}<\infty$ and that $u$ is a solution of \eqref{enlmg}.
Let $\alpha$ and $E_\alpha(\infty)$ be as in Lemma $\ref{energybound-PDE}$ and $\Psi$ be as in \eqref{eq:L2bound}. Then
\begin{equation}\label{uxbound}
\limsup_{t\to\infty}\|u_x(t)\|_{L^2}^2\le \frac{4E_\alpha(\infty)+2\alpha^2\Psi}{\sqrt{(\lambda_1-P)^2+2S(2E_\alpha(\infty)+\alpha^2\Psi)}+(\lambda_1-P)}\, .
\end{equation}
\end{lemma}
\begin{proof}
Let us rewrite $E_\alpha$ as
$$E_\alpha(t)=\frac12 \int_\Omega \left(\alpha u\xit +u_t\xit\right)^2d\xi-\frac{\alpha^2}{2}\|u(t)\|^2_{L^2}+\frac12 \|u(t)\|_{H^2_*}^2
-\frac{P}{2}\|u_x(t)\|_{L^2}^2+\frac{S}{4}\|u_x(t)\|_{L^2}^4\, .$$
Therefore, by dropping the squared integral, we obtain
\begin{equation}\label{rudeH2}
\frac12 \|u(t)\|_{H^2_*}^2-\frac{P}{2}\|u_x(t)\|_{L^2}^2+\frac{S}{4}\|u_x(t)\|_{L^2}^4\le E_\alpha(t)+\frac{\alpha^2}{2}\|u(t)\|^2_{L^2}\, .
\end{equation}
Using \eqref{embedding} into \eqref{rudeH2}, we obtain
$$
\frac{S}{4}\|u_x(t)\|_{L^2}^4+\frac{\lambda_1-P}{2}\|u_x(t)\|_{L^2}^2\le E_\alpha(t)+\frac{\alpha^2}{2}\|u(t)\|^2_{L^2}\, .
$$
By solving this biquadratic inequality and by taking the limsup, we obtain \eqref{uxbound}.
\end{proof}

\begin{lemma}[$L^2$-bound on $u_t$]\label{lem:L2bound-u_t}
Assume that $0\le P<\lambda_1$, that $\limsup_{t\to\infty}\|g(t)\|_{L^2}<\infty$ and that $u$ is a solution of \eqref{enlmg}.
Let $\alpha$ and $E_\alpha(\infty)$ be as in Lemma $\ref{energybound-PDE}$ and $\Psi$ be as in \eqref{eq:L2bound}. Then, for every $\lambda>0$,
\begin{equation}\label{utbound}
\limsup_{t\to\infty}\|u_t(t)\|_{L^2}^2\le
\frac{1+\lambda}{\lambda}\left(2E_\alpha(\infty)+\max_{s\in [0,\Psi]}\left(((\lambda+1)\alpha^2-(\lambda_1-P))s-\frac{S}{2}s^2\right)\right).
\end{equation}
\end{lemma}
\begin{proof} The Minkowski inequality yields
$$\left(\int_\Omega u_t^2\dxit\right)^{1/2}\le\left(\int_\Omega \left(\alpha u\xit +u_t\xit\right)^2d\xi\right)^{1/2}
+\alpha \left(\int_\Omega u^2\dxit\right)^{1/2}\, .$$
Moreover, by using the expression of the energy and \eqref{embedding}, we see that
$$
\|u_t(t)\|_{L^2}\le\left(2E_\alpha(t)-(\lambda_1-P-{\alpha^2})\|u(t)\|^2_{L^2}-\frac{S}{2}\|u(t)\|_{L^2}^4\right)^{1/2}+\alpha \|u(t)\|_{L^2}
$$
for all $t\ge t_0$. Applying Young's inequality, this yields for every $\lambda>0$
$$\|u_t(t)\|_{L^2}^2\le
\frac{1+\lambda}{\lambda}\left(2E_\alpha(t)+(\lambda\alpha^2-(\lambda_1-P-{\alpha^2}))\|u(t)\|^2_{L^2}-\frac{S}{2}\|u(t)\|_{L^2}^4\right).
$$
\end{proof}

\begin{lemma}[$H^2$-bound on $u$]\label{lem:H2bound}
Assume that $0\le P<\lambda_1$, that $\limsup_{t\to\infty}\|g(t)\|_{L^2}<\infty$ and that $u$ is a solution of \eqref{enlmg}.
Let $\alpha$ and $E_\alpha(\infty)$ be as in Lemma $\ref{energybound-PDE}$, let $\Psi$ be as in \eqref{eq:L2bound}. Then we have
\begin{equation}\label{uH2bound}
\limsup_{t\to\infty}\|u(t)\|_{H^2_*}^2\le \frac{2\lambda_1}{\lambda_1-P}\left(E_\alpha(\infty)+ \frac{\alpha^2\, \Psi}{2}\right)\, .
\end{equation}
\end{lemma}
\begin{proof} By using \eqref{embedding}, we see that \eqref{rudeH2} yields
$$
\frac{\lambda_1-P}{2\lambda_1}\|u(t)\|_{H^2_*}^2 \le E_\alpha(t) + \frac{\alpha^2}{2}\|u(t)\|^2_{L^2}
-\frac{S}{4}\|u(t)\|_{L^2}^4\le E_\alpha(t) + \frac{\alpha^2}{2}\|u(t)\|^2_{L^2}
$$
so that, by taking the limsup and using Lemma \ref{lem:L2bound}, we obtain \eqref{uH2bound}.\end{proof}

We conclude this section by noticing that all the bounds obtained so far can be used for the weak solutions of the {\em linear problem}
\begin{empheq}{align}
\label{perturb-eq}
\left\{
\begin{array}{rl}
w_{tt}+\delta w_t + \Delta^2 w + Pw_{xx} - b w = h(\xi,t)  &\textrm{in }\Omega\times(0,T)\\
w = w_{xx}= 0 &\textrm{on }\{0,\pi\}\times[-\ell,\ell]\\
w_{yy}+\sigma w_{xx} = w_{yyy}+(2-\sigma)w_{xxy}= 0 &\textrm{on }[0,\pi]\times\{-\ell,\ell\},
\end{array}
\right.
\end{empheq}
obtained by taking $S=0$ in \eqref{enlmg} and inserting the additional zero order term $b w$ (that will appear naturally while deriving
the exponential decay of the solutions of the nonlinear equation). More precisely, we have

\begin{lemma}\label{boundsforlinear}
Let $h\in C^0(\R_+,L^2(\Omega))$. For any weak solution $w$ of \eqref{perturb-eq}, we have the estimates
\begin{itemize}
\item ($L^2$ bound on $w_t$)
\begin{equation}\label{utbound-lin}
\limsup_{t\to\infty}\|w_t(t)\|_{L^2}^2\le \frac{2}{\lambda_1-P-b}\limsup_{t\to\infty}\|h(t)\|_{L^2}^2,
\end{equation}
\item ($H^2$ bound on $w$)
\begin{equation}\label{uH2bound-lin}
\limsup_{t\to\infty}\|w(t)\|_{H^2_*}^2\le
\frac{\lambda_1}{\lambda_1-P-b}\left(\max\Big(\frac4{\delta^2}, \frac{1}{\lambda_1-P-b}\Big)+\frac{1}{\lambda_1-P-b}\right)\limsup_{t\to\infty}\|h(t)\|_{L^2}^2.\\
\end{equation}
\end{itemize}
\end{lemma}

\section{Proof of Theorem \ref{exuniq}}\label{proof1}

Local and global existence of weak solution of \eqref{enlmg}-\eqref{initialc} are proved in \cite[Theorem 3]{FerGazMor}. Here, inspired by the work of
Ball \cite[Theorem 4]{ball}, we prove that this solution is a strong solution in case the initial data and the forcing term are slightly more regular, that is, the second part of Theorem \ref{exuniq}.

\begin{proposition}\label{regularity1-ball}
Let $u_0 \in H^4\cap H^{2}_*(\Omega)$, $v_0 \in H^2_*(\Omega)$, $T>0$ and $g \in C^1([0, T], L^2(\Omega))$. Then the unique weak solution $u$ of \eqref{enlmg}-\eqref{initialc} satisfies
\[
u \in C([0,T], H^4\cap H^{2}_*(\Omega))\cap C^1([0,T], H^2_*(\Omega))\cap C^2([0,T], L^2(\Omega)).
\]
\end{proposition}
\begin{proof} We label the eigenfunctions $w_j$ of \eqref{eq:eigenvalueH2L2} with a unique index $j$ and, for all integer $k\ge1$, we set
$E_k=\textrm{span}(w_1,\dots,w_k)$ and we consider the orthogonal projection $Q_k:H^{2}_*(\Omega)\to E_k$. We set up the
weak formulation restricted to test functions $v\in E_k$, namely we seek $u_k\in C^2([0,T],E_k)$ that satisfies
\begin{empheq}{align}\label{weakformwk}
\left\{
\begin{array}{r}
((u_k)_{tt},v)_{L^2} + \delta ((u_k)_t,v)_{L^2} + ({u_k,v})_{H^2_*} +[-P+S\int_\Omega (u_k)_x^2]({(u_k)_x, v_x})_{L^2} =({g ,v})_{L^2}\\
u_k(0) = Q_k u_0, \quad (u_k)_ t(0)= Q_k v_0
\end{array}\right.
\end{empheq}
for all $v\in E_k$ and all $t>0$. The coordinates of $u_k$ in the basis $(w_i)$, given by $u^k_i = (u_k,w_i)_{L^2}$, are time-dependent functions and
from \eqref{weakformwk} we see that they solve the following systems of ODE's, for $i=1, \ldots, k$:
\begin{empheq}{align}\label{pvigik}
\left\{
\begin{array}{r}
(u^k_i)_{tt}(t) + \delta (u^k_i)_ t (t)+ \lambda_i u^k_i(t) +m_i^2\left[-P+S\sum_{j=1}^k m_j^2 u^k_j(t)^2\right]u^k_i(t) = ({g(t), w_i})_{L^2}\,,\\
u^k_i(0)= (u_0, w_i)_{L^2}, \quad (u^k_i)_ t(0)=(v_0, w_i)_{L^2}\,.
\end{array}\right.
\end{empheq}

Since the nonlinearity is analytic, from the classical theory of ODEs, we know that \eqref{pvigik} has a unique solution for each $i = 1, \ldots, k$
and that it can be extended to all $[0,T]$. Therefore \eqref{weakformwk} has a unique solution $u_k\in C^2([0,T],E_k)$, given by
\begin{empheq}{align}
u_k(\xi,t)= \sum_{i=0}^ku^k_i(t) w_i(\xi)\, .
\end{empheq}

Since $v_0 \in H^2_*(\Omega)$, and from the ODE in \eqref{weakformwk} we obtain that
\begin{equation}\label{eq:ODEsecorderStep1}
 \|(u_k)_t(0)\|_{H^2_*} \ \ \text{and} \ \ \|(u_k)_{tt}(0)\|_{L^2} \ \  \text{are uniformly bounded}.
\end{equation}

Then we differentiate \eqref{weakformwk} with respect to $t$, we take $v = (u_k)_{tt}$ and we infer that
\begin{empheq}{multline}  \label{eq:ODEsecder1}
\frac{1}{2} \frac{d}{dt}\left( \| (u_k)_{tt}\|_{L^2}^2 + \|(u_k)_t\|_{H^2_*}^2 \right) + \delta \| (u_k)_{tt}\|_{L^2}^2 = \left( g_t, (u_k)_{tt} \right)_ {L^2} +\vspace{5pt}\\
+ \left( -P(u_k)_{xxt} + \left(S \int_{\Omega}((u_k)_x)^2  \right) (u_k)_{xxt} + 2 S \left((u_k)_x, (u_k)_{xt}  \right)_{L^2}(u_k)_{xx},  (u_k)_{tt} \right)_ {L^2} \vspace{5pt}\\
\leq  \|g_t\|_{L^2} \|(u_k)_{tt}\|_{L^2} + \left(P + S \|(u_k)_x\|_{L^2}^2  \right) \| (u_k)_{xxt}\|_{L^2}\|(u_k)_{tt}\|_{L^2} \vspace{5pt}\\ + 2S \|(u_k)_x\|_{L^2} \|(u_k)_{xt}\|_{L^2} \|(u_k)_{xx}\|_{L^2}\|(u_k)_{tt}\|_{L^2}
\leq \|g_t\|_{L^2} \|(u_k)_{tt}\|_{L^2}+ C \|(u_k)_t\|_{H^2_*} \|(u_k)_{tt}\|_{L^2},
\end{empheq}
where, for the last inequality we have used the Poincaré inequality $\|(u_k)_{xt}\|_{L^2} \leq C \|(u_k)_t\|_{H^2_*}$ and that $\|(u_k)_x\|_{L^2}$ and $\|(u_k)_{xx}\|_{L^2}$ are uniformly bounded with respect to $t \in [0,T]$. For the latter, it is proved in \cite[p. 6318]{FerGazMor} that $u_k$ is uniformly bounded in $C([0,T], H^2_*(\Omega))$. Then, using Young's inequality, we infer that
\begin{empheq}{multline}  \label{eq:ODEsecder2}
\frac{1}{2} \frac{d}{dt}\left( \| (u_k)_{tt}\|_{L^2}^2 + \|(u_k)_t\|_{H^2_*}^2 \right) + \frac{\delta}{2} \| (u_k)_{tt}\|_{L^2}^2 \leq \frac{1}{2\delta}\|g_t\|_{L^2}^2 + C \|(u_k)_ t\|_{H^2_*} \|(u_k)_{tt}\|_{L^2}.
\end{empheq}
Hence, from \eqref{eq:ODEsecorderStep1} and Gronwall Lemma, we infer that
\[
\| (u_k)_{tt}\|_{L^2}^2 \ \ \text{and} \ \ \|(u_k)_t\|_{H^2_*}^2 \ \ \text{are uniformly bounded for all $t \in [0,T]$}
\]
and, by the equation
\[
\Delta^2u_k = - (u_k)_{tt} - \delta (u_k)_t+ [p-S\int_{\Omega} (u_k)_x^2](u_k)_{xx} + Q_kg
\]
we obtain that $\Delta^2 u_k$ is uniformly bounded in $L^2(\Omega)$ for all $t \in [0,T]$ and then that $u_k$ is uniformly bounded in $H^4(\Omega)$ for all $t \in [0,T]$. At this point we can proceed as in the proof of \cite[Theorem 3]{FerGazMor}, starting from p.6318, to finish the proof.
\end{proof}

\section{Proof of Theorem \ref{th:peridoc}}

We look at the PDE as the infinite dimensional dynamical system \eqref{infsystem} where the coefficients $g_k$ are defined by
\eqref{coef-g}. Let
$$g^n(\xi, t)= \sum_{k=1}^n g_k(t)w_k(\xi)\,.$$
We aim first to prove the existence of a periodic solution for this finite approximation of the forcing term and therefore deal with the infinite system
\begin{equation}
\ddot{h}_k(t)+\delta\dot{h}_k(t)+\lambda_kh_k(t)+m_k^2\left[-P+S\sum_{j=1}^\infty m_j^2h_j(t)^2\right]h_k(t)  =  g_k(t) \quad \text{ for }k=1,\ldots,n
\end{equation}

\begin{equation}
\ddot{h}_k(t)+\delta\dot{h}_k(t)+\lambda_kh_k(t)+m_k^2\left[-P+S\sum_{j=1}^\infty m_j^2h_j(t)^2\right]h_k(t)  =  0 \quad \text{ for }k\ge n+1.
\end{equation}
This is equivalent to look for a weak periodic solution $u^n$ of the PDE
\begin{empheq}{align}\label{enlmg-newtheo6}
\left\{
\begin{array}{rl}
u_{tt}+\delta u_t + \Delta^2 u +\left[P-S\int_\Omega u_x^2\right]u_{xx}= g^n(\xi,t)  &\textrm{in }\Omega\times(0,\tau)\\
u = u_{xx}= 0 &\textrm{on }\{0,\pi\}\times[-\ell,\ell]\\
u_{yy}+\sigma u_{xx} = u_{yyy}+(2-\sigma)u_{xxy}= 0 &\textrm{on }[0,\pi]\times\{-\ell,\ell\}.
\end{array}
\right.
\end{empheq}
Since we are only interested in existence, we can look for a time periodic solution having all components $h_k$ identically zero for $k\ge n+1$ and this yields, for the $n$ first components of the solution, the finite system
\begin{equation}\label{finsystem-newprooftheo6-2}
\ddot{h}_k(t)+\delta\dot{h}_k(t)+\lambda_kh_k(t)+m_k^2\left[-P+S\sum_{j=1}^n m_j^2h_j(t)^2\right]h_k(t)  = g_k(t) \quad \text{ for }k=1,\ldots,n.
\end{equation}
This means that we seek a $\tau$-periodic solution $u^n$ of \eqref{enlmg-newtheo6} in the form
\begin{equation}\label{un}
u^n(\xi,t):=\sum_{k=1}^n h_k(t)w_k(\xi).
\end{equation}
The Fourier coefficients $h_k$ also depend on $n$ but we voluntarily write $h_k$ to simplify the notations.\par
We introduce the spaces $C_{\tau}^2(\R)$ and $C_{\tau}^0(\R)$ of $C^2$ and $C^0$ $\tau$-periodic scalar functions.
Then we define the linear diagonal operator $L_n: (C_{\tau}^2(\R))^n\to (C_{\tau}^0(\R))^n $ whose $k$-th component is given by
$$L_n^k(h_1,\ldots,h_n) = \ddot{h}_k(t)+\delta\dot{h}_k(t)+(\lambda_k-m_k^2P)h_k(t)\qquad(k=1,\ldots,n)$$
and the potential $G_n$ defined by
$$G_n(h_1,\ldots,h_n)  = \frac{S}4\sum_{j,k=1}^n m_j^2m_k^2h_j^2h_k^2.$$
It is also convenient to use the boldface notation $\bold{s}=(s_1,\ldots,s_n)$ for any $n$-tuple. With these notations, \eqref{finsystem-newprooftheo6-2} becomes
$$L_n(\bold{h}(t)) + \nabla G_n(\bold{h}(t)) = \bold{g}(t).$$
Since $\delta>0$, for all $\bold{q}\in  (C_{\tau}^0(\R))^n$ there exists a unique $\bold{h}\in (C_{\tau}^2(\R))^n$ such that $L_n(\bold{h})=
\bold{q}$ and $\bold{h}$ may be found explicitly by solving the diagonal system of linear
ODEs. Thanks to the compact embedding $(C_{\tau}^2(\R))^n\subset (C_{\tau}^0(\R))^n$, the inverse $L_n^{-1}:(C_{\tau}^0(\R))^n\to (C_{\tau}^0(\R))^n$ is a compact operator. Consider the nonlinear map
$\Gamma_n:(C_{\tau}^0(\R))^n\times[0,1]\to (C_{\tau}^0(\R))^n$ defined by
$$\Gamma_n(\bold{h},\nu)=L_n^{-1}\left(\bold{g}-\nu \nabla G_n(\bold{h}) \right)\qquad\forall(\bold{h},\nu)\in (C_{\tau}^0(\R))^n\times[0,1]\, .$$
The map $\Gamma_n$ is also compact and, moreover, it satisfies the following property: there exists $H_n>0$ (independent of $\nu$) such that if $\bold{h}\in (C_{\tau}^0(\R))^n$ solves
the equation $\bold{h}=\Gamma_n(\bold{h},\nu)$, then
\begin{equation}\label{Hn}
\|\bold{h}\|_{(C_{\tau}^0(\R))^n}\le H_n.
\end{equation}
Indeed, by Lemma \ref{energybound-PDE}, any periodic solution $u$ of
\begin{equation}\label{enlmg-newtheo6-new}
u_{tt}+\delta u_t + \Delta^2 u +\left[P-\nu S \int_\Omega u_x^2\right]u_{xx}= g^n(\xi,t)\  \textrm{ in }\Omega\times(0,\tau)
\end{equation}
satisfies energy bounds that do not depend on $\nu$, namely\par
$\diamondsuit$ for $\delta^2\le 4(\lambda_1-P)$
\begin{equation*}
\max_{t\in [0,\tau]}E_{\delta/2}(t)\le  \frac{2}{\delta^2} \max_{t\in[0,\tau]}\|g^n(t)\|_{L^2}^2\le \frac{2}{\delta^2} \max_{t\in[0,\tau]}\|g(t)\|_{L^2}^2\, ,
\end{equation*}

$\diamondsuit$ for $\delta^2\ge4(\lambda_1-P)$,
$$
\max_{t\in [0,\tau]}E_\mu(t)\le \frac{1}{2(\lambda_1-P)}\max_{t\in[0,\tau]}\|g(t)\|_{L^2}^2\, ,\qquad\mbox{where }\mu=\frac\delta2-\frac12\sqrt{\delta^2-4(\lambda_1-P)}.
$$
These energy bounds give $H^2_*(\Omega)$-bounds on $u$ and $L^2$-bound on $u_t$ as shown by Lemmas \ref{lem:L2bound-u_t} and \ref{lem:H2bound} (we use here the periodicity of $g$ and $u$). Back to the finite dimensional Hamiltonian system \eqref{finsystem-newprooftheo6-2}, this yields the desired $(C_{\tau}^0(\R))^n$-bound in \eqref{Hn}. Hence, since the equation $\bold{h}=\Gamma_n(\bold{h},0)$ admits a unique solution, the Leray-Schauder principle ensures the existence of a solution $\bold{h}\in (C_{\tau}^0(\R))^n$ of $\bold{h}=\Gamma_n(\bold{h},1)$.
This proves the existence of a $\tau$-periodic solution of the finite system \eqref{finsystem-newprooftheo6-2} and, equivalently, of the PDE \eqref{enlmg-newtheo6}. Let us denote this solution by $u^n$, see \eqref{un}.

To complete the proof of Theorem  \ref{th:peridoc}, we now show that the sequence $(u^n)_n$ converges to a periodic solution $u$ of \eqref{enlmg}. Since the energy bounds on $u^n$ are independent of $n$, the $H^2_*$-bounds on $u^n$ and the $L^2$-bounds on $u^n_t$ are also independent of $n$.
The equation in weak form
\begin{empheq}{align}\label{weakform-newtheo6}
\langle u^n_{tt},v \rangle + \delta (u^n_t,v)_{L^2} + (u^n,v)_{H^2_*} +\big[S\|u^n_x\|_{L^2}^2-P\big](u^n_x,v_x)_{L^2}= (g^n,v)_{L^2}\,,
\end{empheq}
for all $t\in[0,\tau]$ and all $v\in H^{2}_*(\Omega)$, then yields a $(H^{2}_*)'$ bound on $u^n_{tt}$. Up to a subsequence, we can therefore pass to the limit
in the weak formulation of \eqref{enlmg-newtheo6}:
$$
\begin{array}{l}
u^n\to u\mbox{ weakly* in }L^\infty([0,\tau],H^{2}_*(\Omega)),\\
u^n_t\to u_t\mbox{ weakly* in } L^\infty([0,\tau],L^2(\Omega)),\\
u^n_{tt}\to u_{tt}\mbox{ weakly* in } L^\infty([0,\tau],(H^{2}_*(\Omega))').
\end{array}
$$
Hence, there exists a $\tau$-periodic solution $u$ of equation \eqref{enlm}, satisfied in the sense of $L^\infty([0,\tau],(H^{2}_*(\Omega))')$.
To conclude, observe that the continuity properties of $u$ follow from Lemma \ref{justification} and therefore $u$ is also a weak solution in the sense of Definition \ref{df:weaksolution}.

\section{Proof of Theorem \ref{SCstability}}\label{proof2}

The proof of Theorem \ref{SCstability} is based on the following statement.

\begin{lemma}\label{lemme-naze}
Assume \eqref{smallginfty}. There exists $g_0=g_0(\delta,S,P,\lambda_1)>0$ such that if
\begin{equation}\label{g0}
g_\infty=\limsup_{t\to\infty}\|g(t)\|_{L^2}<g_0,
\end{equation}
then there exists $\eta>0$ such that
$$\lim_{t\to\infty} {\rm e}^{\eta t}\left(\|u_t(t)-v_t(t)\|_{L^2}+\|u(t)-v(t)\|_{H^2_*}\right) =0$$
for any two solutions $u$ and $v$ of \eqref{enlmg}.
\end{lemma}
\begin{proof} Let $\eta>0$, to be fixed later. If $u$ and $v$ are two solutions of \eqref{enlmg}, then $w=(u-v){\rm e}^{\eta t}$ is such that
$$
\langle w_{tt},\varphi \rangle  + (\delta-2\eta) (w_t,\varphi)_{L^2} + (w,\varphi)_{H^2_*}  -P(w_x,\varphi_x)_{L^2}  -  \eta(\delta-\eta) (w,\varphi)_{L^2}   = (h(\xi,t){\rm e}^{\eta t},\varphi)_{L^2}
$$
for all $t\in[0,T]$ and all $\varphi\in H^{2}_*(\Omega)$, where
$$h(\xi,t) = S\left( u_{xx}(\xi,t)\int_\Omega u_x^2(\xi,t)d\xi- v_{xx}(\xi,t)\int_\Omega v_x^2(\xi,t)d\xi\right).$$
To get estimates on $w$, we estimate first the $L^2$ norm of $h(\xi,t){\rm e}^{\eta t}$.
We write
$$h(\xi,t) {\rm e}^{\eta t}= S\left(u_{xx}(\xi,t){\rm e}^{\eta t} \int_\Omega (u_x^2-v_x^2)(\xi,t)d\xi+w_{xx}(\xi,t)\int_\Omega v_x^2(\xi,t)d\xi\right).$$
Therefore, we have
$$\|h(t) {\rm e}^{\eta t}\|_{L^2}\le S\big(\|u_{xx}(t)\|_{L^2}\|w_x(t)\|_{L^2}\|u_x(t)+v_x(t)\|_{L^2}+\|w_{xx}(t)\|_{L^2}\|v_x(t)\|_{L^2}\big)$$
so that, by combining \eqref{embedding} with Lemmas \ref{lem:L2bound-u_x} and \ref{lem:H2bound}, we deduce that there exists $K_g>0$ such that
\begin{equation}\label{fundamental}
\limsup_{t\to\infty}\|h(t) {\rm e}^{\eta t}\|_{L^2}^2\le K_g \|w(t)\|_{H^2_*}^2
\end{equation}
and, for a family of varying $g\in C^0(\R_+,L^2(\Omega))$,
\begin{equation}\label{Kg}
K_g\to0\qquad\mbox{if}\qquad g_\infty\to0.
\end{equation}

Taking into account the $H^2$-estimate \eqref{uH2bound-lin} for the linear equation \eq{perturb-eq}, we get
\begin{eqnarray*}
\limsup_{t\to\infty}\|w(t)\|_{H^2_*}^2 & \le & \frac{\lambda_1}{\lambda_1-P-\eta}\left(\max\Big(\frac4{\delta^2}, \frac{1}{\lambda_1-P-\eta}\Big)+\frac{1}{\lambda_1-P-\eta}\right)\limsup_{t\to\infty}\|h(t) {\rm e}^{\eta t}\|_{L^2}^2\\
\mbox{by \eqref{fundamental} }\ & \le & \frac{\lambda_1\, K_g}{\lambda_1-P-\eta}\left(\max\Big(\frac4{\delta^2}, \frac{1}{\lambda_1-P-\eta}\Big)+\frac{1}{\lambda_1-P-\eta}\right)\limsup_{t\to\infty}\|w(t)\|_{H^2_*}^2.
\end{eqnarray*}

Therefore we infer that there exists $\eta>0$ such that
\begin{equation}\label{limitH2}
\lim_{t\to\infty}\|w(t)\|_{H^2_*}=0
\end{equation}
as soon as
$$\frac{\lambda_1\, K_g}{\lambda_1-P-\eta}\left(\max\Big(\frac4{\delta^2}, \frac{1}{\lambda_1-P-\eta}\Big)+\frac{1}{\lambda_1-P-\eta}\right)<1.$$
In view of \eqref{Kg} this happens provided that \eqref{g0} holds for a sufficiently small $g_0=g_0(\delta,S,P,\lambda_1)>0$. Hence, if \eqref{g0} is
fulfilled, then \eqref{limitH2} holds and from \eqref{fundamental} we deduce that also
$$\limsup_{t\to\infty}\|h(t){\rm e}^{\eta t}\|_{L^2}=0\, .$$
Therefore, the estimate \eqref{utbound-lin} for the linear equation \eq{perturb-eq} gives
$$\limsup_{t\to\infty}\|w_t(t)\|_{L^2}=0$$
as well. Since $w_t = \left(\eta (u-v) + (u_t-v_t)\right){\rm e}^{\eta t}$ and $\|{\rm e}^{\eta t}(u-v)\|_{L^2}\to 0$ by \eqref{limitH2},
the proof is complete.\end{proof}

Back to the proof of Theorem \ref{SCstability}, assume first that $g$ is $\tau$-periodic for some $\tau>0$. Then Theorem \ref{exuniq} gives the
existence of a $\tau$-periodic solution $U^p$. If $V$ is another periodic solution (of any period!), then from Lemma \ref{lemme-naze} we know that
$$
\lim_{t\to\infty} {\rm e}^{\eta t}\left(\|U^p_t(t)-V_t(t)\|_{L^2}+\|U^p(t)-V(t)\|_{H^2_*}\right) =0
$$
so that the period of $V$ is also $\tau$ and $V=U^p$. This proves uniqueness of the periodic solution.\par
Finally, assume that $g$ is even with respect to $y$. Let $U$ be a solution of \eqref{enlmg}-\eqref{initialc} with initial data being
purely longitudinal, that is, $U(\xi,0),U_t(\xi,0)\in\he$, see \eqref{decomposition}. Writing $U$ in the form \eqref{Fourier}, with its Fourier components satisfying \eqref{infsystem}, we see that the torsional Fourier components $h^n$ of $U$ satisfy
$$
\ddot{h}^n(t)+\delta\dot{h}^n(t)+\nu h^n(t)+m^2\left[-P+S\sum_{j=1}^\infty m_j^2h_j(t)^2\right]h^n(t)=0\, ,\qquad \dot{h}^n(0)=h^n(0)=0
$$
since $g$ is even with respect to $y$ and its torsional Fourier components are zero. Therefore, $h^n(t)\equiv0$ for all torsional Fourier
coefficient $h^n$ and $U$ is purely longitudinal:
\begin{equation}\label{purelong}
U(t)=U^L(t)\, ,\qquad U^T(t)\equiv0.
\end{equation}

Take now any solution $V$ of \eqref{enlmg}: then, by Lemma \ref{lemme-naze}, we have
\begin{eqnarray*}
0 &=& \lim_{t\to\infty}\left(\|U_t(t)-V_t(t)\|_{L^2}^2+\|U(t)-V(t)\|_{H^2_*}^2\right)\\
\mbox{by \eqref{purelong} } &=& \lim_{t\to\infty}\left(\|U_t^L(t)-V_t(t)\|_{L^2}^2+\|U^L(t)-V(t)\|_{H^2_*}^2\right)\\
\mbox{by orthogonality \eqref{decomposition} } &=& \lim_{t\to\infty}\left(\|U_t^L(t)-V_t^L(t)\|_{L^2}^2+\|V_t^T(t)\|_{L^2}^2+
\|U^L(t)-V^L(t)\|_{H^2_*}^2+\|V^T(t)\|_{H^2_*}^2\right)\\
&\ge& \lim_{t\to\infty}\left(\|V_t^T(t)\|_{L^2}^2+\|V^T(t)\|_{H^2_*}^2\right)\, .
\end{eqnarray*}
According to Definition \ref{defstab} this implies that $g$ makes the system \eqref{enlmg} torsionally stable.

\section{Proof of Theorem \ref{th:multiple-periodic-solutions}} \label{section:multuplicity}

To prove Theorem \ref{th:multiple-periodic-solutions} we follow very closely the arguments in \cite[Section 2]{souplet}, the main
difference being the presence of $b\neq0$ in the below equation \eqref{eq:second-order-nu}.

The proof of Theorem \ref{th:multiple-periodic-solutions} is a straightforward consequence of the following statement.

\begin{proposition}\label{propSouplet}
There exist $T>0$ and a $T$-antiperiodic function $f \in C^{\infty}(\mathbb{R})$ such that the equation
\begin{equation}\label{eq:second-order-nu}
\ddot{v} +\dot{v} + b v + v^3 = f(t)
\end{equation}
admits at least two distinct $T$-antiperiodic solutions of class $C^{\infty}(\mathbb{R})$.
\end{proposition}

Indeed, taking Proposition \ref{propSouplet} for granted, let $v^1$ and $v^2$ be two distinct $T-$antiperiodic solutions of
\eqref{eq:second-order-nu} and set
$$
u^i(\xi,t)=v^i(t)\phi(\xi)\qquad(i=1,2)
$$
where $\phi\in C^\infty(\overline{\Omega})$ is an $L^2$-normalized eigenfunction of \eqref{eq:eigenvalueH2L2}, associated to some eigenvalue
$\lambda$. Then it is straightforward that $u^i$ satisfies
$$
u^i_{tt}+\delta u^i_t + \Delta^2 u^i + \left[P - S \int_\Omega (u^i_x)^2\right] u^i_{xx}$$
$$
=\phi(\xi)\big[\ddot{v}^i+\delta\dot{v}^i+(\lambda-Pm^2)v^i+Sm^4(v^i)^3\big]=\phi(\xi)f(t)\quad\textrm{in }\Omega\times(0,T)
\qquad(i=1,2)\, .
$$
Therefore, we have two periodic solutions of \eqref{enlmg} for $\delta=1$, $\lambda-Pm^2=b$, $S=m^{-4}$, $g(\xi,t)=\phi(\xi)f(t)$.
This completes the proof of Theorem \ref{th:multiple-periodic-solutions}, provided that Proposition \ref{propSouplet} holds.\par\smallskip

So, let us now prove Proposition \ref{propSouplet}. We suppose that $u$ and $v$ are two solutions of \eqref{eq:second-order-nu} and
we set $w=v-u$. Then
\begin{eqnarray*}
\ddot{w} +\dot{w} + b w + w^3 =\ddot{v} +\dot{v} + b v + v^3 - (\ddot{u} +\dot{u} + b u + u^3) - 3 v^2 u + 3 v u^2 = - 3 v^2 u + 3 v u^2
\end{eqnarray*}
and from the identity $3 u w^2 + 3 u^2 w = 3 u v^2 - 3 u^2v$ we infer that
\begin{equation}\label{eq:second-order-for-w}
\ddot{w} +\dot{w} + b w + 3 u w^2 + 3 u^2w + w^3 =0.
\end{equation}
So, at every point where $w \neq 0$,  $u$ is a root of a second order polynomial, namely
\begin{equation}\label{eq:second-order-polynomial}
 3 u^2 + 3 w u + \left( w^2 + \frac{\ddot{w} + \dot{w}}{w} + b \right) = 0,
\end{equation}
whose discriminant reads
\begin{equation}\label{eq:discriminant}
 9 w^2 - 12 \left( w^2 + \frac{\ddot{w} + \dot{w}}{w} + b \right) = - 12\left( \frac{w^2}{4}  + \frac{\ddot{w} + \dot{w}}{w} + b \right).
\end{equation}

To construct the appropriate source term $f(t)$, inspired by the expression \eqref{eq:discriminant}, we start with the following local result.

\begin{lemma}\label{lemma:polynomial}
There exist a real polinomial $P$ of degree $5$ and a neighborhood $V$ of $0$ such that
\[
\begin{array}{l}
P(t) t >0 \ \ \text{on} \ \ V \backslash \{0\} \ \ \text{and} \ \ \dot{P}(t) >0 \ \ \text{on} \ \ V,  \vspace{5pt}\\
\phi_P(t) :=  \left\{ \begin{array}{cl}- \left( \frac{P^2(t)}{4} + \frac{\ddot{P}(t) + \dot{P}(t)}{P(t)} + b\right)   & \text{if} \ \ t \in V \backslash\{0\}, \vspace{5pt}\\
0 &  \text{if} \ \ t =0
\end{array}
\right.  \qquad \text{is of class $C^{\infty}(V)$}, \vspace{5pt}\\
\phi_P > 0 \ \ \text{on} \ \ V\backslash\{0\},\ \  \dot{\phi}_P(0) = 0 \ \ \text{and} \ \ \ddot{\phi}_P(0)>0\,.
\end{array}
\]
\end{lemma}
\begin{proof}
We search for a polynomial of the form
\begin{equation}\label{eq:correct-polynomial}
P(t) = t + \frac{A}{2} t^2 + \frac{B}{6} t^3 + \frac{C}{24} t^4 + \frac{D}{120} t^5,
\end{equation}
where $A, B, C$ and $D$ will be suitably chosen. Observing that
\[
\dot{P}(t) = 1 + A t + \frac{B}{2} t^2 + \frac{C}{6} t^3 + \frac{D}{24} t^4 \ \ \text{and} \ \ \ddot{P}(t) = A + B t + \frac{C}{2} t^2+ \frac{D}{6} t^3\, ,
\]
and choosing $A=-1$, we infer that
\[
\begin{array}{rcl}
\displaystyle{\frac{\ddot{P}(t) + \dot{P}(t)}{P(t)} } + b&=& \displaystyle{\frac{1 + A + (A +B) t + \frac{B+C}{2} t^2 + \frac{C+D}{6} t^3 + \frac{D}{24} t^4 }{t + \frac{A}{2} t^2 + \frac{B}{6} t^3 + \frac{C}{24} t^4 + \frac{D}{120} t^5}} + b \vspace{10pt}\\
  &=& \displaystyle{\frac{-1 +B + b + \frac{B+C-b}{2} t + \frac{C+D+bB}{6} t^2 + \frac{D+ bC}{24} t^3 + \frac{bD}{120} t^4 }{1 - \frac{1}{2} t + \frac{B}{6} t^2 + \frac{C}{24} t^3 + \frac{D}{120} t^4}}.
\end{array}
\]
So we choose $B$ and $C$ such that
\[
-1 + B + b =0 \ \ \text{and} \ \ B + C -b = 0.
\]
Hence, choosing $A =-1$, $B =1-b$ and $C = 2b-1$, we may write
\[
\begin{array}{rcl}
\displaystyle{\frac{P^2(t)}{4} + \frac{\ddot{P}(t) + \dot{P}(t)}{P(t)} + b }& =& \displaystyle{ \frac{\frac{C+D+bB}{6} t^2 + \frac{D+bC}{24} t^3 + \frac{bD}{120} t^4}{1 - \frac{1}{2} t + \frac{B}{6} t^2 + \frac{C}{24} t^3 + \frac{D}{120} t^4}} + \displaystyle{\frac{t^2}{4} \left( 1 - \frac{1}{2} t + \frac{B}{6} t^2 + \frac{C}{24} t^3 +\frac{D}{120} t^3\right)^2} \vspace{10pt}\\
& =& \displaystyle{\left(\frac{C + D + bB}{6} + \frac{1}{4}  \right) t^2 + \left( \frac{N(t)}{1 - \frac{1}{2} t + \frac{B}{6} t^2 + \frac{C}{24} t^3 + \frac{D}{120} t^4} + Q(t)  \right) t^3}
\end{array}
\]
where $N(t)$ and $Q(t)$ are polynomials. So, we must choose $D$ such that
\[
\frac{C + D+ bB}{6} + \frac{1}{4} = \frac{2C + 2 D + 2bB+3}{12} <0.
\]
With the choice $B =1-b$ and $C = 2b-1$, the last inequality is equivalent to ask $D < b^2 - 3 b - \frac{1}{2}$ and we may take $D = b^2 - 3 b -1$. Therefore, with
\[
A =-1, \ \ B =1-b, \ \ C = 2b-1, \ \ D = b^2 - 3 b -1
\]
the polynomial $P(t)$ given by \eqref{eq:correct-polynomial} satisfies all of the conditions of this lemma on a suitably small neighborhood $V$ of $0$.
\end{proof}

\begin{proposition}\label{Proposition2.3-souplet}
There exist $T> 0$ and $w\in C^{\infty}(\mathbb{R})$, $T$-antiperiodic,  with
\begin{enumerate}[a)]
\item $w >0$ on $(0,T)$,
 \vspace{5pt}
\item $\phi(t) :=  \left\{ \begin{array}{cl}- \left( \frac{w^2(t)}{4} + \frac{\ddot{w}(t) + \dot{w}(t)}{w(t)} + b\right)   & \text{if} \ \ t \notin T\mathbb{Z}, \vspace{5pt}\\
0 &  \text{if} \ \ t \in T \mathbb{Z}
\end{array}
\right.$
 \vspace{5pt}

\noindent $\phi>0$ on $(0,T)$, $\phi(0) = \dot{\phi}(0) = 0$, $\ddot{\phi}(0)>0$ and $\phi \in C^{\infty}(\mathbb{R})$.
\end{enumerate}
\end{proposition}
\begin{proof} The proof is very similar to the proof of \cite[Proposition 2.3]{souplet} and, for the sake of completeness, we just stress the difference caused by the extra term $b$. Following the proof of \cite[Proposition 2.3]{souplet}, everything remains unchanged except that:
\begin{enumerate}[i)]
\item At \cite[p.1522]{souplet}, the definitions of $\psi$ and $\psi_n$ now read
\[
\psi = - \left( \frac{h^3}{4} + \ddot{h} + \dot{h} + bh \right) \ \ \text{and} \ \ \psi_n = - \left( \frac{w_n^3}{4} + \ddot{w_n} + \dot{w_n} + bw_n \right).
\]
\item At \cite[p.1523, lines 8-10]{souplet}, the estimate $\ddot{h}_n + \dot{h}_n \leq (\widetilde{\ddot{h}} + \widetilde{\dot{h}}) \ast \rho_n$ now reads $\ddot{h}_n + \dot{h}_n + b h_n\leq (\widetilde{\ddot{h}} + \widetilde{\dot{h}} + bh) \ast \rho_n$.

\item At \cite[p.1523, lines 13-14]{souplet}, the estimate $\displaystyle{(\ddot{h}_n + \dot{h}_n) (t) \leq \int_{-1/n}^{+1/n} - \left[ \frac{1}{4} h^3(t -s) + \gamma \right] \rho_n(s) ds}$ now reads
$\displaystyle{(\ddot{h}_n + \dot{h}_n + bh_n) (t) \leq \int_{-1/n}^{+1/n} - \left[ \frac{1}{4} h^3(t -s) + \gamma \right] \rho_n(s) ds}$.

\item And finally, at \cite[p.1523, eq. (2.21)]{souplet}, the inequality
\[
\displaystyle{(h_n^3/4 + \ddot{h}_n + \dot{h}_n) (t) \leq -\gamma + \frac{1}{4} |h_n^3(t) - h^3(t)| + \frac{1}{4}\int_{-1/n}^{+1/n} |h^3(t -s)-h^3(t)| \rho_n(s) ds}
\]
now becomes
\[
\displaystyle{(h_n^3/4 + \ddot{h}_n + \dot{h}_n + bh_n) (t) \leq -\gamma + \frac{1}{4} |h_n^3(t) - h^3(t)| + \frac{1}{4}\int_{-1/n}^{+1/n} |h^3(t -s)-h^3(t)| \rho_n(s) ds}.
\]
\end{enumerate}
This completes the proof.\end{proof}

We are now ready to conclude the proof of Proposition \ref{propSouplet}.

\begin{proof}[Proof of Proposition \ref{propSouplet} completed]
Let $w$ and $\phi(t)$ be as in Proposition \ref{Proposition2.3-souplet} and $\theta(t)$ be the  (discontinuous) $T$-antiperiodic function such that $\theta(t) =1$ on $[0,T)$.  Taking into account equations \eqref{eq:second-order-polynomial} and \eqref{eq:discriminant}, we set
\begin{equation}\label{eq:def-of-u}
\displaystyle{u(t) = \frac{- 3 w(t) + \sqrt{12} \sqrt{\phi(t)} \,\theta(t)}{6} = -\frac{1}{2} w(t) + \frac{1}{\sqrt{3}} \theta(t) \sqrt{\phi(t)}} \ \ \text{and} \ \ v(t) = u(t) + w(t).
\end{equation}
Observe that $\theta(t) \sqrt{\phi(t)}$ is $T$-antiperiodic and \cite[Lemma 2.5]{souplet} guarantees that $\theta(t) \sqrt{\phi(t)}$ is of class $C^{\infty}$.  As a consequence, $u$ and $v$ are of class $C^{\infty}$ and $T$-antiperiodic. Moreover,
\begin{equation}\label{eq:identity-u-v}
\ddot{v} + \dot{v} + b v + v^3 - (\ddot{u} + \dot{u} + b u + u^3) = \ddot{w} + \dot{w} + b w + w^3 + 3 v^2u - 3 v u^2 = \ddot{w} + \dot{w} + b w + w^3 + 3 u w^2 + 3 u^2w.
\end{equation}
Also observe that, by the definition of $u$,
\[
\left( u + \frac{w}{2} \right)^2 = u^2 + u w + \frac{w^2}{4} = - \frac{1}{3} \left( \frac{w^2}{4} +\frac{\ddot{w} + \dot{w}}{w} + b \right)
\]
and we infer from the last identity that
\[
\ddot{w} + \dot{w} + b w + w^3 + 3 u w^2 + 3 u^2w = 0,
\]
which, combined with \eqref{eq:identity-u-v}, implies that
\[
\ddot{v} + \dot{v} + b v + v^3 = \ddot{u} + \dot{u} + b u + u^3.
\]
Then we choose $f = \ddot{u} + \dot{u} + b u + u^3$ and we obtain two periodic solutions.
\end{proof}

\section{Proof of Theorem \ref{SCstability2}}\label{proofSCstab}

We start this proof with a technical result.

\begin{lemma}\label{fubini} Let $u\in C^0(\R_+;H^{2}_*(\Omega))\cap C^1(\R_+;L^2(\Omega))\cap C^2(\R_+;(H^{2}_*(\Omega))')$ be a weak solution of \eqref{enlmg} as in Definition \ref{df:weaksolution}.
Then, for all $0<t<s$, we have
$$
\int_t^s\langle u^T_{tt}(\tau),u^T(\tau)\rangle d\tau = -\int_t^s\|u^T_t(\tau)\|_{L^2}^2\, d\tau+\int_\Omega\Big[u^T_{t}(s)u^T(s)-u^T_{t}(t)u^T(t)\Big]\, .
$$
\end{lemma}
\begin{proof} Let $u_k\in C^2([0,T],E_k)$ be the Galerkin sequence defined in the proof of Proposition \ref{regularity1-ball}. From Step 3 in the proof of \cite[Theorem 3]{FerGazMor} we know that $u_k \to u$ in $C^0([0,T], H^2_*(\Omega))$ for all $T>0$. Moreover, given $T>0$, from \cite[Eq. (21) and Step 3]{FerGazMor} we infer that the sequence $((u_k)_{tt})$ is bounded in $C^0([0,T], (H^2_*(\Omega))')$.  Hence, up to a subsequence,
$(u_k)_{tt}(t) \rightharpoonup u_{tt}(t)$ in $(H^2_*(\Omega))'$ for each $t$ and such convergence, as from Step 4 in the proof of \cite[Theorem 3]{FerGazMor}, reads
\[
\langle u_{tt}(\tau), v\rangle = \lim_{k \to \infty} \int_{\Omega} (u_k)_{tt} (\tau) v \, d\xi  \ \ \text{for all} \ \ v \in H^2_*(\Omega).
\]
Using the orthogonal decomposition \eqref{eq:DEO} we infer that
$$
\langle u^T_{tt}(\tau),u^T(\tau)\rangle=\lim_{k\to\infty}\int_\Omega(u_k)^T_{tt}(\tau)(u_k)^T(\tau)d\xi  \ \ \text{for all $\tau \in [t,s]$}.
$$
Whence, by the Lebesgue Dominated Convergence Theorem, the Fubini Theorem and an integration by parts, we obtain
\begin{eqnarray*}
\int_t^s\langle u^T_{tt}(\tau),u^T(\tau)\rangle d\tau &=& \int_t^s\lim_{k\to\infty}\int_\Omega(u_k)^T_{tt}(\tau)(u_k)^T(\tau)d\xi\, d\tau =\lim_{k\to\infty}\int_t^s\int_\Omega(u_k)^T_{tt}(\tau)(u_k)^T(\tau)d\xi\, d\tau\\
\ &=& \lim_{k\to\infty}\int_\Omega\int_t^s(u_k)^T_{tt}(\tau)(u_k)^T(\tau)d\tau\, d\xi\\
\ &=& \lim_{k\to\infty}\bigg(-\int_t^s\|(u_k)^T_t(\tau)\|_{L^2}^2\, d\tau+\int_\Omega\Big[(u_k)^T_{t}(s)(u_k)^T(s)-(u_k)^T_{t}(t)(u_k)^T(t)\Big]\bigg)\\
\ &=& -\int_t^s\|u^T_t(\tau)\|_{L^2}^2\, d\tau+\int_\Omega\Big[u^T_{t}(s)u^T(s)-u^T_{t}(t)u^T(t)\Big]
\end{eqnarray*}
and the result follows.\end{proof}

Then we establish an exponentially fast convergence result for a related linear problem. The exponential decay is obtained in three steps: first we prove that
the liminf of the norms of the solution tends to 0, then we prove that the limit of the norms tends to 0 which, finally, allows us to argue as in \cite{fitouriharaux} to infer the exponential decay. We point out that we deal with a PDE and not with an ODE as in \cite[Lemma 3.7]{fitouriharaux}.

\begin{lemma}\label{H37}
Assume that the continuous function $a=a(t)$ satisfies $a\ge0$ and $a_\infty:=\limsup_{t\to\infty}a(t)<\infty$.
Assume that
$$\delta>\max\left\{2,\frac{\nu_{1,2}\, a_\infty^2}{\gamma(2\nu_{1,2}-P)}\right\},$$
where $\gamma$ is the optimal constant for inequality \eqref{embedding-H^2_*-u_xx}, and let
$$u\in C^0(\R_+;H^{2}_*(\Omega))\cap C^1(\R_+;L^2(\Omega))\cap C^2(\R_+;(H^{2}_*(\Omega))')$$
be a weak solution of \eqref{enlmg} (see Definition \ref{df:weaksolution}) such that
\begin{empheq}{align}\label{weakform2}
\langle u^T_{tt},v \rangle + \delta (u^T_t,v)_{L^2} + (u^T,v)_{H^2_*}+\big(a(t)-P\big)(u^T_x,v_x)_{L^2}=0\quad\forall t>0\, ,\ \forall v\in\ho\, .
\end{empheq}
Then there exist $\rho,C,\kappa>0$ such that
$$
\Big(\|u^T_t(t)\|_{L^2}^2+\|u^T(t)\|_{H^2_*}^2\Big)\le C\, e^{-\kappa t}\qquad\forall t\ge\rho\, .
$$
\end{lemma}
\begin{proof} We formally take $v=u^T_t(t)$ in \eqref{weakform2} and obtain
$$
\langle u^T_{tt}(t),u^T_t(t)\rangle+\delta\|u^T_t(t)\|_{L^2}^2+(u^T(t),u^T_t(t))_{H^2_*}+\big(a(t)-P\big)\int_\Omega u^T_x(t)u^T_{xt}(t)=0\quad\forall t>0\, .
$$
In fact, one cannot take $v=u^T_t(t)$ in \eqref{weakform2} since we merely have $u^T_t(t)\in L^2(\Omega)$ but this procedure is rigorously justified by
Lemma \ref{justification}. By integrating the above identity over $(t,s)$ for some $0<t<s$, we find
$$
\frac12 \Big[\|u^T_t(t)\|_{L^2}^2+\|u^T(t)\|_{H^2_*}^2-P\|u^T_x(t)\|_{L^2}^2\Big]=\frac12 \Big[\|u^T_t(s)\|_{L^2}^2+\|u^T(s)\|_{H^2_*}^2-P\|u^T_x(s)\|_{L^2}^2\Big]
$$
\begin{equation}\label{whoknows}
+\delta\int_t^s\|u^T_t(\tau)\|_{L^2}^2d\tau +\int_t^s a(\tau)\int_\Omega u^T_x(\tau)u^T_{xt}(\tau) d\tau\, .
\end{equation}
With an integration by parts and by the H\"older inequality, \eqref{whoknows} yields the estimate
$$
\delta\int_t^s\|u^T_t(\tau)\|_{L^2}^2d\tau \le -\frac12 \Big[\|u^T_t(\tau)\|_{L^2}^2+\|u^T(\tau)\|_{H^2_*}^2-P\|u^T_x(\tau)\|_{L^2}^2\Big]_t^s+
A(t)\int_t^s\|u^T_t(\tau)\|_{L^2}\|u^T_{xx}(\tau)\|_{L^2} d\tau
$$
where $A(t):=\sup_{\tau>t}a(\tau)$. By the Young inequality and \eqref{embedding-H^2_*-u_xx} we infer that
\begin{equation}\label{firstbound}
\frac{\delta}{2}\int_t^s\|u^T_t(\tau)\|_{L^2}^2d\tau \le-\frac12 \Big[\|u^T_t(\tau)\|_{L^2}^2+\|u^T(\tau)\|_{H^2_*}^2-P\|u^T_x(\tau)\|_{L^2}^2\Big]_t^s+
\frac{A(t)^2}{2\delta\gamma}\int_t^s\|u^T(\tau)\|_{H^2_*}^2 d\tau\, .
\end{equation}

Then we take $v=u^T(t)$ in \eqref{weakform2} and obtain
\begin{equation}\label{step1}
\langle u^T_{tt}(t)u^T(t)\rangle+\delta\int_\Omega u^T_t(t)u^T(t)+\|u^T(t)\|_{H^2_*}^2+\big(a(t)-P\big)\|u^T_x(t)\|_{L^2}^2=0\quad\forall t>0\, .
\end{equation}
Consider the same $0<t<s$ as above and note that, by integrating \eqref{step1} over $(t,s)$ and using Lemma \ref{fubini}, we get (recall $a\ge0$)
$$
\int_t^s\|u^T(\tau)\|_{H^2_*}^2d\tau-P\int_t^s\|u^T_x(\tau)\|_{L^2}^2d\tau
$$
\begin{equation}\label{secondbound}
\le\int_t^s\|u^T_t(\tau)\|_{L^2}^2\, d\tau+\int_\Omega\Big[u^T_t(t)u^T(t)-u^T_t(s)u^T(s)\Big]-\frac{\delta}{2}\Big[\|u^T(\tau)\|_{L^2}^2\Big]_t^s\, .
\end{equation}

By combining \eqref{firstbound} with \eqref{secondbound} we infer that
$$
\left(\frac{\delta}{2}-1\right)\int_t^s\|u^T_t(\tau)\|_{L^2}^2d\tau
+\left(1-\frac{A(t)^2}{2\delta\gamma}\right)\int_t^s\|u^T(\tau)\|_{H^2_*}^2d\tau-P\int_t^s\|u^T_x(\tau)\|_{L^2}^2d\tau
$$
\begin{equation}\label{mother}
\le\left[-\frac{\|u^T_t(s)\|_{L^2}^2}2 -\frac{\|u^T(s)\|_{H^2_*}^2}2 +\frac{P\|u^T_x(s)\|_{L^2}^2}2-\int_\Omega\Big[u^T_t(s)u^T(s)\Big]-\frac{\delta}{2}\|u^T(s)\|_{L^2}^2\right]
\end{equation}
$$
+\left[\frac{\|u^T_t(t)\|_{L^2}^2}2 +\frac{\|u^T(t)\|_{H^2_*}^2}2 -\frac{P\|u^T_x(t)\|_{L^2}^2}2+\int_\Omega\Big[u^T_t(t)u^T(t)\Big]+\frac{\delta}{2}\|u^T(t)\|_{L^2}^2\right]\, .
$$
Since $\delta>1$ the second line of \eqref{mother} is negative while the third line is upper bounded by
$$
\|u^T_t(t)\|_{L^2}^2 +\frac{\|u^T(t)\|_{H^2_*}^2}2 +\frac{\delta+1}{2}\|u^T(t)\|_{L^2}^2\, .
$$
Therefore, by recalling \eqref{improvedtorsional}, \eqref{mother} yields (for all $t>0$)
\begin{equation}\label{psifinal}
\left(\frac{\delta}{2}-1\right)\int_t^s\|u^T_t(\tau)\|_{L^2}^2d\tau
+\left(1-\frac{A(t)^2}{2\delta\gamma}-\frac{P}{2\nu_{1,2}}\right)\int_t^s\|u^T(\tau)\|_{H^2_*}^2d\tau\le
\|u^T_t(t)\|_{L^2}^2 +\frac{\nu_{1,2}+\delta+1}{2\nu_{1,2}}\|u^T(t)\|_{H^2_*}^2\, .
\end{equation}
Since $\limsup_{t\to\infty}A(t)=a_\infty$, we may take $t$ sufficiently large, say $t\ge\rho$, in such a way that
$$
1-\frac{A(t)^2}{2\delta\gamma}-\frac{P}{2\nu_{1,2}}\ge\varepsilon>0\qquad\forall t\ge\rho\, .
$$
Then, if we let $s\to\infty$ and we put
$$
\psi(t):=\left(\frac{\delta}{2}-1\right)\|u^T_t(t)\|_{L^2}^2+\varepsilon\|u^T(t)\|_{H^2_*}^2\, ,
$$
inequality \eqref{psifinal} implies that
\begin{equation}\label{crucial}
\int_t^\infty\psi(\tau)\, d\tau\le\frac{\psi(t)}{\kappa}\qquad\forall t\ge\rho\qquad\mbox{where}\quad
\frac{1}{\kappa}=\max\left\{\frac{2}{\delta-2},\frac{\nu_{1,2}+\delta+1}{2\varepsilon\nu_{1,2}}\right\}\, .
\end{equation}

This inequality has two crucial consequences. First, we remark that
\begin{equation}\label{consequence}
\liminf_{t\to\infty}\psi(t)=0
\end{equation}
since $\psi\ge0$ and the integral in \eqref{crucial} converges. Second, we see that \eqref{crucial} readily implies
\begin{equation}\label{gronwall}
\int_t^\infty\psi(\tau)\, d\tau\le\left[e^{\kappa\rho}\int_\rho^\infty\psi(\tau)\, d\tau\right]\, e^{-\kappa t}=:C_\rho\, e^{-\kappa t}\qquad\forall t\ge\rho\, .
\end{equation}

From \eqref{consequence} we infer that there exist an increasing sequence $s_m\to\infty$ ($s_m>\rho$ for all $m$) such that
$$
\varepsilon_m:=\frac12 \Big[\|u^T_t(s_m)\|_{L^2}^2+\|u^T(s_m)\|_{H^2_*}^2-P\|u^T_x(s_m)\|_{L^2}^2\Big]\to0\qquad\mbox{as }m\to\infty\, .
$$
Then, by taking $t\in(s_{m-1},s_m)$ and $s=s_m$ in \eqref{whoknows}, we get
\begin{eqnarray*}
\frac{\|u^T_t(t)\|_{L^2}^2+\|u^T(t)\|_{H^2_*}^2-P\|u^T_x(t)\|_{L^2}^2}2 &\!\!\le\!\!& \varepsilon_m
+\delta\int_t^{s_m}\!\!\|u^T_t(\tau)\|_{L^2}^2d\tau +A(t)\int_t^{s_m}\!\!\|u^T_t(\tau)\|_{L^2}\|u^T_{xx}(\tau)\|_{L^2} d\tau\\
\mbox{(by \eqref{gronwall}) }\quad &\!\!\le\!\!& \varepsilon_m+C_1\int_t^\infty\!\!\psi(\tau)\, d\tau\le\varepsilon_m+C_1C_\rho\, e^{-\kappa t}\qquad\forall t\in(s_{m-1},s_m)\, .
\end{eqnarray*}
Since the expression on the left hand side is estimated both from above and below by a constant times $\psi(t)$,
this proves that
$$
\lim_{t\to\infty}\psi(t)=0\, .
$$
Going back to \eqref{whoknows} and by letting $s\to\infty$, this shows that
$$
\frac12 \Big[\|u^T_t(t)\|_{L^2}^2+\|u^T(t)\|_{H^2_*}^2-P\|u^T_x(t)\|_{L^2}^2\Big]\le C_1C_\rho\, e^{-\kappa t}\qquad\forall t>\rho\, .
$$
We conclude by using again the fact that the left hand side can be bounded from below by $c\psi(t)$ for a suitable constant $c>0$.\end{proof}

\begin{proof}[Proof of Theorem \ref{SCstability2}]
We may assume that $\delta^2\ge4(\lambda_1-P)$. Still denoting $g_\infty=\limsup_{t\to\infty}\|g(t)\|_{L^2}$, we infer from the second estimate in Lemma \ref{energybound-PDE}
combined with Lemma \ref{lem:L2bound}
that
$$
\limsup_{t\to \infty}\|u(t)\|^2_{L^2} < \frac{g_\infty^2 }{(\lambda_1-P)^2}
$$
since $S>0$. Now we set $a(t):=\|u_x(t)\|_{L^2}^2\ge0$ and $a_\infty:=\limsup_{t\to\infty}a(t)<\infty$. Using \eqref{uxbound}, we see that $a_\infty$ is bounded from above by
a constant depending on $g_\infty,S,P$ and $\lambda_1$. Hence it is enough to apply Lemma \ref{H37} to conclude.
\end{proof}

\section{Proof of Theorem \ref{smalltorsion}}\label{proof3}

Let $u$ be a weak solution of \eqref{enlmg} and set $a(t):=S\|u_x(t)\|_{L^2}^2$. Following \eqref{decomposition}, we write $u(t)=u^{T}(t)+u^L(t)$ and,
since this decomposition is orthogonal in $H^2_*(\Omega)$ and in $L^2(\Omega)$, we get
\begin{empheq}{align}\label{weakform-decompo}
\langle u^T_{tt},v \rangle  + \delta (u^T_t,v)_{L^2} + (u^T,v)_{H^2_*}  +\big[a(t)-P\big](u^T_x,v_x)_{L^2} \ \ & \\
+\langle u^L_{tt},v \rangle + \delta (u^L_t,v)_{L^2} + (u^L,v)_{H^2_*} +\big[a(t)-P\big](u^L_x,v_x)_{L^2}  & = (g^L,v)_{L^2},
\end{empheq}
for all $t\in[0,T]$ and all $v\in H^{2}_*(\Omega)$, as $g$ is even with respect to $y$ (so that $g=g^L$). Then
\begin{equation}\label{weaktorsion}
\langle u^T_{tt},v \rangle  + \delta (u^T_t,v)_{L^2} + (u^T,v)_{H^2_*}  +\big[a(t)-P\big](u^T_x,v_x)_{L^2}  = 0
\end{equation}
for all $t\in[0,T]$ and all $v\in H^{2}_*(\Omega)$. Indeed, splitting $v=v^{T} + v^L$,  we see from \eqref{weakform-decompo} that
\begin{multline*}
\langle u^T_{tt},v \rangle + \delta (u^T_t,v)_{L^2} + (u^T,v)_{H^2_*} +\big[a(t)-P\big](u^T_x,v_x)_{L^2}
\\ = \langle u^T_{tt},v^T \rangle + \delta (u^T_t,v^T)_{L^2} + (u^T,v^T)_{H^2_*} +\big[a(t)-P\big](u^T_x,v^T_x)_{L^2} = (g,v^T)_{L^2} = 0\,,
\end{multline*}
for all $t\in[0,T]$ and all $v\in H^{2}_*(\Omega)$.\par
Setting $w(t):=u^T(t){\rm e}^{\eta t}$ for some $\eta>0$, \eqref{weaktorsion} becomes
$$
\langle w_{tt},v \rangle+(\delta-2\eta)(w_t,v)_{L^2}+(w,v)_{H^2_*}-P(w_x,v_x)_{L^2}-\eta(\delta-\eta)(w,v)_{L^2}=a(t)(w_{xx},v)_{L^2},
$$
for all $t\in[0,T]$ and all $v\in H^{2}_*(\Omega)$. This shows that $w$ weakly solves \eqref{perturb-eq} with
$$
\delta\mbox{ replaced by }\delta-2\eta\, ,\qquad b=\eta(\delta-\eta)\, ,\qquad h(\xi,t)=a(t)w_{xx}(\xi,t)\, .
$$
Take $\eta\in(0,\delta/2)$ such that $K:=\nu_{1,2}-P-\eta(\delta-\eta)>0$ and assume that
\begin{equation}\label{assumptionasmall}
\limsup_{t\to\infty}|a(t)|^2<\frac{\gamma\, K^2}{\nu_{1,2}}\left[1+\max\left\{\frac{4K}{(\delta-2\eta)^2}\, ,\, 1\right\}\right]^{-1}
\end{equation}
where $\gamma$ is as in \eqref{embedding-H^2_*-u_xx}. From \eqref{uH2bound-lin} and \eqref{embedding-H^2_*-u_xx} we infer that
\begin{equation*}
\limsup_{t\to\infty}\|w(t)\|_{H^2_*}^2\le\frac{\nu_{1,2}}{\gamma\, K}\left(\max\Big(\frac4{(\delta-2\eta)^2}, \frac{1}{K}\Big)+\frac{1}{K}\right)
\limsup_{t\to\infty}|a(t)|^2\cdot\limsup_{t\to\infty}\|w(t)\|_{H^2_*}^2
\end{equation*}
which, together with \eqref{assumptionasmall}, proves that
$$
\lim_{t\to\infty}\|w(t)\|_{H^2_*}^2=0\, .
$$
Then we deduce from \eqref{utbound-lin} that
$$
\limsup_{t\to\infty}\|w_t(t)\|_{L^2}^2\le \frac{2}{K}\limsup_{t\to\infty}|a(t)|^2 \|w_{xx}(t)\|_{L^2}^2 = 0.
$$
Finally, undoing the change of unknowns and back to $u$, we infer that
$$\lim_{t\to\infty} {\rm e}^{\eta t}\left(\|u^T(t)\|_{H^2_*}^2+\|u^T_t(t)\|_{L^2}^2\right)=0,$$
which proves the statement.

\bibliographystyle{abbrv}

\begin{thebibliography}{10}

\bibitem{abramst} M. Abramowitz and I.A. Stegun.
\newblock {\em Handbook of mathematical functions with formulas, graphs, and mathematical tables}.
\newblock National Bureau of Standards Applied Mathematics Series {\bf 55}, Washington, D.C. (1964).

\bibitem{2014al-gwaizNATMA}
M.~Al-Gwaiz, V.~Benci, and F.~Gazzola.
\newblock Bending and stretching energies in a rectangular plate modeling
  suspension bridges.
\newblock {\em Nonlinear Anal.}, 106:18--34, 2014.

\bibitem{ammann}
O.~H. Ammann, T.~von K{\'a}rm{\'a}n, and G.~B. Woodruff.
\newblock The failure of the Tacoma Narrows Bridge.
\newblock Technical Report, Federal Works Agency. Washington, D. C., 1941.

\bibitem{arioligazzola}
G. Arioli, F. Gazzola.
\newblock Torsional instability in suspension bridges: the Tacoma Narrows Bridge case.
\newblock {\em Communications in Nonlinear Science and Numerical Simulation}, 42:342--357, 2017.

\bibitem{ball}
J.~M. Ball.
\newblock Initial-boundary value problems for an extensible beam.
\newblock {\em J. Math. Anal. Appl.}, 42:61--90, 1973.

\bibitem{bebuga} E. Berchio, D. Buoso, F. Gazzola, {\em On the variation of longitudinal and torsional frequencies in a partially hinged rectangular
plate}, ESAIM COCV 24, 63-87, 2018.

\bibitem{bfg1}
E.~Berchio, A.~Ferrero, and F.~Gazzola.
\newblock Structural instability of nonlinear plates modelling suspension
  bridges: mathematical answers to some long-standing questions.
\newblock {\em Nonlinear Analysis: Real World Applications}, 28:91--125, 2016.

\bibitem{berger}
H.~M. Berger.
\newblock A new approach to the analysis of large deflections of plates.
\newblock {\em J. Appl. Mech.}, 22:465--472, 1955.

\bibitem{billah} K.Y. Billah, and R.H. Scanlan.
\newblock Resonance, {T}acoma {N}arrows {B}ridge failure, and undergraduate physics textbooks.
\newblock {\em Amer. J. Physics}, 59:118--124, 1991


\bibitem{braess}
D. Braess, S. Sauter, and C. Schwab.
\newblock On the justification of plate models.
\newblock {\em J. Elasticity}, 103:53--71, 2011.

\bibitem{burg}
D.~Burgreen.
\newblock Free vibrations of a pin-ended column with constant distance between pin ends.
\newblock {\em J. Appl. Mech.}, 18:135--139, 1951.


\bibitem{dma}
De Miranda Associati.
\newblock {\em Structural Engineering}.
\newblock http://www.demiranda.it/

\bibitem{mdm} M. De Miranda, private communication.

\bibitem{eurocode} Eurocode 1, {\em Actions on structures. Parts 1-4: General actions - Wind actions},
The European Union Per Regulation 305/2011, Directive 98/34/EC \& 2004/18/EC. http://www.phd.eng.br/wp-content/uploads/2015/12/en.1991.1.4.2005.pdf

\bibitem{FerGazMor}
V. Ferreira, F. Gazzola and E. Moreira dos Santos.
\newblock Instability of modes in a partially hinged rectangular plate.
\newblock {\em J. Diff. Eq.}, 261, 6302--6340, 2016.

\bibitem{2015ferreroDCDSA}
A.~Ferrero and F.~Gazzola.
\newblock A partially hinged rectangular plate as a model for suspension bridges.
\newblock {\em Discrete Contin. Dyn. Syst.}, 35(12):5879--5908, 2015.

\bibitem{fitouriharaux} C. Fitouri, A. Haraux. {\em Boundedness and stability for the damped and forced single well Duffing equation}, Discrete Contin.
Dyn. Syst. 33, 211-223, 2013.

\bibitem{2012GasmiHaraux}
S.~Gasmi and A.~Haraux.
\newblock N-cyclic functions and multiple subharmonic solutions of Duffing's equation
\newblock {\em J. Math. Pures Appl.}, 97 (2012) 411--423.

\bibitem{2015gazzola}
F.~Gazzola.
\newblock {\em Mathematical models for suspension bridges}. MS\&A. Modeling, Simulation and Applications, 15.
\newblock Springer, Cham, 2015.

\bibitem{ghisigobbinoharaux}
M. Ghisi, M. Gobbino and A. Haraux, {\em An infinite dimensional Duffing-like evolution equation with linear dissipation and an asymptotically small source
term}, Nonlinear Anal. Real World Appl. 43 (2018) 167-–191

\bibitem{Giosan}
I.~Giosan and P.~Eng.
\newblock Structural Vortex Shedding Response Estimation Methodology and Finite Element Simulation \newblock Vortex Shedding Induced Loads on Free Standing Structures.

\bibitem{grunau} H.-Ch. Grunau.
\newblock Nonlinear questions in clamped plate models.
\newblock {\em Milan J. Math.}, 77:171--204, 2009.

\bibitem{gruswe} H.-Ch. Grunau and G. Sweers.
\newblock A clamped plate with a uniform weight may change sign.
\newblock {\em Discrete Contin. Dyn. Syst. Ser. S},  7:761--766, 2014.

\bibitem{gruswe2} H.-Ch. Grunau and G. Sweers.
\newblock In any dimension a ``clamped plate'' with a uniform weight may change sign.
\newblock {\em Nonlinear Anal. TMA}, 97:119--124, 2014.

\bibitem{harauxbook}
A. Haraux.
\newblock Nonlinear vibrations and the wave equation,
\newblock Springer Briefs in Mathematics, BCAM Springer, Cham (2018)

\bibitem{knightly}
G.~H. Knightly and D.~Sather.
\newblock Nonlinear buckled states of rectangular plates.
\newblock {\em Arch. Rational Mech. Anal.}, 54:356--372, 1974.

\bibitem{mansfield}
E.~H. Mansfield.
\newblock {\em The bending and stretching of plates}.
\newblock Cambridge University Press, Cambridge, second edition, 1989.

\bibitem{navier}
C.-L. Navier.
\newblock Extrait des recherches sur la flexion des plans elastiques.
\newblock {\em Bull. Sci. Soc. Philomathique de Paris}, 5:95--102, 1823.

\bibitem{sweers}
S.A. Nazarov, A. Stylianou, and G. Sweers.
\newblock Hinged and supported plates with corners.
\newblock {\em Zeit. Angew. Math. Physik}, 63:929--960, 2012.


\bibitem{Pradeep}
S. Pradeep, K. Shrivastava.
\newblock On the stability of the damped Mathieu equation.
\newblock Mech. Res. Comm. 15 (1988), no. 6, 353--359.


\bibitem{scanlan} R.H. Scanlan.
\newblock The action of flexible bridges under wind, {I}: flutter theory, {II}: buffeting theory.
\newblock {\em J. Sound and Vibration}, 60:187--199 \& 201--211, 1978.

\bibitem{scott}
R.~Scott.
\newblock In the wake of Tacoma. Suspension bridges and the quest for aerodynamic stability.
\newblock {\em ASCE, Reston}, 2001.

\bibitem{tac2}
F.C.~Smith and G.S.~Vincent.
\newblock {\it Aerodynamic stability of suspension bridges: with special reference to the Tacoma Narrows Bridge, Part II: Mathematical analysis}.
\newblock Investigation conducted by the Structural Research Laboratory, University of Washington, University of Washington Press, Seattle, 1950.

\bibitem{souplet}
P. Souplet.
\newblock Uniqueness and nonuniqueness results for the antiperiodic solutions of some second-order nonlinear evolution equations.
\newblock {\em Nonlinear Analysis TMA}, 26(9):1511--1525, 1996.




\bibitem{souplet2}
P. Souplet.
\newblock Optimal uniqueness condition for the antiperiodic solutions of some nonlinear parabolic equations.
Nonlinear Anal. 32 (1998), no. 2, 279--286.

\bibitem{tacoma}
Tacoma Narrows Bridge collapse.
\newblock http://www.youtube.com/watch?v=3mclp9QmCGs (1940)


\bibitem{temam} R. Temam.
\newblock Infinite-dimensional dynamical systems in mechanics and physics,
\newblock Applied Mathematical Sciences 68, Springer, New York (1997)

\bibitem{ventsel}
E.~Ventsel and T.~Krauthammer.
\newblock {\em Thin plates and shells: theory: analysis, and applications}.
\newblock CRC press, 2001.

\bibitem{villaggio}
P.~Villaggio.
\newblock {\em Mathematical models for elastic structures}.
\newblock Cambridge University Press, Cambridge, 1997.

\bibitem{woinowsky}
S.~Woinowsky-Krieger.
\newblock The effect of an axial force on the vibration of hinged bars.
\newblock {\em J. Appl. Mech.}, 17:35--36, 1950.


\end{thebibliography}

\end{document}